\documentclass{amsart}
\usepackage{a4wide}
\usepackage[utf8]{inputenc}
\usepackage{amssymb}
\usepackage{amsmath}
\usepackage{mathrsfs}
\usepackage{amsthm}
\usepackage{mathabx}
\usepackage{mathtools}
\usepackage{leftidx}
\usepackage{color}
\usepackage{enumitem}

\usepackage{caption}
\usepackage{tikz}
\usepackage{xcolor}
\usepackage{upgreek}
\usepackage{hyperref}
\usepackage{bbm}

\newcommand{\N}{\mathbb{N}}
\newcommand{\Z}{\mathbb{Z}}
\newcommand{\R}{\mathbb{R}}
\newcommand{\C}{\mathbb{C}}
\newcommand{\K}{\mathbb{K}}
\newcommand{\T}{\mathbb{T}}
\renewcommand{\P}{\mathbb{P}}
\newcommand{\E}{\mathbb{E}}
\newcommand{\PP}{\mathbb{P}}
\newcommand{\ce}{\coloneqq}
\newcommand{\ec}{\eqqcolon}
\newcommand{\1}{\mathbf{1}}

\newcommand{\calB}{\mathcal{B}}

\newcommand{\calF}{\mathcal{F}}

\newcommand{\calL}{\mathcal{L}}

\newcommand{\calO}{\mathcal{O}}
\newcommand{\calP}{\mathcal{P}}

\newcommand{\filtrF}{\mathscr{F}}

\newcommand{\dx}{~\mathrm{d}x}

\newcommand{\rmd}{\mathrm{d}}

\newcommand{\hra}{\hookrightarrow}

\newcommand{\seq}{\subseteq}

\newcommand{\ve}{\varepsilon}
\newcommand{\vp}{\varphi}

\newcommand{\ee}{\mathrm{e}}

\newcommand{\wt}{\widetilde}

\newcommand{\nn}{|\!|\!|}

\newcommand{\F}{\mathcal{F}}
\newcommand{\G}{\mathcal{G}}

\newcommand{\gHX}{{\mathcal{L}_2(H,X)}}
\newcommand{\CufgX}{C_{u_0,f,g,X}}

\newcommand{\CUffggY}{C_{U_0,\mathbf{f},\mathbf{g},Y}}
\newcommand{\CufgY}{C_{u_0,f,g,Y}}

\newcommand{\KUffggY}{K_{U_0,\mathbf{f},\mathbf{g},Y}}
\newcommand{\KufgY}{K_{u_0,f,g,Y}}

\newcommand{\bfE}{\mathbf{E}}
\newcommand{\bfF}{\mathbf{F}}
\newcommand{\bfG}{\mathbf{G}}
\newcommand{\bfH}{\mathbf{H}}

\newcommand{\CFX}{C_{F,X}}

\newcommand{\CGX}{C_{G,X}}

\newcommand{\Cstab}{C_{\text{stab}}}
\newcommand{\Cbdd}{C_{\text{bdd}}}
\newcommand{\CY}{C_Y}

\newcommand{\ds}{\,\mathrm{d}s}
\newcommand{\dWHs}{\,\mathrm{d}W_H(s)}

\newcommand{\LFX}{L_{F,X}}
\newcommand{\LFY}{L_{F,Y}}
\newcommand{\LGX}{L_{G,X}}
\newcommand{\LGY}{L_{G,Y}}

\newcommand{\iu}{\mathrm{i}}

\newcommand{\gHY}{{\mathcal{L}_2(H,Y)}}
\newcommand{\rmS}{\mathrm{S}}
\newcommand{\vareps}{\varepsilon}

\DeclareMathOperator{\dist}{dist}

\DeclareMathOperator{\diag}{diag}

\DeclareMathOperator{\curl}{curl}

\newtheorem{Satz}{Satz}[section]
\newtheorem{definition}[Satz]{Definition}
\newtheorem{theorem}[Satz]{Theorem}
\newtheorem{lemma}[Satz]{Lemma}	
\newtheorem{proposition}[Satz]{Proposition}
\newtheorem{corollary}[Satz]{Corollary}
\newtheorem{remark}[Satz]{Remark}
\newtheorem{example}[Satz]{Example}
\newtheorem{assumption}[Satz]{Assumption}

\allowdisplaybreaks
\numberwithin{equation}{section}

\begin{document}
\title[Pathwise Uniform Convergence of Discretisation Schemes]{Pathwise Uniform Convergence of Time Discretisation Schemes for SPDEs}
\author{Katharina Klioba}\address{
Hamburg University of Technology,
Institute of Mathematics,  D-21073 Hamburg, Germany}
\email{Katharina.Klioba@tuhh.de}

\author{Mark Veraar}
\address{Delft Institute of Applied Mathematics\\
Delft University of Technology \\ P.O. Box 5031\\ 2600 GA Delft\\The
Netherlands} \email{M.C.Veraar@tudelft.nl}

\date{\today}

\thanks{The second author is supported by the VICI subsidy VI.C.212.027 of the Netherlands Organisation for Scientific Research (NWO)}



\begin{abstract}
In this paper, we prove convergence rates for time discretisation schemes for semi-linear stochastic evolution equations with additive or multiplicative Gaussian noise, where the leading operator $A$ is the generator of a strongly continuous semigroup $S$ on a Hilbert space $X$, and the focus is on non-parabolic problems. The main results are optimal bounds for the {\em uniform strong error}
\[{\rm E}_{k}^{\infty} \coloneqq \Big(\E \sup_{j\in \{0, \ldots, N_k\}} \|U(t_j) - U^j\|^p\Big)^{1/p},\]
where $p \in [2,\infty)$, $U$ is the mild solution, $U^j$ is obtained from a time discretisation scheme, $k$ is the step size, and $N_k = T/k$. The usual schemes such as the exponential Euler, the implicit Euler, and the Crank--Nicolson method, etc.\ are included as special cases. Under conditions on the nonlinearity and the noise, we show
\begin{itemize}
\item ${\rm E}_{k}^{\infty}\lesssim k \sqrt{\log(T/k)}$ (linear equation, additive noise, general $S$);
\item ${\rm E}_{k}^{\infty}\lesssim \sqrt{k} \sqrt{\log(T/k)}$ (nonlinear equation, multiplicative noise, contractive $S$);
\item ${\rm E}_{k}^{\infty}\lesssim k \sqrt{\log(T/k)}$ (nonlinear wave equation, multiplicative noise)
\end{itemize}
for a large class of time discretisation schemes.
The logarithmic factor can be removed if the exponential Euler method is used with a (quasi)-contractive $S$. The obtained bounds coincide with the optimal bounds for SDEs. Most of the existing literature is concerned with bounds for the simpler {\em pointwise strong error}
\[{\rm E}_k \coloneqq \bigg(\sup_{j\in \{0,\ldots,N_k\}}\E \|U(t_j) - U^{j}\|^p\bigg)^{1/p}.\]
Applications to Maxwell equations, Schr\"odinger equations, and wave equations are included. For these equations, our results improve and reprove several existing results with a unified method and provide the first results known for the implicit Euler and the Crank--Nicolson method.
\end{abstract}

\maketitle

\section{Introduction}
In this paper, we consider stochastic PDEs driven by an additive or multiplicative Gaussian noise. The equations we consider can be written as abstract stochastic evolution equations on a Hilbert space $X$ of the form
\begin{align}
\label{eq:stEvolEqnintro}
   \Bigg\{\begin{split} \rmd U &=(A U + F(U))\,\rmd t  + G(U) \,\rmd W_H~~~\text{ on } [0,T],\\ U(0) &= u_0 \in L^p(\Omega;X).
   \end{split}
\end{align}
Here, $A$ is the generator of a $C_0$-semigroup $(S(t))_{t\geq 0}$, $W_H$ is a cylindrical Brownian motion, $F$ and $G$ are globally Lipschitz, $u_0$ is the initial data, and $p \in [2,\infty)$.

Our aim is to obtain strong convergence rates for temporal discretisation schemes that cover the hyperbolic setting. The hyperbolic setting has been extensively studied in recent years (see \cite{AC18, ACLW16, BLM21, berg2020exponential, BC23, CCHS20, CL22, CLS13, CQS16, cox2019weak, Cui21, CuiHong, HM19, HHS22, JdNJW21, KLP20, KLL12, KLL13, KLS10, Wang15, WGT14} and references therein). In the parabolic setting, (i.e., $(S(t))_{t\geq 0}$ being an analytic semigroup) regularisation phenomena occur, which make it possible to prove very different convergence results. In the non-parabolic case, new methods to show convergence rates are needed and related to a way to obtain regularity. Kato's setting for the hyperbolic case from his seminal work \cite{Kato75} creates a way to obtain this regularity, which has proven to be very useful in the analysis of quasilinear equations as well as their numerical treatment \cite{DorHoch, HochbruckPazur, HochPaSch, LubichKovacsKatoSetting,   schnaubelt2023error}.

The main idea in Kato's setting is to consider two spaces $X$ and $Y$ with $Y\hookrightarrow X$ (or sometimes even three spaces) on which the operator $A$ and the nonlinearities $F$ and $G$ can be analysed. In this way, one can create  regularity of $U$, and obtain better mapping properties of the nonlinearities. In numerical approximations, the obtained regularity can be used to obtain convergence rates, as illustrated for the deterministic case in the references above.

The above setting often also applies to the parabolic case, in which, however, the required mapping properties of $F$ on $Y$ can often be avoided due to the regularising effect of the convolution with the analytic semigroup $S$. For these equations, it does not seem necessary to work with the Kato setting, as regularisation phenomena can be exploited. For details on the parabolic case, the reader is referred to \cite{ACQ20, BL13, BeJe19, BHRR, CoxNee13,  DHW, GyMi09, JK09, JK11, Milstein, KamBlo, KLL15, KLL18, Kruse, LPS14, MilsteinNonCommutative} and references therein, as well as Remark \ref{rem:Milstein}. Consequently, our focus lies on the hyperbolic setting.

\subsection{Setting}

In the above-mentioned literature on the hyperbolic case (and often in the parabolic case), the error considered is the {\em pointwise strong error}
\begin{align}\label{eq:pointwiseerrorestintro}
\sup_{j\in \{0,\ldots, N_k\}}\E\|U(t_j) - U^j\|^p,
\end{align}
where $U$ is the mild solution to \eqref{eq:stEvolEqnintro}, and $(U^j)_{j=0}^{N_k}$ is an approximation of the solution given by a temporal discretisation scheme of the form $U^0 = u_0$,
\begin{equation}
\label{eq:Ujschementro}
    U^{j} = R_k U^{j-1} + k R_k F(U^{j-1})+ R_k G(U^{j-1}) \Delta W_{j}, \ \ j=1, \ldots, N_k.
\end{equation}
Here, $N_k = T/k$ is the number of points, $k = t_j-t_{j-1}$ is the uniform step size, $t_j=jk$, and $\Delta W_{j} = W_H(t_j) - W_H(t_{j-1})$. The operator $R_k$ is an approximation of the semigroup $S$ at time $k$.

When performing numerical simulations to approximate the solution of a stochastic equation, one naturally wants the simulation to be close to the solution of \eqref{eq:stEvolEqnintro}. However, \eqref{eq:pointwiseerrorestintro} being small does not provide enough information to conclude this, see Example \ref{ex:exampleErrorDifference}. Also, from a probabilistic point of view, \eqref{eq:pointwiseerrorestintro} contains no information on the convergence of the path. Instead, it is a more meaningful question to find convergence rates for the {\em uniform strong error}
\begin{align}\label{eq:uniformerrorestintro}
\E \sup_{j\in \{0, \ldots, N_k\}} \|U(t_j) - U^j\|^p,
\end{align}
where now the supremum over $j$ is inside the expectation. In the deterministic setting, there is no difference between \eqref{eq:pointwiseerrorestintro} and \eqref{eq:uniformerrorestintro}.
It is a widely known open problem in the field to find optimal estimates for \eqref{eq:uniformerrorestintro}. Such estimates where the supremum is inside the expectation are usually called maximal estimates, and there is an enormous literature on maximal estimates for general stochastic processes \cite{Talagrand}. However, for processes that do not have any Gaussian or martingale structure, it can be quite complicated to prove (sharp) maximal estimates. Even maximal estimates for the mild solution $U$ to \eqref{eq:stEvolEqnintro} with $F=0$ and $G(u)$ replaced by a progressively measurable $g\in L^2(\Omega\times(0,T);X)$, are unknown in general (see the survey \cite[Section 4]{vNV20} for details).
The difference between the errors \eqref{eq:pointwiseerrorestintro} and \eqref{eq:uniformerrorestintro} is illustrated in the following simple example.
\begin{example}\label{ex:exampleErrorDifference}
Let $\Omega = [0,1]$ and let $\P$ denote the Lebesgue measure. For $\gamma\in (0,1]$, let $v_N:\Omega \times[0,1]\to \R$ be given by $v_N(\omega,t) = 1$ if $|t-\omega|<1/(2N^{\gamma})$, and zero otherwise. Then one can check that the following error estimates hold:
\[\sup_{t\in [0,1]}\E|v_N(t)|^p\leq \frac{1}{N^{\gamma}} \ \ \ \text{and} \ \ \ \E \sup_{t\in [0,1]}|v_N(t)|^p = 1. \]
One even has $\sup_{t\in [0,1]}|v_N(\omega,t)| = 1$ for any $\omega\in \Omega$.
This shows the discrepancy between having the supremum inside the expectation or not. Continuity of $v_N$ plays no role here. Indeed, one can easily replace the indicator function by a continuous piecewise constant function without influencing the above error estimates.
\end{example}

In the case where $S$ generates a $C_0$-{\em group}, it is known how to estimate the uniform strong error \eqref{eq:uniformerrorestintro} for the {\em exponential Euler method} (i.e., $R_k = S(k)$). In this case, one can use the group structure in the following way
\[\int_{0}^{t} S(t - s)g(s) d W_H(s) = S(t) \int_0^{t} S(-s) g(s) d W_H(s),\]
and, similarly, for the discrete approximation. This makes it possible to avoid maximal estimates for stochastic convolutions and use martingale techniques instead. This technique was first applied in \cite{Wang15} to obtain optimal convergence rates for the uniform strong error of the exponential Euler method for abstract wave equations. Later, this technique was extended to other settings (see \cite{AC18, berg2020exponential, CCHS20, CHLZ19}), and, in particular, applied to stochastic Schr\"odinger and Maxwell equations. However, if $S$ is not a group, this technique is no longer applicable.  Equations in which $S$ is not a group include transport equations, equations with dissipation (e.g.\ damped wave equations), parabolic equations, etc. Of course, there are also many important systems where groups are unavailable (e.g.\ if a parabolic equation is coupled to a wave or transport equation).
Even more importantly, for schemes involving rational approximations (e.g.\ implicit Euler, Crank--Nicolson), it is unclear how to use the $C_0$-group structure to estimate the uniform strong error, since the group does not appear in the scheme.

On the other hand, for other discretisation schemes estimates for the simpler pointwise strong error \eqref{eq:pointwiseerrorestintro} are available (see e.g.\ the above-mentioned papers in the hyperbolic case). Moreover, simulations suggest that optimal rates of convergence for the uniform strong error \eqref{eq:uniformerrorestintro} hold as well. The main goal of our work is to prove such optimal bounds for \eqref{eq:uniformerrorestintro} for more general semigroups and more general schemes. In particular, we prove such bounds under the condition that $S$ and $R$ are contractive. This solves the open problem on optimal rates for \eqref{eq:uniformerrorestintro} for this class of semigroups and numerical schemes up to a logarithmic factor.

\subsection{Some of the main results for multiplicative noise}
As in Kato's setting for the hyperbolic case, let $X$ and $Y$ be Hilbert spaces with $Y\hookrightarrow X$. For $\alpha\in (0,1]$ we say that $R$  approximates $S$ to order $\alpha$ on $Y$ if there is a constant $C_{\alpha} \ge 0$ such that
for all $x \in Y$, $k>0$, and $j\in \{0,\ldots,N_k\}$
\begin{equation*}
    \|(S(t_j)-R_k^j)x\|_X \le C_\alpha k^\alpha\|x\|_Y,
\end{equation*}
where $R_k^j=(R_k)^j$ denotes the $j$-th power of the scheme at time step $k$. Our main result on convergence rates for \eqref{eq:uniformerrorestintro} is as follows.
\begin{theorem}\label{thm:intromain}
    Let $X$ and $Y$ be Hilbert spaces such that $Y\hookrightarrow X$. Let $A$ be the generator of a $C_0$-contraction semigroup $(S(t))_{t\geq0}$ on $X$  and $Y$. Suppose that $(R_k)_{k>0}$ is a time discretisation scheme which is contractive on both $X$ and $Y$, that $R$ approximates $S$ to order $\alpha\in (0, 1/2]$ on $Y$, and that $Y \hra D((-A)^{\alpha})$.
    Suppose that $F:X\to X$ and $G:X\to \calL_2(H,X)$ are Lipschitz continuous, and that $F:Y\to Y$ and $G:Y\to \calL_2(H,Y)$ are of linear growth.
    Let $p \in [2,\infty)$, $u_0 \in L^p(\Omega;Y)$, and $U$ be the mild solution to \eqref{eq:stEvolEqnintro}. Let $k\in (0,T/2]$ and let $(U^j)_{j=0}^{N_k}$ be given by \eqref{eq:Ujschementro}.
    Then there is a constant $C_T>0$ not depending on $u_0$ and $k$ such that
    \begin{equation}
    \label{eq:convRateintro}
        \bigg\lVert\max_{0 \le j \le N_k}\|U(t_j)-U^j\|_X\bigg\rVert_{L^p(\Omega)}
        \le C_T(1+\|u_0\|_{L^p(\Omega;Y)}) k^{\alpha} \sqrt{\log(T/k)}.
           \end{equation}
    In particular, the approximations $(U^j)_j$ converge at rate $\alpha$ as $k \to 0$ up to a logarithmic factor.
\end{theorem}

Theorem \ref{thm:intromain} applies to, among others,
\begin{itemize}
    \item exponential Euler (EE): $R_k = S(k)$;
    \item implicit Euler (IE): $R_k = (1-kA)^{-1}$;
    \item Crank--Nicolson (CN): $R_k = (2+kA)(2-kA)^{-1}$.
\end{itemize}
Higher-order implicit Runge-Kutta methods such as Radau methods, BDF(2), Lobatto IIA, IIB, and IIC, and some DIRK schemes are covered as well. The contractivity of the scheme $R$ in the case of (EE) and (IE) follows from the contractivity of the semigroup $S$. For other rational schemes, the contractivity of $R_k=r(kA)$ follows from the holomorphy of the corresponding rational function $r:\C_{-}\to\C$ and $|r(z)|\le 1$ for all $z \in \C_{-}$, which, in particular, is satisfied for $A$-acceptable or $A$-stable schemes. These assertions follow from functional calculus (see Proposition \ref{prop:functionalcalculus}).

In the above, one usually takes $Y$ to be a suitable intermediate space between $X$ and $D(A)$. In the special and important case that $Y = D(A)$, one can take $\alpha = \frac12$ for all of the aforementioned schemes. More general convergence rates can be found in Table \ref{tab:convrate}.

\begin{table}[ht]
\begin{tabular}{|c|c|c|c|}
  \hline
  & Exponential Euler & Implicit Euler & Crank--Nicolson \\ \hline
  $\alpha$ & $\beta\wedge \frac12$ & $\frac{\beta}{2}\wedge \frac12$ & $\frac{2\beta}{3}\wedge \frac12$ \\
  \hline
\end{tabular}
\caption{Convergence rates $\alpha$ in case $Y = D((-A)^{\beta})$ in Theorem \ref{thm:intromain}}\label{tab:convrate}
\end{table}

Up to the logarithmic factor, the estimate \eqref{eq:convRateintro} is optimal in the sense that the rate is the same as the rate for the initial value term on its own (i.e.\ with $F= 0$ and $G=0$). Theorem \ref{thm:intromain} follows from Theorem \ref{thm:convergenceRate}. In the case of the exponential Euler method, we show that the logarithmic factor can be omitted, see Corollary \ref{cor:expEulerMultiplicative}. In the case of additive noise, a similar result is obtained in Theorem \ref{thm:convRateAdditiveGeneral} for the range $\alpha\in (0,1]$ for semigroups and schemes which are not necessarily contractive.

The error estimate \eqref{eq:convRateintro} can be extended from the grid points to the full time interval $[0,T]$ assuming higher integrability of the initial values. Provided that $u_0 \in L^{p_0}(\Omega;Y)$ holds for some $p_0 \in (2,\infty)$ in addition to the assumptions of Theorem \ref{thm:intromain}, the pathwise uniform error on the full time interval can be estimated as (see Theorem \ref{thm:generalschemeuniform0T} below)
\begin{equation}\label{eq:introfulltime}
    \bigg\lVert\sup_{t \in [0,T]}\|U(t)-\tilde{U}(t)\|_X\bigg\rVert_{L^p(\Omega)}
        \le C_T(1+\|u_0\|_{L^{p_0}(\Omega;Y)}) k^{\alpha} \sqrt{\log(T/k)}
\end{equation}
for all $p \in [2,p_0)$ and the piecewise constant extension $\tilde{U}$ of $(U_j)_{j=0,\ldots,N_k}$ to $[0,T]$. This rate of convergence is known to be optimal already for scalar SDEs. In practice, this implies that the rate of convergence in the grid points is maintained already for a piecewise constant interpolation to other times. The error estimate relies on new optimal path regularity estimates of stochastic convolutions in suitable log-Hölder spaces, which will be presented in Proposition \ref{prop:pathRegMildSol}.

Applications to Schr\"odinger  and Maxwell equations are included in the main text (see Subsections \ref{subsec:SchroedingerAdd}, \ref{subsec:SchroedingerMult}, and \ref{subsec:Maxwell}).
Our results improve several results from the literature to more general schemes and general rates $\alpha$. In Section \ref{sec:rateOfConvergenceWave}, we include a setting for abstract wave equations, which was considered in \cite{Wang15} only for the exponential Euler method. We prove similar higher-order convergence rates for more general schemes and, in particular, recover \cite{Wang15} as a special case.

Let us emphasise that schemes involving rational approximations, such as the implicit Euler or the Crank--Nicolson method, are in the focus of our work. While we improve existing results for the exponential Euler method, the main novelty of our work lies in the possibility to treat other schemes with a semigroup approach. To the best of the authors' knowledge, the present work is the first contribution to pathwise uniform convergence rates for hyperbolic problems from a theoretical standpoint, both in the generality and for the concrete examples listed above. The main innovations are:
\begin{itemize}
	\item first optimal pathwise uniform convergence rates for the implicit Euler method, the Crank--Nicolson method, and any other contractive time discretisation scheme for hyperbolic SPDEs
	\item first use of Kato's framework for SPDEs to systematically treat hyperbolic problems
	\item maximal estimates for the convergence rate rather than pointwise estimates
    \item path regularity results allowing to consider the error on the full time interval
	\item novel pathwise uniform stability estimates
	\item convergence up to order $1$ for abstract wave equations for any contractive scheme
\end{itemize}

To make the above results applicable to implementable numerical schemes for SPDEs, one would additionally need a space discretisation. Since the main novelty of our work lies in the treatment of temporal discretisations, we will only consider the latter. Space discretisation is usually performed by means of spectral Galerkin methods \cite{JdNJW21, Milstein, KamBlo, WGT14}, finite differences \cite{ACQ20, CQS16,  GyMi09} or finite elements \cite{ACLW16, CLS13, KLP20, KLL12, KLL13, KLS10, Kruse}, sometimes combined with a discontinuous Galerkin approach \cite{BLM21, HHS22}, or other methods in space or space-time \cite{BL13, CuiHong, CHLZ19, DHW, HM19, le2023class}. 

A detailed understanding of the global Lipschitz setting is a quintessential step towards the treatment of local Lipschitz nonlinearities, which occur more frequently in practice. Our result should be seen as a first step, and we plan to continue our work on uniform strong errors in a local Lipschitz setting in the near future.

It was recently shown in \cite{CHJNW} that one can transfer \eqref{eq:pointwiseerrorestintro} to \eqref{eq:uniformerrorestintro} using some of the H\"older continuity in the $p$-th moment at the price of decreasing the convergence rate via the Kolmogorov-Chentsov theorem.
The strength of this lies in the generality of possible applications. However, to get practically useful bounds in concrete cases, there are limitations. A more detailed comparison is made in Remark \ref{rem:Kolmogorovapproach}. 

\subsection{Method of proof}

For the proof of the convergence rate, we need several ingredients. First of all, we need to prove that the mild solution actually is continuous with values in the subspace $Y$. This can be seen as the replacement of the usual regularisation one has for parabolic equations in spirit of the Kato setting explained before. Surprisingly, we do not need any Lipschitz assumptions on $F$ and $G$ as mappings from $Y$ to $Y$, but linear growth conditions suffice. This is crucial since Lipschitz estimates typically fail for Nemytskij mappings on Sobolev spaces of higher order (see \cite{DahlbergAffineLinear} and Remark \ref{rem:Lipschitz-growthY}).

A key estimate in the proof is a new maximal inequality for discrete convolutions. In particular, this inequality will be used to prove the stability of schemes such as \eqref{eq:Ujschementro}, i.e.,
\[\E \sup_{j\in \{0, \ldots, N_k\}} \|U^j\|^p_{Y} \leq C,\]
where $C$ is independent of the step size $k$. But it also plays a role in further estimates for the convergence.

A second key ingredient is another estimate recently proven in \cite{JanMarkMainPaper}, which allows estimating stochastic integral processes that contain a supremum
\begin{equation}\label{eq:logtermsintro}
\E \sup_{i\in \{1, \ldots, n\}} \sup_{t\geq 0}\Big\|\int_0^t  \Phi_i(s) d W_H(s)\Big\|_X^p
\end{equation}
by certain square functions with a logarithmic dependency on $n$ (see Proposition \ref{prop:PropLogMainPaper} below).

Finally, to prove the desired convergence rate of Theorem \ref{thm:intromain} we need to split the error obtained in \eqref{eq:Ujschementro} into
\[1 \text{ (initial value part)} + 4 \text{ (deterministic terms)} + 5 \text{ (stochastic terms)} = 10 \text{ terms}.\]
To estimate these terms we require precise estimates for $\|S(t_j) - R_k^j\|_{\calL(Y,X)}$, $\E\|U(t)-  U(s)\|^p$, stability estimates, and maximal estimates for continuous and discrete convolutions.

In the end, we derive an estimate for the error in terms of itself, and we apply a standard discrete Gronwall argument to deduce the desired error bound. In the case of the exponential Euler method, some terms disappear since $S(t_j) = R_k^j$, which makes it possible to omit the logarithmic terms originating from terms such as \eqref{eq:logtermsintro}.

\subsection{Overview}

\begin{itemize}
\item Section \ref{sec:prelim} contains the preliminaries for the rest of the paper.
\item Section \ref{sec:rateOfConvergenceAdditive} discusses the case of {\em additive noise} and semigroups that are not necessarily contractive. We prove convergence of rate $\alpha$ up to order one, in case the noise and data are regular enough. This is proved under the assumption that the numerical scheme $R_k$ approximates the semigroup at rate $\alpha$. Results are illustrated for the Schrödinger equation in which case the obtained results improve several bounds from the literature for the exponential Euler method, and provide the first uniform bounds for a large class of other numerical methods including the implicit Euler and the Crank--Nicolson method.
\item In Section \ref{sec:wellposed} we introduce the nonlinear evolution equation with {\em multiplicative noise} that we consider in the rest of the paper. After recalling a standard well-posedness result, we introduce a special case of the Kato setting and prove that the solution has regularity in the subspace $Y$ in case of linear growth in the $Y$-setting (see Theorem \ref{thm:wellposedY}).
\item Section \ref{sec:stability} is concerned with the stability of the discretisation schemes for the nonlinear evolution equation introduced in Section \ref{sec:wellposed}. The main stability result can be found in Proposition \ref{prop:stab} and only requires linear growth. Hence, it is applicable on both $X$ and $Y$.
\item Section \ref{sec:rateOfConvergenceMultNoise} is central in the paper, and here we prove Theorem \ref{thm:intromain} for the nonlinear evolution equation introduced in Section \ref{sec:wellposed} (see Theorem \ref{thm:convergenceRate} for the extended version). Moreover, we prove the error bound \eqref{eq:introfulltime} on the full time interval in Theorem \ref{thm:generalschemeuniform0T}. For this, we first establish a new optimal path regularity result for the solution in Proposition \ref{prop:pathRegMildSol}, which is of independent interest. In Subsections \ref{subsec:SchroedingerMult} and \ref{subsec:Maxwell}, we present applications to the Schrödinger equation as well as the Maxwell equation. A numerical simulation of the Schrödinger equation in Subsection \ref{subsec:numerical} confirms the analytical convergence rates obtained.
\item In Section \ref{sec:rateOfConvergenceWave}, we consider abstract stochastic wave equations, and obtain convergence rates up to order one (see Theorem \ref{thm:convergenceRateWave}). Although we are not in the setting of Section \ref{sec:rateOfConvergenceMultNoise}, an inspection of the proofs given there shows that certain terms behave better for abstract wave equations due to their second-order nature. Again, convergence rates are obtained for a large class of numerical schemes, and versions of \eqref{eq:introfulltime} are obtained. Examples with trace class, space-time white noise, and smooth noise are included and can be found in Subsections \ref{subsec:traceclassnoisewave}, \ref{subsec:exampleWavewhitenoise}, and \ref{subsec:smoothnoisewave}, respectively. All these results are new for schemes different from the exponential Euler method. Most notably, for smooth noise, we can explain the numerical convergence rates one sees in \cite[Figure 6.1]{Wang15} for the implicit Euler and the Crank--Nicolson method.
\end{itemize}

\subsubsection*{Acknowledgements}

The first author wishes to thank the DAAD for the financial support to visit TU Delft for one semester in 2022, and the colleagues in Delft for their hospitality. Both authors thank Jan van Neerven and Christian Seifert for helpful discussion and comments, and Martin Hutzenthaler for suggesting adding error estimates on the full time interval. The authors also thank Sonja Cox for indicating the optimal $\sqrt{\log(N)}$-dependency in Proposition \ref{prop:PropLogMainPaper} and Emiel Lorist for pointing out the simple short-cut for proving it. Further, the authors thank the anonymous referees for their feedback, which has helped to improve the quality and readability of the paper significantly.

\section{Preliminaries}
\label{sec:prelim}

\subsection*{Notation}
Throughout the paper, we fix a probability space $(\Omega, \filtrF, \PP)$ with filtration $(\filtrF_t)_{t \in [0,T]}$. Denote the progressive $\sigma$-algebra on $(\Omega, \filtrF, \PP)$ by $\calP$ and the progressively measurable subspace of a given space by the index $\calP$. Moreover, $H$, $X$, and $Y$ denote Hilbert spaces, where $H$ is used to define the $(\filtrF_t)_{t \in [0,T]}$-cylindrical Brownian motion $W_H$. Subsequently, the space of Hilbert--Schmidt operators from $H$ to $X$ is denoted by $\calL_2(H,X)$ and the Borel $\sigma$-algebra of $X$ by $\calB(X)$. Subsequently, we consider the final time $T>0$ to be fixed and consider a uniform time grid with $t_j = jk$, where $k>0$ is the time step and $j=0, \ldots, N_k$ with $N_k = T/k \in \N$, and define $\lfloor t \rfloor \ce \max\{t_j:\,t_j\le t\}$ for $t \in [0,T]$. By $(S(t))_{t\geq 0}$, we denote a $C_0$-semigroup and by $(R_k)_{k>0}$ a numerical scheme that approximates $S$. For a given evolution equation, $(U(t))_{t\in [0,T]}$ is the exact solution and $U^j$ the numerical solution approximating $U$ at time $t_j$ for $j=0,\ldots,N_k$. For $f$ and $g$ in the respective spaces, let $\|f\|_{p,q,Z} \ce \|f\|_{L^p(\Omega;L^q(0,T;Z))}$ and $\nn g\nn_{p,q,Z} \ce \|g\|_{L^p(\Omega;L^q(0,T;\calL_2(H,Z)))}$. We use the notation $f(x) \lesssim g(x)$ to denote that there is a constant $C \ge 0$ such that for all $x$ in the respective set, $f(x) \le C g(x)$.

\subsection{Stochastic integration}\label{sec:SI}

The space $\calL_2(H,X)$ of Hilbert--Schmidt operators from $H$ to $X$ consists of all bounded operators $R:H \to X$ such that
\begin{equation*}
    \|R\|_{\gHX}^2 \ce \sum_{i \in I}\|Rh_i\|_X^2 <\infty,
\end{equation*}
where $(h_i)_{i \in I}$ is an orthonormal basis of $H$. If $R\in \calL_2(H,X)$, the sum contains at most countably many non-vanishing terms.
For $R\in \gHX$, $(h_i)_{i \in I}$ as before, and $\gamma = (\gamma_n)_{n\geq 1}$ centered i.i.d.\ normally distributed random variables we define
\begin{equation}\label{eq:convradonW}
R \gamma = \sum_{n\geq 1} \gamma_n R h_n,
\end{equation}
where the convergence is in $L^p(\Omega;X)$ for $p < \infty$ and almost surely (see \cite[Corollary 6.4.12]{AnalysisBanachSpacesII}).

In the stochastic integrals appearing in expressions such as \eqref{eq:logtermsintro}, the integrator is an $H$-cylindrical Brownian motion to take $\calL_2(H,X)$-valued integrands into account. An \emph{$H$-cylindrical Brownian motion} is a mapping $W_H:L^2(0,T;H) \to L^2(\Omega)$ such that
\begin{enumerate}[label=(\roman*)]
    \item $W_H b$ is Gaussian for all $b \in L^2(0,T;H)$,
    \item $\E(W_H b_1 \cdot W_H b_2) = \langle b_1,b_2 \rangle_{L^2(0,T;H)}$ for all $b_1, b_2 \in L^2(0,T;H)$,
    \item $W_H b$ is $\F_t$-measurable for all $b\in L^2(0,T;H)$ with support in $[0,t]$,
    \item $W_H b$ is independent of $\F_s$ for all $b\in L^2(0,T;H)$ with support in $[s,T]$,
\end{enumerate}
where we include a complex conjugate on $W_H b_2$ in case we want to use a complex $H$-cylindrical Brownian motion.
For $h \in H$ and $t \in [0,T]$, we use the shorthand notation $W_H(t)h \ce W_H(\1_{(0,t)} \otimes h)$. Consequently, $(W_H(t)h)_{t \in[0,T]}$ is a Brownian motion for each fixed $h \in H$, which is standard if and only if $\|h\|_H=1$. In the special case $H=\R$, this notion coincides with real-valued Brownian motions. We refer to an $H$-valued stochastic process $(W(t))_{t \geq 0}$ as a \emph{$Q$-Wiener process} if $W(0)=0$, $W$ has continuous trajectories and independent increments, and $W(t)-W(s)$ is normally distributed with parameters $0$ and $(t-s)Q$ for $t\geq s \geq 0$. The operator $Q$ is in $\calL(H)$, positive self-adjoint, and of trace class. One can show that $W$ is a $Q$-Wiener process if and only if there exists an $H$-cylindrical Brownian motion $W_H$ such that $Q^{1/2}W_H\ce\sum_{n\geq 1} Q^{1/2} h_n W_H(t) h_n = W(t)$ for an orthonormal basis $(h_n)_{n \ge 1}$ of $H$ (cf. \eqref{eq:convradonW}). To consider an equation such as \eqref{eq:stEvolEqnintro} with a $Q$-Wiener process $W$ instead of a cylindrical Brownian motion, one can replace $G$ by $G Q^{1/2}$ and reduce to the cylindrical case.
For further properties of $H$-cylindrical Brownian motions, $Q$-Wiener processes and the Itô integral, we refer to \cite{DaPratoZabczyk14}.

To estimate Itô integrals w.r.t. such $H$-cylindrical Brownian motions, the Burkholder--Davis--Gundy inequalities are particularly helpful. They imply that
\begin{equation}\label{eq:BDG}
\bigg(\E \sup_{t \in [0,T]} \left\| \int_0^t g(s) \dWHs \right\|_X^p\bigg)^{1/p} \leq B_p \|g\|_{L^p(\Omega;L^2(0,T;\gHX))}.
\end{equation}
In particular, one can take $B_2 = 2$ (by Doob's maximal inequality \cite[Thm.~3.2.2]{AnalysisBanachSpacesI} and the Itô isometry) and $B_p = 4\sqrt{p}$ for $p> 2$. Indeed, this follows by combining the scalar result of  \cite[Theorem A]{CarlenKree} and \cite[Theorem 2]{Ren} with the reduction technique in \cite[Theorem 3.1]{KallenbergSz} and the simple estimate $\|(\xi^2+\eta^2)^{1/2}\|_p\leq (\|\xi\|_p^2 + \|\eta\|_p^2)^{1/2}$ valid for real-valued random variables $\xi$ and $\eta$ and $p\in [2, \infty)$.

\begin{definition}
    A $C_0$-semigroup $(S(t))_{t \ge 0}$ is said to be \textit{quasi-contractive} with parameter $\lambda\ge 0$ if $\|S(t)\| \le e^{\lambda t}$ for all $t \ge 0$.
\end{definition}

The following maximal inequality for stochastic convolutions follows from \cite{HausSei}, where the contractive case is treated. The quasi-contractive case follows from a scaling argument.
\begin{theorem}
\label{thm:maxIneqQuasiContractive}
Let $X$ be a Hilbert space and let $(S(t))_{t \ge 0}$ be a quasi-contractive semigroup on $X$ with parameter $\lambda \ge 0$. Then for $p \in [2,\infty)$
    \begin{equation*}
        \E \sup_{t \in [0,T]} \left\| \int_0^t S(t-s)g(s) \dWHs \right\|_X^p \le \ee^{p\lambda T} B_p^p \|g\|_{L^p(\Omega;L^2(0,T;\gHX))}^p,
    \end{equation*}
    where $B_{p}$ is the constant from \eqref{eq:BDG}. In particular, one can take $B_2 = 2$ and $B_p = 4\sqrt{p}$ for $2 < p <\infty$.
\end{theorem}

Next, we state a special maximal inequality, which will be needed to estimate stochastic integral terms without semigroups. A similar result with constant of order $\log(N)$ can be found in \cite[Proposition 2.7]{JanMarkMainPaper}. 
\begin{proposition}
\label{prop:PropLogMainPaper}
    Let $X$ be a Hilbert space and let $0 < p < \infty$. Let $\Phi \ce (\Phi^{(j)})_{j=1}^N$ be a finite sequence in $L_\calP^p(\Omega; L^2(0,T;\gHX))$ and set
    \begin{equation*}
        I_N^\Phi(p) \ce \bigg( \E \sup_{t \in [0,T], j \in \{1, \ldots, N\}} \bigg\| \int_0^t \Phi_s^{(j)} \dWHs\bigg\|_X^p\bigg)^{1/p}.
    \end{equation*}
    Then for some $K_p \geq 0$,
    \begin{equation*}
        I_N^\Phi(p) \le K_p \max\big\{\sqrt{\log(N)}, \sqrt{p}\big\} \|\Phi\|_{L^p(\Omega;\ell_N^\infty(L^2(0,T;\gHX)))}\quad \text{ if } N \ge 2.
    \end{equation*}
    If $2 \le p <\infty$, this estimate holds with $K_p=K \ce 4\exp(1+ \frac{1}{2\mathrm{e}})\approx 13.07$, which is $p$-independent.
\end{proposition}
The above result was pointed out to the authors by Sonja Cox. The short proof below was pointed out to us by Emiel Lorist. 
\begin{proof}
To prove the result, by approximation, we may assume that each $\Phi^{(j)}$ is contained in $L^\infty(\Omega;L^2(0,T;\gHX))$. First, consider $p_N=\log(N)$ with $N\geq 8$. Then using $\ell^{p_N} \hookrightarrow \ell^\infty$ contractively, and the Burkholder--Davis--Gundy inequalities with $B_p \leq 4\sqrt{p}$ in $X$ (see \eqref{eq:BDG}), we find
\begin{align*}
I_N^\Phi(p_N) &\leq \bigg( \sum_{j=1}^N \E \sup_{t \in [0,T]} \bigg\| \int_0^t \Phi_s^{(j)} \dWHs\bigg\|_X^{p_N}\bigg)^{1/p_N} \leq 4\sqrt{p_N} \bigg( \sum_{j=1}^N \E\|\Phi^{(j)}\|_{L^2(0,T;\gHX)}^{p_N}\bigg)^{1/p_N}
\\ & \leq 4\sqrt{p_N} N^{1/p_N} \|\Phi\|_{L^{p_N}(\Omega;\ell_N^\infty(L^2(0,T;\gHX)))}.
\end{align*}
Since $\sqrt{p_N} N^{1/p} = \mathrm{e}\sqrt{\log(N)} $, this proves the result for $p = p_N$. To deduce the result for arbitrary $p\in (0,p_N)$ note that by
Lenglart's inequality for increasing functions \cite[Theorem 2.2]{GeiSch} and with $r = p/p_N\in (0,1)$
\begin{align*}
I_N^\Phi(p)^p  = I_N^\Phi(r p_N)^{rp_N }  & \leq r^{-r} \big(4\mathrm{e}\sqrt{\log(N)}\big)^{p} \E \|\Phi\|_{\ell_N^\infty(L^2(0,T;\gHX))}^{rp_N } 
\\ & = r^{-r} \big(4\mathrm{e}\sqrt{\log(N)}\big)^{p} \|\Phi\|_{L^p(\Omega;\ell_N^\infty(L^2(0,T;\gHX)))}^p.
\end{align*}
Taking $1/p$-th powers, the result follows. Moreover, for $p\in [2, p_N)$ the result with the stated constant follows after using $r^{-r/p} = (\frac{p_N}{p})^{1/p_N}\leq (\frac{p_N}{2})^{1/p_N}\leq \exp(\frac{1}{2\mathrm{e}})$.

If $p\in(p_N,\infty)$, then using Minkowski's inequality, we obtain  
\begin{align*}
I_N^\Phi(p)^p &\leq \E \bigg|\sum_{j=1}^N  \sup_{t \in [0,T]} \bigg\| \int_0^t \Phi_s^{(j)} \dWHs\bigg\|_X^{p_N}\bigg|^{p/p_N} \leq \bigg(\sum_{j=1}^N  \bigg|\E \sup_{t \in [0,T]} \bigg\| \int_0^t \Phi_s^{(j)} \dWHs\bigg\|_X^{p}\bigg|^{p_N/p}\bigg)^{p/p_N}
\\ & \leq N^{p/p_N} \sup_{j\in \{1, \ldots N\}}\E \sup_{t \in [0,T]} \bigg\| \int_0^t \Phi_s^{(j)} \dWHs\bigg\|_X^{p} \leq (4\mathrm{e} \sqrt{p})^p  \sup_{j\in \{1, \ldots N\}} \E\|\Phi^{(j)}\|_{L^2(0,T;\gHX)}^{p},
\end{align*}
where we used \eqref{eq:BDG}  once more. Taking $1/p$-th powers and pulling the supremum over $j$ inside the expectation, the required estimate follows.

It remains to comment on the case $2\leq N\leq 7$. Again by Lenglart's inequality, it suffices to consider $p\in [2, \infty)$. In this case, the triangle inequality and \eqref{eq:BDG} give
\begin{align*}
I_N^\Phi &\leq \bigg(\sum_{j=1}^N \E \sup_{t \in [0,T]} \bigg\| \int_0^t \Phi_s^{(j)} \dWHs\bigg\|_X^p\bigg)^{1/p}
\leq B_p \bigg(\sum_{j=1}^N \|\Phi^{(j)}\|_{L^p(\Omega;L^2(0,T;\gHX))}^p\bigg)^{1/p}\\
&\leq 4\sqrt{p} N^{1/p} \|\Phi\|_{L^p(\Omega;\ell_N^\infty(L^2(0,T;\gHX)))}
\\ & \leq 4\exp\Big(1+ \frac{1}{2\mathrm{e}}\Big)\max\{\sqrt{\log(N)}, \sqrt{p}\} \|\Phi\|_{L^p(\Omega;\ell_N^\infty(L^2(0,T;\gHX)))},
\end{align*}
where the last estimate follows from $N^{1/p}\leq \sqrt{7} \leq \exp(1+ \frac{1}{2\mathrm{e}})$ for $2\leq N\leq 7$.
\end{proof}

\subsection{Approximation of semigroups and interpolation}

An integral part of approximating solutions of a stochastic evolution equation concerns the approximation of a semigroup by some scheme. The following definition allows us to quantify the approximation behaviour.

\begin{definition}
\label{def:orderScheme}
Let $X$ be a Hilbert space. An \emph{$\calL(X)$-valued scheme} is a function $R: [0,\infty) \to \calL(X)$. We denote $R_k \ce R(k)$ for $k \ge 0$. Let $Y$ be a Hilbert space which is continuously and densely embedded in $X$. If $A$ generates a $C_0$-semigroup $(S(t))_{t \ge 0}$ on $X$, an $\calL(X)$-valued scheme $R$ is said to \emph{approximate $S$ to order $\alpha>0$ on $Y$} or, equivalently, \emph{$R$ converges of order $\alpha$ on $Y$} if for all $T>0$ there is a constant $C_\alpha \ge 0$ such that
\begin{equation*}
    \|(S(jk)-R_k^j)u\|_X \le C_\alpha k^\alpha\|u\|_Y
\end{equation*}
for all $u \in Y$, $k>0$, and $j \in \N$ such that $jk \in [0,T]$.
An $\calL(X)$-valued scheme $R$ is said to be \emph{contractive} if $\|R_k\|_{\calL(X)} \le 1$ for all $k \ge 0$.
\end{definition}
Subsequently, we will omit the index for norms in the space $X$.
In the absence of nonlinear and noise terms, the following schemes approximate $S$ to different orders:

\begin{itemize}
\item exponential Euler (EE): $R_k = S(k)$, any order $\alpha >0$ on $X$;
\item implicit Euler (IE): $R_k = (1-kA)^{-1}$, order $\alpha \in (0,1]$ on $D((-A)^{2\alpha})$;
\item Crank--Nicolson (CN): $R_k = (2+kA)(2-kA)^{-1}$, order $\alpha \in (0,2]$ on $D((-A)^{3\alpha/2})$ provided that $(S(t))_{t \ge 0}$ is contractive.
\end{itemize}
Contractivity of the semigroup and the approximating scheme play a central role in our theory. While the contractivity of (EE) is immediate from the contractivity of the semigroup, we state a useful sufficient condition to verify the contractivity of rational schemes such as (IE) and (CN) below. One of the standard assumptions in the theory of semigroup approximation is that the scheme $R$ stems from a rational function $r:\C_{-}\to\C$ with $|r(z)|\le 1$ for all $z$ in the negative open halfplane $\C_-$. Under an additional consistency condition, this condition is known as A-acceptability \cite{brenner1979rational}, and it certainly holds for $A$-stable schemes \cite{DahlquistAstable}.

\begin{proposition}
\label{prop:functionalcalculus}
    Let $A$ be the generator of a $C_0$-semigroup of contractions on a Hilbert space $X$. Suppose that $r:\C_-\to \C$ is holomorphic, $|r(z)|\leq 1$ for all $z \in \C_{-}$, and let $R_k = r(kA)$ for $k>0$. Then $R$ is contractive.
\end{proposition}
\begin{proof}
    This is a consequence of the properties of the bounded $H^\infty$-calculus of $-A$ as the negative generator of a contraction semigroup, since $R_k=r(kA)=r(-k(-A))$ is defined via $H^\infty$-calculus. The underlying theorem can be found in \cite[Thm.~10.2.24]{AnalysisBanachSpacesII}. 
\end{proof}

As a consequence of this proposition, contractive schemes include (IE), (CN), and some higher-order implicit Runge-Kutta methods such as Radau methods, BDF(2), Lobatto IIA, IIB, and IIC as well as some DIRK schemes.

A common choice for the spaces $Y$ on which a given scheme approximates $S$ are domains of fractional powers of $A$. An important property of these spaces is that they embed into the real interpolation spaces with parameter $\infty$, i.e., for $\alpha>0$
\begin{equation}
\label{eq:DAalphaembeds}
    D(A^{\alpha}) \hookrightarrow D_A(\alpha, \infty).
\end{equation}
Here, $D_A(\alpha,\infty)$ denotes the real interpolation space $(X,D(A))_{\alpha,\infty}$. On later occasions, also the real interpolation spaces $(X,D(A))_{\alpha,2}$ will be used. See \cite{Lun,Tr1} for details on interpolation spaces.

Embeddings of the form \eqref{eq:DAalphaembeds} and properties of $D_A(\alpha,\infty)$ allow us to obtain decay rates for semigroup differences as follows.
Let $(S(t))_{t \ge 0}$ be a $C_0$-semigroup such that $\|S(t)\| \le Me^{\lambda t}$ for some $M \ge 1$ and $\lambda \ge 0$ for all $t \ge 0$. Such $M$ and $\lambda$ exist for every $C_0$-semigroup \cite[Prop.~5.5]{EngelNagel}.
Then $\|S(t)-S(s)\|_{\calL(X)} \le 2M\ee^{\lambda T}$ for $0 \le s \le t \le T$. Since
\begin{equation*}
    \|[S(t)-S(s)]x\|_X = \left\|\int_s^tS(r)Ax\;\mathrm{d}r\right\|_X \le M\ee^{\lambda T}(t-s)\|x\|_{D(A)}
\end{equation*}
for $x \in D(A)$, we have $\|S(t)-S(s)\|_{\calL(D(A),X)}\le 2M\ee^{\lambda T}(t-s)$. By interpolation,
\begin{equation*}
    \|S(t)-S(s)\|_{\calL(D_A(\alpha,\infty),X)} \le 2^{1-\alpha}M\ee^{\lambda T}(t-s)^\alpha \le 2M \ee^{\lambda T}(t-s)^\alpha
\end{equation*}
for $\alpha \in (0,1)$. Let $Y$ be another Hilbert space such that $Y \hra X$. Under the assumption that $Y \hra D_A(\alpha,\infty)$ continuously for some $\alpha \in (0,1)$ or $Y \hra D(A)$ continuously, in which case we set $\alpha=1$, this implies
\begin{equation}
\label{eq:interpolationSgDifferenceGeneral}
    \|S(t)-S(s)\|_{\calL(Y,X)} \le 2 \CY M \ee^{\lambda T} (t-s)^\alpha,
\end{equation}
where $\CY$ denotes the embedding constant of $Y$ into $D_A(\alpha,\infty)$ or $D(A)$.

\subsection{Gronwall type lemmas}

We need the following variants of the classical Gronwall inequality.
\begin{lemma}\label{lem:gronwallvar}
Let $\phi:[0,T]\to [0,\infty)$ be a continuous function and let $\alpha,\beta\in [0,\infty)$ be constants. Suppose that
\[\phi(t) \leq \alpha+\beta\Big(\int_0^t \phi(s)^2 ds\Big)^{1/2}, \ \ t\in [0,T].\]
Then
\[\phi(t) \leq \alpha (1+\beta^2 t)^{1/2}\exp\Big(\frac12+\frac12 \beta^2 t\Big) , \ \ t\in [0,T].\]
\end{lemma}
\begin{proof}
Using $(a+b)^2\leq (1+\theta) a^2 + (1+\theta^{-1}) b^2$ for $a,b\geq 0$ and $\theta>0$, we can write
\[\phi(t)^2 \leq (1+\theta) \alpha^2 +\beta^2 (1+\theta^{-1})\int_0^t \phi(s)^2 ds, \ \ t\in [0,T].\]
Therefore, applying Gronwall's inequality we see that
\[\phi(t)^2  \leq (1+\theta) \alpha^2 \exp(\beta^2 (1+\theta^{-1}) t).\]
Taking $\theta = \beta^2 t$ we obtain
\[\phi(t)^2  \leq (1+\beta^2 t) \alpha^2 \exp(\beta^2 t +1),\]
which gives the desired estimate.
\end{proof}

In the same way, one can prove the following discrete analogue by using the discrete version of Gronwall's lemma instead (see \cite[Proposition 5]{holteGronwall}).
\begin{lemma}
\label{lem:KruseGronwall}
    Let $\alpha,\beta \ge 0$ and $(\varphi_j)_{j \ge 0}$ be a non-negative sequence. If
    \begin{equation*}
        \varphi_j \le \alpha + \beta\left(\sum_{i=0}^{j-1} \varphi_i^2\right)^{1/2}~\text{ for } j \ge 0,
    \end{equation*}
    then
    \begin{equation*}
        \varphi_j \le \alpha (1+\beta^2 j)^{1/2} \exp\left(\frac12+\frac12\beta^2 j\right)~\text{ for }j \ge 0.
    \end{equation*}
\end{lemma}

\section{Convergence rates for additive noise}
\label{sec:rateOfConvergenceAdditive}

In this section, we present several results on convergence rates for linear equations with additive noise. The reason to start with this case is twofold. Higher convergence rates can be proved in this case. Moreover, it allows us to explain the new  techniques in a simpler setting, which can help understand the more complicated multiplicative setting of Section \ref{sec:rateOfConvergenceMultNoise}.

Consider the stochastic evolution equation with additive noise of the form
\begin{equation}
\label{eq:StEvolEqnAdditive}
   \rmd U = AU\,\rmd t + g(t)\,\rmd W_H(t) \text{ on }[0,T],~~U(0)=u_0 \in L_{\calF_0}^p(\Omega;X),
\end{equation}
where $A$ generates a $C_0$-semigroup $(S(t))_{t \ge 0}$ on a Hilbert space $X$ with norm $\|\cdot\|$, $W_H$ is an $H$-cylindrical Brownian motion for some Hilbert space $H$, and $p \in [2,\infty)$. For Hölder continuous noise $g\in L_\calP^p(\Omega;C^\alpha([0,T];\gHX))$, $\alpha \in (0,1]$, mapping into a space $Y\hookrightarrow X$, we prove rates of convergence for time discretisation schemes. An improvement of the rate is shown for the exponential Euler method for quasi-contractive semigroups. Results are illustrated for the nonlinear Schrödinger equation in Subsection \ref{subsec:SchroedingerAdd}.

The mild solution to \eqref{eq:StEvolEqnAdditive} for $t\in [0,T]$ is uniquely given by \cite[Chapters~5,6]{DaPratoZabczyk14}.
\begin{equation}
\label{eq:mildSolAdditive}
    U(t)=S(t)u_0+\int_0^{t} S(t-s) g(s)\dWHs.
\end{equation}
To approximate it, we employ a time discretisation scheme $R: [0,\infty) \to \calL(X)$ with time step $k>0$ on a uniform grid $\{t_j=jk:~j=0,\ldots, N_k\}\subseteq [0,T]$ with final time $T=t_{N_k}>0$ and $N_k=\frac{T}{k}\in \N$ being the number of time steps. The discrete solution is given by $U^0 \ce u_0$ and
\begin{align}
\label{eq:defUjAdditive}
    U^j &\ce R_k U^{j-1} +  R_k g(t_{j-1})\Delta W_j= R_k^j u_0 + \sum_{i=0}^{j-1} R_k^{j-i}g(t_i)\Delta W_{i+1} ,~~j=1,\ldots,N_k,
\end{align}
with Wiener increments $\Delta W_j \ce W_H(t_j)-W_H(t_{j-1})$, where we used \eqref{eq:convradonW}.

\subsection{General semigroups}
Our first result concerns general $C_0$-semigroups $S$. A further improvement under further conditions on $S$ is discussed in Subsection \ref{subsec:quasiContractive}.
Below, we denote the Hölder seminorm in $C^\alpha([0,T];\gHX)$ by $[\cdot]_{\alpha,X}$ for $\alpha \in (0,1]$ and let
\begin{equation} \label{eq:shorthandOnlyg}
    \nn g\nn_{p,\infty,Y} \ce \|g\|_{L^p(\Omega;C([0,T];\calL_2(H,Y)))} ,\quad g \in L^p(\Omega;C([0,T];\gHY)).
\end{equation}

\begin{theorem}
\label{thm:convRateAdditiveGeneral}
    Let $X$ and $Y$ be Hilbert spaces such that $Y\hookrightarrow X$. Let $A$ be the generator of a $C_0$-semigroup $(S(t))_{t \ge 0}$ on $X$ with $\|S(t)\| \le Me^{\lambda t}$ for some $M \ge 1$ and $\lambda \ge 0$. Let $(R_{k})_{k>0}$ be a time discretisation scheme and assume that $R$ approximates $S$ to order $\alpha \in (0,1]$ on $Y$. Suppose that $Y \hra D_A(\alpha,\infty)$ continuously if $\alpha \in (0,1)$ or $Y \hra D(A)$ continuously if $\alpha=1$. Let $p \in [2,\infty)$, $u_0 \in L_{\calF_0}^p(\Omega;Y)$, and $g\in L_\calP^p(\Omega;C([0,T];\gHY))$ as well as $g \in L_\calP^p(\Omega;C^\alpha([0,T];\gHX))$. Denote by $U$ the mild solution of \eqref{eq:StEvolEqnAdditive} and by $(U^j)_{j=0,\ldots,N_k}$ the temporal approximations as defined in \eqref{eq:defUjAdditive}. Then for $N_k \ge 2$
    \begin{equation*}
        \left\lVert\max_{0 \le j \le N_k}\|U(t_j)-U^j\|\right\rVert_p
        \le \big(C_1 + C_2 \sqrt{\max\{\log (T/k),p\}}\big) k^{\alpha}
    \end{equation*}
    with constants $C_1\ce C_\alpha \|u_0\|_{L^p(\Omega;Y)}$ and
    \begin{align*}
        C_2 &\ce \frac{K\sqrt{T}}{\sqrt{2\alpha+1}}\left( M \ee^{\lambda T}\big\| [g]_{\alpha,X} \big\|_p
        +\left(2M \ee^{\lambda T}\CY + C_\alpha \right) \nn g\nn_{p,\infty,Y} \right),
    \end{align*}
    where $C_\alpha$ is as in Definition \ref{def:orderScheme}, $K=4\exp(1+\frac{1}{2\mathrm{e}})$, and $\CY$ denotes the embedding constant of $Y$ into $D_A(\alpha,\infty)$ or $D(A)$.
    
    In particular, the approximations $(U^j)_j$ converge at rate $\min\{\alpha,1\}$ up to a logarithmic correction factor as $k \to 0$.
\end{theorem}

\begin{proof}
    Define $S^k(t)\ce R_k^j$ for $t\in (t_{j-1},t_j] $ and let $\lfloor t \rfloor$ as introduced above. Then the discrete solutions are given by the integral representation
    \begin{equation*}
        U^j = R_k^ju_0 + \int_0^{t_j} S^k(t_j-s)g(\lfloor s \rfloor) \dWHs.
    \end{equation*}
    Combining this representation with the mild solution formula \eqref{eq:mildSolAdditive}, the error can be bounded by
    \begin{align}
    \label{eq:E1234additive}
        E &\ce \Big\lVert\max_{0 \le j \le N_k}\|U(t_j)-U^j\|\Big\rVert_p \le \Big\lVert\max_{0 \le j \le N_k}\|[S(t_j) - R_k^j]u_0 \|\Big\rVert_p\nonumber\\
        &\phantom{=}+\Big\lVert\max_{0 \le j \le N_k}\Big\|\int_0^{t_j} S(t_j-s) [g(s)-g(\lfloor s \rfloor)] \dWHs \Big\|\Big\rVert_p\nonumber\\
        &\phantom{=}+\Big\lVert\max_{0 \le j \le N_k}\Big\|\int_0^{t_j} [S(t_j-\lfloor s \rfloor) -S(t_j-s)]g(\lfloor s \rfloor) \dWHs \Big\|\Big\rVert_p\nonumber\\
        &\phantom{=}+\Big\lVert\max_{0 \le j \le N_k}\Big\|\int_0^{t_j} [S(t_j-\lfloor s\rfloor) -S^k(t_j-s)]g(\lfloor s \rfloor) \dWHs \Big\|\Big\rVert_p\nonumber\\
        &\ec E_1 + E_2+E_3+E_4.
    \end{align}
    We proceed to estimate all four terms individually. Since $R$ approximates $S$ to order $\alpha$ on $Y$,
    \begin{equation}
    \label{eq:E1additive}
        E_1 \le C_\alpha \|u_0\|_{L^p(\Omega;Y)} k^\alpha.
    \end{equation}
    For the second term, we note that for $s \in [t_\ell, t_{\ell+1})$ for some $0 \le \ell \le N_k-1$, the definition of the Hölder seminorm $[\cdot]_\alpha$ implies that $\PP$-almost surely
    \begin{align*}
        \Big\| \sum_{i=0}^{j-1} \1_{[t_i,t_{i+1})}(s) S(t_j-s) [g(s)-g(t_i)]\Big\|_\gHX
        &\le \|S(t_j-s)\|_{\calL(X)} \|g(s)-g(t_\ell)\|_\gHX
\\ &        \le M\ee^{\lambda T} [g]_{\alpha,X} (s-t_\ell)^{\alpha}.
    \end{align*}
    Proposition \ref{prop:PropLogMainPaper} with $\Phi_s^{(j)}= \sum_{i=0}^{j-1} \1_{[t_i,t_{i+1})}(s) S(t_j-s)[g(s)-g(t_i)]$ then yields
    \begin{align}
    \label{eq:E2additive}
        &E_2 = \bigg\lVert\max_{0 \le j \le N_k}\bigg\|\int_0^{t_j} \sum_{i=0}^{j-1}\1_{[t_i,t_{i+1})}(s) S(t_j-s) [g(s)-g(t_i)] \dWHs \bigg\|\bigg\rVert_p\nonumber\\
        &\le K\sqrt{\max\{\log(N_k),p\}} \Big\lVert\Big(\int_0^{T}\max_{1 \le j \le N_k}\| \Phi_s^{(j)} \|_\gHX^2 \ds\Big)^{1/2}\Big\rVert_p\nonumber\\
        &\le KM \ee^{\lambda T} \sqrt{\max\{\log(N_k),p\}} \bigg\lVert\bigg(\sum_{l=0}^{N_k-1}\int_{t_\ell}^{t_{\ell+1}} [g]_{\alpha,X}^2(s-t_\ell)^{2\alpha} \ds\bigg)^{1/2}\bigg\rVert_p\nonumber\\
        &\le KM \ee^{\lambda T}\frac{1}{\sqrt{2\alpha+1}}  \sqrt{\max\{\log(N_k),p\}} k^{\alpha+1/2} \bigg\|\bigg(\sum_{l=0}^{N_k-1} [g]_{\alpha,X}^2 \bigg)^{1/2}\bigg\|_p\nonumber\\
        & = KM \ee^{\lambda T} \big\| [g]_{\alpha,X} \big\|_p \frac{\sqrt{T}}{\sqrt{2\alpha+1}}\sqrt{\max\{\log(N_k),p\}}k^\alpha,
    \end{align}
    where we have used Hölder continuity of $g$.
    Analogously, with $\Phi_s^{(j)}= \sum_{i=0}^{j-1} \1_{[t_i,t_{i+1})}(s) [S(t_j-t_i)-S(t_j-s)]g(t_i)$ for $E_3$ we obtain
    \begin{equation}
    \label{eq:E3additive}
        E_3 \le 2KM \ee^{\lambda T}\CY \frac{\sqrt{T}}{\sqrt{2\alpha+1}} \nn g \nn_{p,\infty,Y} \sqrt{\log(N_k)}k^{\alpha}
    \end{equation}
    using pathwise boundedness of $g$, i.e., $g(\omega,\cdot): [0,T] \to \calL_2(H,Y)$ being bounded for $\PP$-almost every $\omega \in \Omega$, and noting that by  \eqref{eq:interpolationSgDifferenceGeneral}
    \begin{align*}
        \big\|[S(t_j-t_\ell)-S(t_j-s)] g(t_\ell)\big\|_\gHX
        \le 2M \ee^{\lambda T}\CY (s-t_\ell)^\alpha \|g(t_\ell)\|_\gHY
    \end{align*}
    holds $\PP$-almost surely.
    Likewise, with $\Phi_s^{(j)}= \sum_{i=0}^{j-1} \1_{[t_i,t_{i+1})}(s) [S(t_j-t_i)-R_k^{j-i}]g(t_i)$, we obtain
    \begin{equation}
    \label{eq:E4additive}
        E_4 \le KC_\alpha \frac{\sqrt{T}}{\sqrt{2\alpha+1}} \nn g\nn_{p,\infty,Y} \sqrt{\log(N_k)} k^{\alpha},
    \end{equation}
    since $R$ approximates $S$ to order $\alpha$ on $Y$.
    The error bound follows from inserting \eqref{eq:E1additive}, \eqref{eq:E2additive}, \eqref{eq:E3additive}, and \eqref{eq:E4additive} into \eqref{eq:E1234additive}.
\end{proof}

For the \emph{exponential Euler method}, less regularity of the initial value suffices for the same convergence behaviour. The exponential Euler method is obtained by setting $R_k=S(k)$ in \eqref{eq:defUjAdditive}, i.e., we would solve exactly in the absence of noise $g$.

\begin{corollary}[Exponential Euler]
\label{cor:expEulerAdditive}
    Let $X$ and $Y$ be Hilbert spaces such that $Y\hookrightarrow X$. Let $A$ be the generator of a $C_0$-semigroup $(S(t))_{t \ge 0}$ on $X$ with $\|S(t)\| \le Me^{\lambda t}$ for some $M \ge 1$ and $\lambda \ge 0$. Assume that $g\in L_\calP^p(\Omega;C([0,T];\gHY))$ and $g \in  L_\calP^p(\Omega;C^\alpha([0,T];\gHX))$ for some $\alpha \in (0,1]$. Suppose that $Y \hra D_A(\alpha,\infty)$ continuously if $\alpha \in (0,1)$ or $Y \hra D(A)$ continuously if $\alpha=1$. Let $p \in [2,\infty)$ and $u_0 \in L_{\calF_0}^p(\Omega;X)$.
    Denote by $U$ the mild solution of \eqref{eq:StEvolEqnAdditive} and by $(U^j)_{j=0,\ldots,N_k}$ the temporal approximations as defined in \eqref{eq:defUjAdditive} obtained with the exponential Euler method $R \ce S$. Then for $N_k \ge 2$
    \begin{equation*}
        \bigg\lVert\max_{0 \le j \le N_k}\|U(t_j)-U^j\|\bigg\rVert_p
        \le C \sqrt{\max\{\log (T/k),p\}} k^{\alpha}
    \end{equation*}
    with constant
    \begin{align*}
        C &\ce KM \ee^{\lambda T} \frac{\sqrt{T}}{\sqrt{2\alpha+1}}\left(\left\| [g]_{\alpha,X} \right\|_p  + 2\CY \nn g\nn_{p,\infty,Y} \right),
    \end{align*}
    where $K=4\exp(1+\frac{1}{2\mathrm{e}})$ and $\CY$ denotes the embedding constant of $Y$ into $D_A(\alpha,\infty)$ or $D(A)$.

    In particular, if $Y \hra D(A)$ and $g$ is Lipschitz continuous as a map to $\gHX$, the approximations $(U^j)_j$ converge at rate $1$ up to a logarithmic correction factor as $k \to 0$.
\end{corollary}

\begin{proof}
    We split the error as in \eqref{eq:E1234additive}. For the exponential Euler method, the terms $E_1$ and $E_4$ in \eqref{eq:E1234additive} vanish due to $S(t_j)-R_k^j= S(jk)-S(k)^j=S(jk)-S(jk)=0$ and, likewise, $S(t_j-t_i)-R_k^{j-i}=0$. The error bound follows from inserting the bounds \eqref{eq:E2additive} and \eqref{eq:E3additive} of the remaining terms into \eqref{eq:E1234additive}.
\end{proof}

\subsection{Quasi-contractive Semigroups}
\label{subsec:quasiContractive}
Considering quasi-contractive semigroups, that is, $C_0$-semigroups $(S(t))_{t \ge 0}$ for which $\|S(t)\| \le e^{\lambda t}$ for some $\lambda \ge 0$ for all $t \ge 0$, allows us to eliminate the logarithmic factor for the exponential Euler method. The principle that lies at the heart of our proof is the maximal inequality from Theorem \ref{thm:maxIneqQuasiContractive}, which is used to estimate the stochastic convolutions in the error term. Depending on the spatial regularity of the noise $g$, the convergence rate $\alpha \in (0,1]$ is attained without a logarithmic correction factor.

\begin{theorem}[exponential Euler, quasi-contractive case]
\label{thm:convRateAdditiveExpEulerQuasiC}
    Adopt the notation and assumptions of Corollary \ref{cor:expEulerAdditive}. In addition, assume that $\|S(t)\| \le \ee^{\lambda t}$ for some $\lambda \ge 0$ for all $t \in [0,T]$. Then for $N_k \geq 2$
    \begin{equation*}
        \left\lVert\max_{0 \le j \le N_k}\|U(t_j)-U^j\|\right\rVert_p
        \le C k^\alpha
    \end{equation*}
    with constant
    \begin{align*}
        C \ce \frac{B_p\sqrt{T}}{\sqrt{2\alpha+1}} \left(\ee^{\lambda T}\left\| [g]_{\alpha,X} \right\|_p + 2 \CY \ee^{2\lambda T} \nn g \nn_{p,\infty,Y}\right),
    \end{align*}
    where $B_{p}$ is the constant from Theorem \ref{thm:maxIneqQuasiContractive}.
\end{theorem}

\begin{proof}
    We bound the error as in \eqref{eq:E1234additive}, where the first and fourth term vanish as discussed in the proof of Corollary \ref{cor:expEulerAdditive}. We proceed to bound the remaining terms using the maximal inequality from Theorem \ref{thm:maxIneqQuasiContractive} instead of Proposition \ref{prop:PropLogMainPaper} to obtain
    \begin{align}
    \label{eq:E2additiveQuasiC}
        E_2 &\le \bigg\lVert\sup_{t \in [0,T]} \bigg\|\int_0^{t} S(t-s) [g(s)-g(\lfloor s \rfloor)] \dWHs \bigg\|\bigg\rVert_p \nonumber\\
        &\le B_p \ee^{\lambda T}\bigg\lVert\bigg(\int_0^{T}\left\| g(s)-g(\lfloor s \rfloor )\right\|_\gHX^2 \ds\bigg)^{1/2}\bigg\rVert_p\nonumber\\
        &\le B_p \ee^{\lambda T}\bigg\lVert\bigg(\sum_{i=0}^{N_k-1}\int_{t_i}^{t_{i+1}}[g]_{\alpha,X}^2(s-t_i)^{2\alpha} \ds\bigg)^{1/2}\bigg\rVert_p\nonumber\\
        & \le \frac{B_p\ee^{\lambda T}\sqrt{T}}{\sqrt{2\alpha+1}} \big\| [g]_{\alpha,X} \big\|_p k^\alpha
    \end{align}
    by Hölder continuity of $g$. Analogously, for $E_3$ we deduce from the semigroup bound \eqref{eq:interpolationSgDifferenceGeneral} that
    \begin{align}
        \label{eq:E3additiveQuasiC}
        E_3 &\le \bigg\lVert\sup_{t \in [0,T]}\bigg\|\int_0^{t} S(t-s)[S(s-\lfloor s \rfloor) -I]g(\lfloor s \rfloor) \dWHs \bigg\|\bigg\rVert_p\nonumber\\
        &\le B_p \ee^{\lambda T}\bigg\lVert\bigg(\int_0^{T}\left\| [S(s-\lfloor s \rfloor) -I]g(\lfloor s \rfloor)\right\|_\gHX^2 \ds\bigg)^{1/2}\bigg\rVert_p\nonumber\\
        &\le 2 B_p \ee^{2\lambda T} \CY \bigg\lVert\bigg(\sum_{i=0}^{N_k-1}\int_{t_i}^{t_{i+1}} (s-t_i)^{2\alpha}\left\|g(t_i)\right\|_\gHY^2 \ds\bigg)^{1/2}\bigg\rVert_p\nonumber\\
        & \le 2 B_p \ee^{2\lambda T}\CY \frac{\sqrt{T}}{\sqrt{2\alpha+1}} \nn g\nn_{p,\infty,Y} k^\alpha.
    \end{align}
    The final error bound follows from adding \eqref{eq:E2additiveQuasiC} and \eqref{eq:E3additiveQuasiC}.
\end{proof}

In particular, convergence rate $1$ is attained without logarithmic correction factor for spatially sufficiently regular noise $g$. General, possibly irregular initial values $u_0 \in L_{\calF_0}^p(\Omega;X)$ are still admissible as the following corollary shows.

\begin{corollary}
    Let $X$ be a Hilbert space and let $A$ be the generator of a quasi-contractive $C_0$-semigroup on $X$ with parameter $\lambda>0$. Assume that $g\in L_\calP^p(\Omega;C([0,T];\calL_2(H,D(A))))$ and is pathwise Lipschitz continuous as a map to $\gHX$. Let $p \in [2,\infty)$ and $u_0 \in L_{\calF_0}^p(\Omega;X)$. Denote by $U$ the mild solution of \eqref{eq:StEvolEqnAdditive} and by $(U^j)_{j=0,\ldots,N_k}$ the temporal approximations as defined in \eqref{eq:defUjAdditive} obtained with the exponential Euler method $R \ce S$.
    Then there is a constant $C\ge 0$ depending on $(g,T,p,\alpha,\lambda,X,D(A))$ such that for $N_k \geq 2$
    \begin{equation*}
        \left\lVert\max_{0 \le j \le N_k}\|U(t_j)-U^j\|\right\rVert_p
        \le C k,
    \end{equation*}
    i.e., the approximations $(U^j)_j$ converge at rate $1$ as $k \to 0$.
\end{corollary}

\subsection{Application to the linear Schrödinger equation with additive noise}
\label{subsec:SchroedingerAdd}

In this subsection, we study convergence rates of time discretisations of the linear stochastic Schrödinger equation with a potential and additive noise
\begin{align}
\label{eq:linearSchroedingerAddNoise}
    \Bigg\{\begin{split} \rmd u &= -\iu(\Delta + V) u \;\rmd t-\iu\; \rmd W~~~ \text{ on }[0,T],\\
    u(0)&=u_0
    \end{split}
\end{align}
in $\R^d$ for $d \in \N$, where $\{W(t)\}_{t\ge 0}$ is a square-integrable
$\K$-valued $Q$-Wiener process (see Subsection \ref{sec:SI}), $\K\in \{\R,\C\}$, with respect to a normal filtration $(\filtrF_t)_{t \ge 0}$, $V$ is a $\K$-valued potential, $u_0$ is an $\filtrF_0$-measurable random variable, $\iu$ is the imaginary unit, and $\Delta$ the Laplace operator on $\R^d$. Next, we introduce conditions on the dimension and the regularity of $V$. With a slight variation of the methods below, one can also consider \eqref{eq:linearSchroedingerAddNoise} on $[0,L]^d$ with periodic boundary conditions. More general domains with Dirichlet or Neumann boundary conditions can be treated as well, but for this, suitable adjustments are needed in the proofs below.

Let $\sigma \ge 0$ and, for this subsection only, write $L^2=L^2(\R^d)$ and $H^\sigma=H^\sigma(\R^d)$. We will also be using the Bessel potential spaces $H^{\sigma,q}(\R^d)$, which coincide with the classical Sobolev spaces $W^{\sigma,q}(\R^d)$ if $\sigma\in \N$ and $q\in (1, \infty)$. For details on these spaces the reader is referred to \cite{BL76,Tr1}.

To ensure the well-posedness of \eqref{eq:linearSchroedingerAddNoise}, we assume one of the following mutually exclusive conditions holds.
\begin{assumption}
\label{ass:sigmadVSchrodinger}
    Let $\sigma \ge 0$, $d \in \N$ and $V \in L^2$ such that
    \begin{enumerate}[label=(\roman*)]
    \item $\sigma > \frac{d}{2}$ and $V \in H^\sigma$, or \label{item:sigmad2}
    \item $\sigma = 0$ and $V \in H^\beta$ for some $\beta>\frac{d}{2}$, or \label{item:sigma0}
    \item $\sigma \in (0,1)$, $d> 2 \sigma$, and $V \in H^\beta$ for some $\beta>\frac{d}{2}$, or
    \label{item:sigma01}
    \item $\sigma =1$, $d\ge 2$, and $V \in H^\beta$ for some $\beta>\frac{d}{2}$. \label{item:sigma1}
\end{enumerate}
\end{assumption}
In particular, this assumption implies that $Vu \in H^\sigma$ for any $u \in H^\sigma$ and $\|Vu\|_{H^\sigma} \le C_V \|u\|_{H^\sigma}$ for some constant $C_V \ge 0$ depending on $V$. This follows from the algebra property of $H^\sigma$ in case \ref{item:sigmad2}. Note that while \ref{item:sigmad2} is taken verbatim from \cite[Prop.~4.1]{AC18}, cases \ref{item:sigma0} and \ref{item:sigma1} assume less regularity in our assumption and case \ref{item:sigma01} is new. In the second case \ref{item:sigma0}, Hölder's inequality and the Sobolev embedding $H^\beta \hra L^\infty$ for $\beta >\frac{d}{2}$ yield
\begin{equation*}
    \|Vu\|_{L^2} \le \|V\|_{L^\infty}\|u\|_{L^2} \lesssim \|V\|_{H^\beta}\|u\|_{L^2}
\end{equation*}
in the case \ref{item:sigma0}, see \cite[Prop.~4.1]{AC18}. The case \ref{item:sigma01} is covered by Lemma \ref{lem:SoAddSigma01case} below. Lastly, $\|Vu\|_{H^1} \lesssim \|u\|_{H^1}$ in the case \ref{item:sigma1} follows from Hölder's inequality, once with $p=2\beta$ and $q=\frac{4\beta}{2\beta-2}$, $\beta>1$, and the embeddings $H^\beta \hra L^\infty$, $H^1 \hra L^q$, as well as $H^\beta \hra H^{1,2\beta}$ via
\begin{align*}
    \|Vu\|_{H^1}^2 &\lesssim \|Vu\|_{L^2}^2+\|Vu'\|_{L^2}^2+\|V'u\|_{L^2}^2 \\
    &\le \|V\|_{L^\infty}^2 (\|u\|_{L^2}^2+\|u'\|_{L^2}^2)+\|V'\|_{L^{2\beta}}^2\|u\|_{L^q}^2\\
    &\lesssim (\|V\|_{H^\beta}^2+\|V\|_{H^{1,2\beta}}^2) \|u\|_{H^1}^2 \lesssim \|V\|_{H^\beta}^2 \|u\|_{H^1}^2.
\end{align*}
Hence, multiplication by $V$ is a bounded operator on $H^\sigma$ if Assumption \ref{ass:sigmadVSchrodinger} holds.

\begin{lemma}
\label{lem:SoAddSigma01case}
        Let $\sigma \in (0,1)$, $d \in \N$ such that $d > 2\sigma$, and $V \in H^{\beta}(\R^d)$ for some $\beta>\frac{d}{2}$. Then  $\|Vu\|_{H^\sigma} \le C_V \|u\|_{H^\sigma}$ for some constant $C_V \ge 0$ for all $u \in H^\sigma(\R^d)$.
\end{lemma}
\begin{proof}
    Let $q_1 = \frac{2d}{d-2\sigma}$ and $q_2=\frac{d}{\sigma}$. Then $\frac{1}{q_1}+\frac{1}{q_2}=\frac{1}{2}$ and $q_1<\infty$ because $d>2\sigma$. By classical Sobolev and Bessel potential space embeddings \cite[Thm.~6.5.1]{BL76}, $H^{d/2} \hra H^{\sigma,q_2}$, $H^\sigma \hra L^{q_1}$, and $H^\beta \hra C_b(\R^d) \hra L^\infty$. Thus, an application of the product estimate \cite[Prop.~2.1.1]{TaylorBook} yields
    \begin{align*}
        \|Vu\|_{H^\sigma} &\lesssim \|V\|_{H^{\sigma,q_2}}\|u\|_{L^{q_1}} + \|V\|_{L^\infty}\|u\|_{H^\sigma}\lesssim (\|V\|_{H^{d/2}}+\|V\|_{H^\beta})\|u\|_{H^\sigma} \lesssim \|V\|_{H^\beta}\|u\|_{H^\sigma}. \qedhere
    \end{align*}
\end{proof}

Since $-\iu \Delta$ generates a contractive semigroup \cite[Lemma~2.1]{AC18}, its bounded perturbation $-\iu(\Delta + V)$ generates a quasi-contractive semigroup \cite[Thm.~III.1.3]{EngelNagel}. Thus, we are in the setting of Subsection \ref{subsec:quasiContractive}. Global existence and uniqueness of mild solutions $U \in L^p(\Omega;C([0,T];H^\sigma))$ to \eqref{eq:linearSchroedingerAddNoise} in $H^\sigma$ are guaranteed provided that $p \in [2,\infty)$, $u_0 \in L_{\calF_0}^p(\Omega;H^\sigma)$, $Q^{1/2} \in \calL_2(L^2,H^\sigma)$, and Assumption \ref{ass:sigmadVSchrodinger} holds.

Therefore, the Schrödinger equation \eqref{eq:linearSchroedingerAddNoise} can be rewritten in the form of \eqref{eq:StEvolEqnAdditive} on $X=H^\sigma$ with an $H$-cylindrical Brownian motion $W_H$ for $H=L^2$.

For the exponential Euler method, we recover the error bound from \cite[Thm.~4.3]{AC18}, showing convergence of rate $1$ in the case of sufficiently regular $Q^{1/2}$ under less regularity assumptions on $V$. Moreover, under weaker regularity assumptions on $Q^{1/2}$ and $V$, we additionally provide an error bound for fractional convergence rates $\alpha \in (0,1]$.

\begin{theorem}
\label{thm:SoExpEulerAdditive}
    Let $\sigma \ge 0$, $d \in \N$, and $V \in L^2$ satisfy Assumption \ref{ass:sigmadVSchrodinger}, and let $p \in [2,\infty)$. Assume that $u_0 \in L_{\mathcal{F}_0}^p(\Omega;H^\sigma)$ and
    $Q^{1/2} \in \calL_2(L^2,H^{\sigma+2\alpha})$ for some $\alpha \in (0,1]$. Denote by $U$ the mild solution of the linear stochastic Schrödinger equation  with additive noise \eqref{eq:linearSchroedingerAddNoise} and by $(U^j)_{j=0,\ldots,N_k}$ the temporal approximations as defined in \eqref{eq:defUjAdditive} obtained with the exponential Euler method $R\ce S$. Then there exists a constant $C \ge 0$ depending on $(V, u_0,T,p,\alpha,\sigma,d)$ such that for $N_k \geq 2$
    \begin{equation*}
        \left\| \max_{0 \le j \le N_k} \|U(t_j)-U^j\|_{H^\sigma} \right\|_p \le C \|Q^{1/2}\|_{\calL_2(L^2,H^{\sigma+2\alpha})} k^\alpha.
    \end{equation*}
\end{theorem}
\begin{proof}
    As discussed above, $A= -\iu(\Delta + V)$ generates a quasi-contractive semigroup on $H^\sigma$. Furthermore, setting $g = -\iu Q^{1/2}$ allows us to rewrite \eqref{eq:linearSchroedingerAddNoise} in the form of a stochastic evolution equation \eqref{eq:StEvolEqnAdditive}. Thus, Theorem \ref{thm:convRateAdditiveExpEulerQuasiC} is applicable with $X=H^\sigma$ and $H=L^2$. It remains to check that $g \in L_\calP^p(\Omega; C([0,T];\gHY))$ for some $Y \hra D_A(\alpha,\infty)$ and that $g \in L^p_\calP(\Omega; C^\alpha([0,T];\gHX))$. The latter holds for any $\alpha \in (0,1]$ due to $g$ being constant in time. Taking $Y=H^{\sigma+2\alpha} = (H^\sigma, H^{\sigma+2})_{\alpha,2}=(H^\sigma,D(A))_{\alpha,2}\hra (H^\sigma, D(A))_{\alpha,\infty}$, the first condition is satisfied as well. Corollary \ref{cor:expEulerAdditive} yields the desired error bound.
\end{proof}

Furthermore, Theorem \ref{thm:convRateAdditiveGeneral} enables us to extend \cite[Thm.~4.3]{AC18} to general discretisation schemes $R$ involving rational approximations, at the price of an additional logarithmic factor. We state it for the implicit Euler and the Crank--Nicolson method.

\begin{theorem}
\label{thm:SoIECNAdditive}
    Let $\sigma \ge 0$, $d \in \N$, and $V \in L^2$ satisfy Assumption \ref{ass:sigmadVSchrodinger}, and let $p \in [2,\infty)$. Let $(R_k)_{k>0}$ be the implicit Euler method (IE) or the Crank--Nicolson method (CN) and set $\ell=4$ or $\ell=3$, respectively. Assume that $u_0 \in L_{\mathcal{F}_0}^p(\Omega;H^{\sigma+\ell\alpha})$ and
    $Q^{1/2} \in \calL_2(L^2,H^{\sigma+\ell\alpha})$ for some $\alpha \in (0,1]$.
    Denote by $U$ the mild solution of the linear stochastic Schrödinger equation  with additive noise \eqref{eq:linearSchroedingerAddNoise} and by $(U^j)_{j=0,\ldots,N_k}$ the temporal approximations as defined in \eqref{eq:defUjAdditive}.
    Then there exists a constant $C \ge 0$ depending on $(V, u_0,T,p,\alpha,\sigma,d,\ell)$ such that for $N_k \ge 2$
    \begin{equation*}
        \left\| \max_{0 \le j \le N_k} \|U(t_j)-U^j\|_{H^\sigma} \right\|_p \le C \big(1+\|Q^{1/2}\|_{\calL_2(L^2,H^{\sigma+\ell\alpha})}\big)\sqrt{\log(T/k)} k^\alpha.
    \end{equation*}
\end{theorem}
\begin{proof}
    This follows from Theorem \ref{thm:convRateAdditiveGeneral} noting that (IE) approximates $S$ to order $\alpha$ on $D((-A)^{2\alpha})$ and this fractional domain is given by $D((\iu \Delta)^{2\alpha})=H^{\sigma+4\alpha}$, which is chosen as the space $Y$. Likewise, (CN) approximates $S$ to order $\alpha$ on $D((-A)^{3\alpha/2})=H^{\sigma+3\alpha}$.
\end{proof}

Comparing this result to Theorem \ref{thm:SoExpEulerAdditive} for the exponential Euler method (EE), it becomes apparent that lower-order schemes like (IE) need higher regularity of the noise $Q^{1/2}$ to achieve the same rate of convergence ($\calL_2(L^2,H^{\sigma+4\alpha})$ compared to $\calL_2(L^2,H^{\sigma+2\alpha})$). For instance, for $Q^{1/2}\in\calL_2(L^2,H^{\sigma+2})$, the rates for (EE), (CN), and (IE) are $1$, $\frac23$, and $\frac12$, respectively. If $Q^{1/2}\in\calL_2(L^2,H^{\sigma+3})$, (EE) and (CN) have the same convergence rates up to a logarithmic factor, and if $Q^{1/2}\in\calL_2(L^2,H^{\sigma+4})$, so does (IE), all provided that $V$ and $u_0$ are sufficiently smooth. 

Note that in the absence of a potential, the same convergence rates are obtained without any limitation on the dimension $d \in \N$ in terms of the parameter $\sigma$. An analogue of Theorem \ref{thm:SoIECNAdditive} can be obtained for other implicit Runge-Kutta methods if the space is known on which the scheme approximates the semigroup to a given order.

\section{Well-posedness}
\label{sec:wellposed}

We consider the stochastic evolution equation with multiplicative noise
\begin{align}
\label{eq:stEvolEqnFG_WP}
    \Bigg\{\begin{split} \rmd U &=(A U + F(t,U))\,\rmd t  + G(t,U) \,\rmd W_H~~~\text{ on } [0,T],\\ U(0) &= u_0 \in L_{\filtrF_0}^p(\Omega;X)
    \end{split}
\end{align}
for $1 \le p <\infty$ and $A$ generating a $C_0$-semigroup $(S(t))_{t \ge 0}$ of contractions on $X$. In this section, we present progressive measurability, linear growth and global Lipschitz conditions on $F$ and $G$ ensuring the well-posedness of the above equation.

\begin{assumption}
\label{ass:FG_WP}
    Let $X$ be a Hilbert space and let $p \in [2,\infty)$.
    Let $F:\Omega \times [0,T] \times X \to X, F(\omega, t,x) = \tilde{F}(\omega,t,x) + f(\omega,t)$ and $G:\Omega \times [0,T] \times X \to \gHX, G(\omega, t,x) = \tilde{G}(\omega,t,x) + g(\omega,t)$ be strongly $\calP\otimes \calB(X)$-measurable, and such that $\tilde{F}(\cdot,\cdot,0) = 0$ and $\tilde{G}(\cdot,\cdot,0) = 0$, and suppose
    \begin{enumerate}[label=(\alph*)]
        \item\label{item:FG_WP_globalLipschitz} \emph{(global Lipschitz continuity on $X$)} there exist constants $\CFX, \CGX \ge 0$ such that for all $\omega \in \Omega, t \in [0,T]$ and $x,y\in X$, it holds that
        \begin{align*}
            \|\tilde{F}(\omega,t,x)-\tilde{F}(\omega,t,y)\| &\le \CFX\|x-y\|,\\
            \|\tilde{G}(\omega,t,x)-\tilde{G}(\omega,t,y)\|_\gHX &\le \CGX\|x-y\|,
        \end{align*}
        \item\label{item:FG_WP_integrability} \emph{(integrability)} $f \in L^p_\calP(\Omega; L^1(0,T;X))$ and $g \in L^p_\calP(\Omega; L^2(0,T;\gHX))$.
    \end{enumerate}
\end{assumption}
Note that Assumption \ref{ass:FG_WP} implies linear growth of $F$ and $G$:
\begin{equation}\label{eq:lineargrowthX}
            \|\tilde{F}(\omega,t,x)\| \le \CFX(1+\|x\|),~ \|\tilde{G}(\omega,t,x)\|_\gHX \le \CGX(1+\|x\|),
\end{equation}
where the constant $1$ can be left out, but is included for later use in Theorem \ref{thm:wellposedY}.

Well-posedness shall be understood in the sense of existence and uniqueness of mild solutions to \eqref{eq:stEvolEqnFG_WP}. Denote by $L^0(\Omega;V)$ the space of all strongly measurable $V$-valued random variables for Banach spaces $V$.
\begin{definition}
    A $U\in L^0_{\calP}(\Omega;C([0,T];X))$ is called a \emph{mild solution} to \eqref{eq:stEvolEqnFG_WP} if a.s.\ for all $t \in [0,T]$
    \begin{equation*}
        U(t) = S(t)u_0 + \int_0^t S(t-s) F(s,U(s)) \,\rmd s + \int_0^t S(t-s) G(s,U(s)) \,\rmd W_H(s).
    \end{equation*}
\end{definition}

The following well-posedness result is more or less standard \cite[Chapters~6,7]{DaPratoZabczyk14}.
\begin{theorem}
\label{thm:wellposed}
    Suppose that Assumption \ref{ass:FG_WP} holds for some $p \in [2,\infty)$. Let $A$ be the generator of a $C_0$-contraction semigroup $(S(t))_{t \ge 0}$ on $X$. Let $u_0 \in L_{\calF_0}^p(\Omega;X)$.
   Then \eqref{eq:stEvolEqnFG_WP} has a unique mild solution $U \in L^p(\Omega;C([0,T];X))$. Moreover,
    \begin{align*}
        \|U\|_{L^p(\Omega;C([0,T];X))} \le \Cbdd^X\Big(&1+\|u_0\|_{L^p(\Omega;X)} + \|f\|_{L^p(\Omega;L^1(0,T;X))}+ B_{p}\|g\|_{L^p(\Omega;L^2(0,T;\gHX))} \Big),
    \end{align*}
    where $\Cbdd^X \ce (1+C^2 T)^{1/2}\ee^{(1+C^2T)/2}$  with
        $C \ce C_{F,X}T^{1/2}+B_{p}C_{G,X}$, and $B_{p}$ is the constant from Theorem \ref{thm:maxIneqQuasiContractive}.
\end{theorem}
\begin{proof}
    First, the local existence and uniqueness of solutions are to be proven. Second, local solutions are concatenated to obtain global existence and uniqueness. We only sketch the steps. Let $\delta \in (0,T]$. Define the spaces $Z_\delta \ce L^p(\Omega;C([0,\delta];X))$, $Z \ce Z_T$, $Z_\delta^{\calP}$ as the subset of all adapted $v \in Z_\delta$ and $Z^{\calP} \ce Z_T^{\calP}$. For $v \in Z_{\delta}^{\calP}$, we define the fixed point functional
    \begin{equation}
    \label{eq:defGamma}
        \Gamma v(t) \ce S(t)u_0 + \int_0^t S(t-s)F(s,v(s))\,\rmd s + \int_0^t S(t-s) G(s,v(s))\,\rmd W_H(s).
    \end{equation}
    The problem of finding local mild solutions of \eqref{eq:stEvolEqnFG_WP} then reduces to finding fixed points $v \in Z_\delta^{\calP}$ of $\Gamma$. The contraction mapping theorem yields such unique fixed points provided that $\Gamma$ is a contraction which maps $Z^{\calP}$ and thus $Z_\delta^{\calP}$ into itself. That is, \textit{(i)} continuity of paths of $\Gamma v$ and maximal estimates for $v \in Z_\delta^{\calP}$ (see Theorem \ref{thm:maxIneqQuasiContractive}) as well as \textit{(ii)} adaptedness of $\Gamma v$, and that \textit{(iii)} $\Gamma$ is a (strict) contraction on $Z_\delta^{\calP}$. Lastly, we consider the evolution equation on $[\delta,2\delta]$ with initial value $U(\delta)$ to extend the solution to larger time intervals.

    It remains to prove the a priori estimate for the mild solution $U$. Let $r \in [0,T]$. Let $\psi(r) = 1+\left\|\sup_{t\in [0,r]}\|U(t)\|\right\|_p$.
     From the triangle inequality, Theorem \ref{thm:maxIneqQuasiContractive} and \eqref{eq:lineargrowthX} we see that
    \begin{align*}
    \psi(r) & \leq 1+\|u_0\|_{L^p(\Omega;X)} + C_{F,X}\left\|\int_0^r 1+ \|U(s)\|\ds \right\|_p+\left\|f\right\|_{L^p(\Omega;L^1(0,r;X))} \\ & \ +B_{p} \left[C_{G,X} \left\|\left(\int_0^r (1+\|U(s)\|)^2 ds\right)^{1/2}\right\|_p \ds+\|g\|_{L^p(\Omega;L^2(0,r;\gHX))} \right]
    \\ & \leq c_{u_0,f,g} + C_{F,X}\int_0^r \psi(s) \ds+  B_{p} C_{G,X}\left(\int_0^r \psi(s)^2 \ds\right)^{1/2}
    \\ & \leq c_{u_0,f,g} + C \left(\int_0^r \psi(s)^2 \ds\right)^{1/2},
    \end{align*}
where $c_{u_0, f, g} = 1+\|u_0\|_{L^p(\Omega;X)} + \|f\|_{L^p(\Omega;L^1(0,T;X))}+ B_{p}\|g\|_{L^p(\Omega;L^2(0,T;\gHX))}$ and $C = C_{F,X} T^{1/2}+B_{p} C_{G,X}$. Here we used Minkowski's inequality to pull in the $L^{p}(\Omega)$ and  $L^{p/2}(\Omega)$ norms.
Lastly, the version of Gronwall's inequality from Lemma \ref{lem:gronwallvar} yields the desired result
\begin{equation*}
\psi(T) \leq c_{u_0, f, g} (1+ C^2 T)^{1/2} e^{(1+C^2T)/2}. \quad\qedhere
\end{equation*}
\end{proof}

Lastly, we present a well-posedness result on subspaces $Y \hra X$ which does not require Lipschitz continuity of $\tilde{F}, \tilde{G}$ on $Y$ but merely linear growth. The reader is referred to Remark \ref{rem:Lipschitz-growthY} below for a discussion where we explain why Lipschitz continuity on $Y$ should be avoided.

\begin{theorem}
\label{thm:wellposedY}
    Suppose that Assumption \ref{ass:FG_WP} holds. Let $Y \hra X$ be a Hilbert space and $A$ the generator of a $C_0$-contraction semigroup $(S(t))_{t \ge 0}$ on both $X$ and $Y$. Let $p \in [2,\infty)$ and $u_0 \in L_{\calF_0}^p(\Omega;Y)$.
    Additionally, suppose that $f\in L^p_\calP(\Omega;L^1(0,T;Y))$, $g\in L_\calP^p(\Omega;L^2(0,T;\gHY))$, $F:\Omega \times [0,T] \times Y \to Y$, $G: \Omega \times [0,T] \times Y \to \gHY$ are strongly $\calP\otimes \mathcal{B}(Y)$-measurable, and there are $\LFY,\LGY \ge 0$ such that for all $\omega \in\Omega$, $t \in [0,T]$, and $x \in Y$,
    \begin{equation*}
        \|\tilde{F}(\omega,t,x)\|_Y \le \LFY(1+\|x\|_Y),~
        \|\tilde{G}(\omega,t,x)\|_\gHY \le \LGY(1+\|x\|_Y).
    \end{equation*}
    Under these conditions the mild solution $U\in L^p(\Omega;C([0,T];X))$ to \eqref{eq:stEvolEqnFG_WP} is in $L^p(\Omega;C([0,T];Y))$ and
    \begin{align*}
        \|U\|_{L^p(\Omega;C([0,T];Y))}\leq \Cbdd^Y\Big(&1+\|u_0\|_{L^p(\Omega;Y)}+ \|f\|_{L^p(\Omega;L^1(0,T;Y))}+B_{p}\|g\|_{L^p(\Omega;L^2(0,T;\gHY))}\Big),
    \end{align*}
    where $\Cbdd^Y \ce (1+C^2 T)^{1/2}\ee^{(1+C^2T)/2}$ with
    $C \ce L_{F,Y}T^{1/2}+B_{p}L_{G,Y}$, and $B_{p}$ is the constant from Theorem \ref{thm:maxIneqQuasiContractive}.
\end{theorem}
The constant $C$ appears exponentially in the above. In the special case $p=2$, $L_{F,Y} = L_{G,Y} = T = 1$, this leads to $\Cbdd^Y\leq \sqrt{10} e^{5} \leq 470$.

\begin{proof}
Recall that by Banach's fixed point theorem for $\delta\leq T_0$, where $T_0\in(0,1]$ only depends on $p$, $C_{F,X}$, $C_{G,X}$ and $X$,
one has $U = \lim_{n\to \infty} U_n$ in $L^p(\Omega;C([0,\delta];X))$, where $U_0 = u_0$ and $U_{n+1} = \Gamma(U_n)$ with $\Gamma$ as defined in \eqref{eq:defGamma}.
Since $F$ and $G$ map $Y$ into $Y$, we can also consider $\Gamma$ as a mapping on $Z^{2}\ce L_\calP^p(\Omega;L^2(0,\delta;Y))$ to eventually show that $U$ is in $L^p_{\calP}(\Omega;C([0,\delta];Y)) \seq Z^{2}$. Note that for $U\in Z^2$, $F(\cdot, U)$ and $G(\cdot, U)$ are progressively measurable as $Y$ and $\gHY$-valued mappings by \cite[Theorem 1.1.6]{AnalysisBanachSpacesI}. Moreover, we claim that for all $v\in Z^{2}$,
\begin{align}
\label{eq:Gammav}
    \|\Gamma(v)\|_{L^p(\Omega;C([0,\delta];Y))} &\leq \|u_0\|_{L^p(\Omega;Y)} + \|f\|_{L^p(\Omega;L^1(0,\delta;Y))}\nonumber\\
    &\phantom{\leq }+B_{p}\|g\|_{L^p(\Omega;L^2(0,\delta;\gHY))}+\big(\LFY + B_{p} \LGY\big) (1+ \|v\|_{Z^{2}}).
\end{align}
Indeed, since $S$ is contractive, the maximal inequality, linear growth of $\tilde{F}$ and $\tilde{G}$ on $Y$, and $\delta\leq 1$ imply
\begin{align*}
    \|\Gamma(v)  - S(\cdot) u_0\|_{L^p(\Omega;C([0,\delta];Y))}
    &\leq \|F(\cdot,v)\|_{L^p(\Omega;L^1(0,\delta;Y))} +  B_{p}\|G(\cdot,v)\|_{L^p(\Omega;L^2(0,\delta;\gHY))}\\
    &\leq \|f\|_{L^p(\Omega;L^1(0,\delta;Y))}+\LFY\left(\delta + \|v\|_{L^p(\Omega;L^1(0,\delta;Y))}\right)\\
    &\phantom{\leq }+B_{p}\left(\|g\|_{L^p(\Omega;L^2(0,\delta;\gHY))}+\LGY \left(\sqrt{\delta}+\|v\|_{L^p(\Omega;L^2(0,\delta;Y))}\right)\right)\\
    &\leq \|f\|_{L^p(\Omega;L^1(0,\delta;Y))}+B_{p}\|g\|_{L^p(\Omega;L^2(0,\delta;\gHY))}\\
    &\phantom{\leq }+\big(\LFY+B_{p}\LGY\big) \left(1+\|v\|_{Z^2}\right).
\end{align*}
Therefore, \eqref{eq:Gammav} follows. Now \eqref{eq:Gammav} implies
\begin{align*}
    \|\Gamma(v)\|_{Z^2} &\leq \delta^{1/2} \|\Gamma(v)\|_{L^p(\Omega;C([0,\delta];Y))}\\
    &\leq \theta (1+\|u_0\|_{L^p(\Omega;Y)}+\|f\|_{L^p(\Omega;L^1(0,\delta;Y))}+\|g\|_{L^p(\Omega;L^2(0,\delta;\gHY))}+\|v\|_{Z^2}),
\end{align*}
where $\theta = \delta^{1/2}\max\{1,B_{p},\LFY  + B_{p} \LGY \}$. Choosing $\delta\in (0,T_0]$ such that $\theta\leq \frac12$, iteratively we obtain that for $n\geq 1$,
\begin{align*}
    \|U_{n}\|_{Z^2} & \leq \theta (1+\|u_0\|_{L^p(\Omega;Y)} + \|f\|_{L^p(\Omega;L^1(0,\delta;Y))}+\|g\|_{L^p(\Omega;L^2(0,\delta;\gHY))}) +\theta\|U_{n-1}\|_{Z^2}\\
    & \leq \theta (1+\|u_0\|_{L^p(\Omega;Y)}+\|f\|_{L^p(\Omega;L^1(0,\delta;Y))}+\|g\|_{L^p(\Omega;L^2(0,\delta;\gHY))})\\
    &\phantom{\leq }+\theta^2 (1+\|u_0\|_{L^p(\Omega;Y)}+\|f\|_{L^p(\Omega;L^1(0,\delta;Y))}+\|g\|_{L^p(\Omega;L^2(0,\delta;\gHY))}+\|U_{n-2}\|_{Z^2})\\
    & \leq \ldots \leq \sum_{j=1}^n \theta^j (1+\|u_0\|_{L^p(\Omega;Y)}
    +\|f\|_{L^p(\Omega;L^1(0,\delta;Y))}+\|g\|_{L^p(\Omega;L^2(0,\delta;\gHY))}) + \theta^n \|U_0\|_{Z^2}\\
    &\leq 1+\|f\|_{L^p(\Omega;L^1(0,\delta;Y))}+\|g\|_{L^p(\Omega;L^2(0,\delta;\gHY))}
    +2\|u_0\|_{L^p(\Omega;Y)}.
\end{align*}
In conclusion, $(U_n)_{n\in \N}$ is bounded in $Z^2$. By reflexivity of $Y$, and thus of $Z^2$ (see \cite[Corollary 1.3.22]{AnalysisBanachSpacesI}), there is a subsequence $(U_{n_j})_{j\in\N}$ and $V\in Z^2$ such that $U_{n_j}\to V$ weakly in $Z^2$ and
\begin{align}
\label{eq:estVu0}
    \|V\|_{Z^2}\leq 1+\|f\|_{L^p(\Omega;L^1(0,\delta;Y))}+\|g\|_{L^p(\Omega;L^2(0,\delta;\gHY))}+2\|u_0\|_{L^p(\Omega;Y)}.
\end{align}
Since $U_{n}\to U$ in $L^p(\Omega;C([0,\delta];X))$, it follows that $V = U$. Since $U = \Gamma(U)$, \eqref{eq:Gammav} and \eqref{eq:estVu0} give that $U$ is in $L^p(\Omega;C([0,\delta];Y))$.
The same argument can be applied on $[j\delta, (j+1)\delta]$ using the initial value $U(j\delta)\in L^p(\Omega;Y)$ for $j=1,2, \ldots$ to obtain the statement on $[0,T]$.

The final a priori estimate follows as in Theorem \ref{thm:wellposed}, where we note that the Lipschitz conditions on $F$ and $G$ were not used in the estimate.
\end{proof}

\begin{remark}\label{rem:Lipschitz-growthY}

In applications, one often takes $X = L^2(O)$ and $Y = H^1(O)$ with $O\subseteq \R^d$, and $F$ is a Nemtskij operator for a given nonlinearity $\phi:\R\to \R$, i.e. $F(x)(\xi) = \phi(x(\xi))$ for $x\in L^2(O)$ and $\xi\in O$. Lipschitz continuity of such mappings holds for $F$ seen as a mapping from $X$ to $X$ if $\phi$ is Lipschitz. Also, linear growth holds for $F$ as a mapping from $Y$ into $Y$ if $\phi$ is Lipschitz. A less trivial fact is that $F$ is continuous from $Y$ into $Y$ (see \cite[Proposition 2.6.4]{TaylorBook}), but nothing more can be expected. For instance, Lipschitz continuity of $F:Y\to Y$ would require the estimate
\[\|\phi'(x) x'  - \phi'(y) y'\|_{L^2(O)}\leq C\|x  -y\|_{H^1(O)}.\]
The latter is true if and only if
$\|(\phi'(x) - \phi'(y))x'\|_{L^2(O)}\leq \tilde{C}\|x  -y\|_{H^1(O)}$. This cannot be expected even if $\phi\in C^\infty(\R^d)$ with bounded derivatives. Indeed, a product of $x-y$ and $x'$ needs to be estimated, but this cannot be done in terms of $\|x  -y\|_{H^1(O)}$.
Similarly, problems would occur for $Y = H^{\alpha}(O)$ for other values of $\alpha>0$. For a detailed exposition which estimates can be expected for $\phi(x) - \phi(y)$, the reader is referred to \cite[Section 2.7]{TaylorBook}.

\end{remark}

\section{Stability}
\label{sec:stability}

Before analysing the convergence of temporal approximations to solutions of the stochastic evolution equation \eqref{eq:stEvolEqnFG_WP} with multiplicative noise,
the question of stability of time discretisation schemes arises. We aim to prove the stability of contractive time discretisation schemes under linear growth assumptions on $F$ and $G$, and contractivity conditions on the scheme $R$. We formulate the result for mappings on $X$, but they will also be applied on $Y$ later on.

Let $R_k: X \to X$ be a contractive time discretisation scheme with time step $k>0$ on a uniform grid $\{t_j=jk:~j=0,\ldots, N_k\}\subseteq [0,T]$ with $T=t_{N_k}>0$ and $N_k=\frac{T}{k}\in \N$. We consider the temporal approximations of the mild solution to \eqref{eq:stEvolEqnFG_WP}
given by $U^0 \ce u_0$ and
\begin{align}
\label{eq:defUjStab}
    U^j &\ce R_k U^{j-1} + k R_k F(t_{j-1},U^{j-1})+ R_k G(t_{j-1},U^{j-1})\Delta W_j
\end{align}
with Wiener increments $\Delta W_j \ce W_H(t_j)-W_H(t_{j-1})$ (see \eqref{eq:convradonW})  for $1 \le j \le N_k$. The above definition of $U^j$ can be reformulated as the discrete variation-of-constants formula
\begin{equation}
\label{eq:VoCUStab}
    U^j = R_k^ju_0+k\sum_{i=0}^{j-1}R_k^{j-i}F(t_i,U^i)+\sum_{i=0}^{j-1} R_k^{j-i} G(t_i,U^i)\Delta W_{i+1}
\end{equation}
for $j=0,\ldots,N_k$.

\begin{proposition}[Stability]
\label{prop:stab}
    Let $X$ be a Hilbert space, $p \in [2, \infty)$ and $u_0 \in L_{\calF_0}^p(\Omega;X)$. Suppose that $F:\Omega \times [0,T] \times X \to X$, $G: \Omega \times [0,T] \times X \to \gHX$ are strongly $\calP\otimes \mathcal{B}(X)$-measurable, where $F = \tilde{F} +f$ and $G = \tilde{G}+g$, $f\in L^p_\calP(\Omega;C([0,T];X))$, $g\in L_\calP^p(\Omega;C([0,T];\gHX))$,  and there are $\LFX,\LGX \ge 0$ such that for all $\omega \in\Omega$, $t \in [0,T]$ and $x \in X$,
    \begin{equation*}
        \|\tilde{F}(\omega,t,x)\|_X \le \LFX(1+\|x\|_X),~
        \|\tilde{G}(\omega,t,x)\|_\gHX \le \LGX (1+\|x\|_X).
    \end{equation*}
    Let $(R_{k})_{k>0}$ be a contractive time discretisation scheme and $N_k \geq 2$.
    Then the temporal approximations $(U^j)_{j=0,\ldots,N_k}$ obtained via \eqref{eq:defUjStab} are stable in the sense of
        \[1+\left\| \max_{0 \le j \le N_k} \|U^j\|\right\|_p \le \Cstab c_{u_0, f, g,T},\]
    where $\Cstab \ce (1+C^2T)^{1/2}e^{(1+C^2T)/2}$ with $C\ce\LFX T^{1/2}+B_p\LGX$,
    \[c_{u_0, f, g,T}\ce  1+\|u_0\|_{L^p(\Omega;X)} + \|f\|_{L^p(\Omega;C([0,T];X))} T + \|g\|_{L^p(\Omega;C([0,T];\gHX))} B_p T^{1/2},\]
    and $B_{p}$ is the constant from Theorem \ref{thm:maxIneqQuasiContractive}.
\end{proposition}
Examples for contractive schemes include the exponential Euler, the implicit Euler, and the Crank--Nicolson method, as well as $A$-stable higher-order implicit Runge-Kutta methods such as Radau methods, BDF(2), Lobatto IIA, IIB, and IIC (see Proposition \ref{prop:functionalcalculus}).

The exponential dependence in Proposition \ref{prop:stab} comes from an application of Gronwall's inequality. Therefore, to make the result suitable for numerical applications, some optimization of the constants was necessary. In the special case that $L_{F,X} = L_{G,X}=T=1$, and $p=2$ one can check that $\Cstab=\sqrt{10} e^{5} \leq 470$ which seems a reasonable constant for error estimates in applications. Later on, we will also apply Proposition \ref{prop:stab} in case the space $X$ is replaced by $Y$ in the setting of Section \ref{sec:wellposed}. 

\begin{proof}
    Let $\vp_N \ce 1+\| \max_{0 \le j \le N} \|U^j\|\|_p$ and $N \in \{0,\ldots,N_k\}$. Then the variation-of-constants formula \eqref{eq:VoCUStab} and contractivity of $R_k$ allow us to bound
    \begin{align}
    \label{eq:boundStabilityProof}
        \vp_N &\le 1+ \|u_0\|_{L^p(\Omega;X)} + k \sum_{i=0}^{N-1} \left\|\max_{0 \le j \le i}\|F(t_j,U^j)\|\right\|_p\nonumber\\
        &\phantom{\le }+ \left\|\max_{0 \le j \le N} \left\|\sum_{i=0}^{j-1} R_k^{j-i}G(t_i,U^i)\Delta W_{i+1} \right\|\right\|_p.
    \end{align}
    Invoking linear growth of $\tilde{F}$ and boundedness of $f$ for the third term, we obtain the bound
    \begin{align}
    \label{eq:stabProof2ndTerm}
        k\sum_{i=0}^{N-1} \left\|\max_{0 \le j \le i}\|F(t_j,U^j)\|\right\|_p
        &\le k\sum_{i=0}^{N-1}\bigg\|\max_{0\le j \le i}\left(\LFX\left(1+\|U^j\| \right) + \|f(t_j)\|\right)\bigg\|_p\nonumber\\
        &\le k\sum_{i=0}^{N-1}\bigg(\LFX\bigg(1+\bigg\|\max_{0\le j \le i}\|U^j\| \bigg\|_p  \bigg) + \|f\|_{L^p(\Omega;C([0,T];X))}\bigg)\nonumber\\
        &= C_{1,f} t_N  + \LFX k \sum_{i=0}^{N-1} \varphi_i\leq  C_{1,f} t_N  + \LFX t_N^{1/2} \bigg(k \sum_{i=0}^{N-1} \varphi_i^2\bigg)^{1/2},
    \end{align}
    where we have set $C_{1,f} \ce \|f\|_{L^p(\Omega;C([0,T];X))}$, and used the Cauchy--Schwarz inequality and $N k = t_N$ in the last line. It remains to bound the last term in \eqref{eq:boundStabilityProof}.

    Since $R_k$ is a contraction, by the Sz.-Nagy dilation theorem \cite[Theorem I.4.2]{Nagybook} we can find a Hilbert space $\wt{X}$, a contractive injection $Q:X\to \wt{X}$, a contractive projection $P:\wt{X}\to X$, and a unitary $\wt{R}_k$ on $\wt{X}$ such that
    \[R_k^i = P \wt{R}_k^i Q \ \ \text{for all $i\geq 0$}.\]
    Let $G^k(s) \ce G(t_i,U^i)$ and
    $S^k(s) \ce \wt{R}_k^{-i} $ for $s \in [t_i,t_{i+1}), 0 \le i \le N_k-1$. Then it follows from Theorem \ref{thm:maxIneqQuasiContractive} that
    \begin{align}
    \label{eq:stabProof3rdTerm}
     \bigg\|\max_{0 \le j \le N} \bigg\| \sum_{i=0}^{j-1} R_k^{j-i}G(t_i,U^i) \Delta W_{i+1} \bigg\|\bigg\|_p  &= \bigg\|\max_{0 \le j \le N} \bigg\|\sum_{i=0}^{j-1} \wt{R}_k^{j-i} Q G(t_i,U^i)\Delta W_{i+1} \bigg\|\bigg\|_p
       \nonumber\\ &= \bigg\|\max_{0 \le j \le N} \bigg\|\sum_{i=0}^{j-1} \wt{R}_k^{-i} Q G(t_i,U^i)\Delta W_{i+1} \bigg\|\bigg\|_p
    \nonumber\\ & \leq \bigg\|\sup_{t\in [0,t_N]}\Big\|\int_0^{t}S^k(s)Q G^k(s) \dWHs \Big\|\bigg\|_p
       \nonumber\\ & \leq B_{p} \bigg\|\Big(\int_0^{t_N} \|G^k(s)\|_{\gHX}^2 ds \Big)^{1/2} \bigg\|_p
       \nonumber\\ & \leq B_{p} \left( k\sum_{i=0}^{N-1} \left\|\|G(t_{i}, U^i)\|_{\gHX} \right\|_p^2 \right)^{1/2}
        \nonumber\\ & \leq B_{p} \LGX\left( k\sum_{i=0}^{N-1} \varphi_i^2 \right)^{1/2}
       + C_{2,g} t_N^{1/2},
       \end{align}
       where we have set $C_{2,g} \ce B_p\|g\|_{L^p(\Omega;C([0,T];\gHX))}$.

     Inserting \eqref{eq:stabProof2ndTerm} and \eqref{eq:stabProof3rdTerm} in \eqref{eq:boundStabilityProof} gives the bound
    \begin{align*}
        \vp_N \le 1+\|u_0\|_{L^p(\Omega;X)} + C_{1,f} t_N + C_{2,g} t_N^{1/2}
        + (\LFX t_N^{1/2}+B_p\LGX) \left( k\sum_{i=0}^{N-1} \vp_i^2\right)^{1/2}.
    \end{align*}
    Setting $C\ce \LFX t_N^{1/2}+B_p\LGX$ and $c_{u_0, f, g,t_N}\ce 1+\|u_0\|_{L^p(\Omega;X)} + C_{1,f} t_N + C_{2,g} t_N^{1/2}$, we obtain from the discrete version of Gronwall's Lemma \ref{lem:KruseGronwall} that
    \begin{align*}
        \vp_N \le c_{u_0, f, g} (1+C^2kN)^{1/2} \ee^{(1+C^2kN)/2}.
    \end{align*}
    This implies the desired statement for $N=N_k$ noting that $t_{N_k}=kN_k=T$.
\end{proof}

\section{Convergence rates for multiplicative noise}
\label{sec:rateOfConvergenceMultNoise}

Our aim is to prove rates of convergence of contractive time discretisation schemes for nonlinear stochastic evolution equations of the form
\begin{equation}
\label{eq:StEvolEqnFG}
    \rmd U = (AU + F(t,U))\,\rmd t + G(t,U)\,\rmd W_H(t),~~U(0)=u_0 \in L^p(\Omega;X)
\end{equation}
with $t \in [0,T]$ on a Hilbert space $X$ with norm $\|\cdot\|$, where $W_H$ is an $H$-cylindrical Brownian motion for some Hilbert space $H$ and $p \in [2,\infty)$. The operator $A$ is assumed to generate a contractive $C_0$-semigroup $(S(t))_{t\ge 0}$ on $X$ and $F,G$ are assumed to be progressively measurable, of linear growth and globally Lipschitz as detailed in Assumption \ref{ass:FG_WP}. Hence, we have the unique mild solution given by a fixed point of
\begin{equation}
\label{eq:mildSoldefU}
    U(t) = S(t)u_0 + \int_0^t S(t-s)F(s,U(s))\,\rmd s + \int_0^t S(t-s)G(s,U(s))\,\rmd W_H(s)
\end{equation}
for $t \in [0,T]$, see Section \ref{sec:wellposed}.

To obtain convergence rates for temporal discretisations of the mild solution, we assume additional structure of the nonlinearity $F$ and the noise $G$.
Let $Y$ be another Hilbert space such that $Y \hra X$ and the semigroup $(S(t))_{t \ge 0}$ is also contractive on $Y$. We will assume $F$ and $G$ map $Y$ into $Y$ and enjoy linear growth conditions as on $X$ also on $Y$. Note that Lipschitz continuity is not assumed on $Y$ contrary to $X$. This additional structure resembling the famous Kato setting \cite{Kato75}, which was briefly mentioned in the introduction, allows for convergence rates of temporal discretisations for a large class of schemes introduced in Subsection \ref{subsec:mainThmMultNoise}. The quantitative error estimate in Theorem \ref{thm:convergenceRate} is the main result of this paper, stating that the additional structure suffices to obtain the order of the scheme as the convergence rate of the temporal approximations up to a logarithmic correction factor for sufficiently regular initial data. For the \emph{exponential Euler method}, the logarithmic correction factor can be omitted, as illustrated in Subsection \ref{subsec:expEulerMultNoise}. The main error estimate of Theorem \ref{thm:convergenceRate} is extended to the full time interval $[0,T]$ in Subsection \ref{subsec:fullTimeInterval} As an application, we revisit the Schrödinger equation, now with a multiplicative potential, in Subsection \ref{subsec:SchroedingerMult}, including its numerical simulation in Subsection \ref{subsec:numerical}, and consider the stochastic Maxwell's equations in Subsection \ref{subsec:Maxwell}.

\subsection{General contractive time discretisation schemes}
\label{subsec:mainThmMultNoise}

We now detail the assumptions on the structure of $F$ and $G$ on $Y$. Note that the assumption also implies that the conditions of Theorems \ref{thm:wellposed} and \ref{thm:wellposedY} hold.

\begin{assumption}\label{ass:FG_structured}
    Let $X,Y$ be Hilbert spaces such that $Y \hra X$ continuously, and let $p \in [2,\infty)$.
    Let $F:\Omega \times [0,T] \times X \to X, F(\omega, t,x) = \tilde{F}(\omega,t,x) + f(\omega,t)$ and $G:\Omega \times [0,T] \times X \to \gHX, G(\omega, t,x) = \tilde{G}(\omega,t,x) + g(\omega,t)$ be strongly $\calP\otimes \calB(X)$-measurable, and such that $\tilde{F}(\cdot,\cdot,0) = 0$ and $\tilde{G}(\cdot,\cdot,0) = 0$, and suppose
    \begin{enumerate}[label=(\alph*)]
        \item\label{assCond:LipschitzFG} \emph{(global Lipschitz continuity on $X$)} there exist constants $\CFX, \CGX \ge 0$ such that for all $\omega \in \Omega, t \in [0,T]$, and $x,y\in X$, it holds that
        \begin{equation*}
            \|\tilde{F}(\omega,t,x)-\tilde{F}(\omega,t,y)\| \le \CFX\|x-y\|,~ \|\tilde{G}(\omega,t,x)-\tilde{G}(\omega,t,y)\|_\gHX \le \CGX\|x-y\|,
        \end{equation*}
        \item \label{assCond:HoldercontinuityFG}\emph{(Hölder continuity with values in $X$)} for some $\alpha \in (0,1]$,
        \begin{align*}
            C_{\alpha,F}\ce \sup_{\omega\in \Omega, x \in X} [F(\omega,\cdot,x)]_\alpha <\infty,~C_{\alpha,G}\ce \sup_{\omega\in \Omega, x \in X} [G(\omega,\cdot,x)]_\alpha <\infty,
        \end{align*}
        \item \label{item:FG_rate_Yinvariance} \emph{($Y$-invariance)} $F:\Omega \times [0,T] \times Y \to Y$ and $G:\Omega \times [0,T] \times Y \to \gHY$ are strongly $\calP\otimes \calB(Y)$-measurable, $f \in L^p_\calP(\Omega; C([0,T];Y))$, and $g \in L^p_\calP(\Omega; C([0,T];\gHY))$,
        \item \emph{(linear growth on $Y$)} there exist constants $\LFY, \LGY \ge 0$ such that for all $\omega \in \Omega, t \in [0,T]$, and $x\in Y$, it holds that
        \begin{equation*}
            \|\tilde{F}(\omega,t,x)\|_Y \le \LFY(1+\|x\|_Y),~ \|\tilde{G}(\omega,t,x)\|_\gHY \le \LGY(1+\|x\|_Y).
        \end{equation*}
    \end{enumerate}
\end{assumption}

Condition \ref{assCond:HoldercontinuityFG} can be weakened to
the existence of some $\alpha \in (0,1]$ such that
\begin{equation*}
    \sup_{x \in X} \sup_{0 \le s \le t \le T} \frac{F(\cdot,t,x)-F(\cdot,s,x)}{(t-s)^\alpha} \in L^p(\Omega)
\end{equation*}
 and likewise for $G$, i.e., pathwise Hölder continuity uniformly in $x \in X$ is sufficient together with existence of $p$-th moments of the Hölder seminorms. Assumption \ref{ass:FG_structured} implies that \eqref{eq:StEvolEqnFG} has a unique mild solution.

To bound the error arising from time discretisation of the mild solution, moment bounds of differences of the mild solution at different time points as in the following lemma are required.  As a shorthand notation in accordance with \eqref{eq:shorthandOnlyg}, let
\begin{equation}
\label{eq:shorthandNormfg}
    \|f\|_{p,q,Z} \ce \|f\|_{L^p(\Omega;L^q(0,T;Z))},~ \nn g\nn_{p,q,Z} \ce \|g\|_{L^p(\Omega;L^q(0,T;\calL_2(H,Z)))}
\end{equation}
for Hilbert spaces $Z$, $p \in [2,\infty)$, and $q \in [1,\infty]$. We further introduce the constants
\begin{align}
\label{eq:defCu0fgZ}
    C_{u_0,f,g,Z} &\ce 1+\Cbdd^Z(1+\|u_0\|_{L^p(\Omega;Z)}+\|f\|_{p,1,Z}+\nn g\nn_{p,2,Z})
\end{align}
for $Z \in \{X,Y\}$ with $\Cbdd^X$ and $\Cbdd^Y$ as in Theorems \ref{thm:wellposed} and \ref{thm:wellposedY}, respectively. Then the estimate
\begin{equation}
\label{eq:WPboundCu0fgZ}
    1+\bigg\|\sup_{r \in [0,T]} \|U(r)\|_Z\bigg\|_p \le C_{u_0,f,g,Z} < \infty
\end{equation}
holds for $Z\in \{X, Y\}$.

\begin{lemma}
\label{lem:mildsolest}
Suppose that Assumption \ref{ass:FG_structured} holds for some $\alpha \in (0,1]$ and $p \in [2,\infty)$. Let $A$ be the generator of a $C_0$-contraction semigroup $(S(t))_{t \ge 0}$ on both $X$ and $Y$.
Suppose that $Y \hra D_A(\alpha,\infty)$ continuously if $\alpha \in (0,1)$ or $Y \hra D(A)$ continuously if $\alpha=1$.
Let $u_0 \in L_{\calF_0}^p(\Omega;Y)$. Then for all $0 \le s \le t \le T$ the mild solution $U$ of \eqref{eq:StEvolEqnFG} satisfies
    \begin{align*}
        &\left(\E\|U(t)-U(s)\|^p\right)^{1/p}  \le L_1 (t-s) + L_2 (t-s)^{1/2} +L_3 (t-s)^{\alpha}
    \end{align*}
    with constants $L_1\ce\CFX  \CufgX+\|f\|_{p,\infty,X}$, $L_2 \ce B_{p} (\CGX  \CufgX+
        \nn g\nn_{p,\infty,X})$, and
    \begin{equation*}
        L_3 \ce 2\CY \Big[ \|u_0\|_{L^p(\Omega;Y)} + T\LFY\CufgY + \|f\|_{p,1,Y}+ B_{p} \big(T^{1/2}\LGY \CufgY
        + \nn g\nn_{p,2,Y}\big) \Big],
    \end{equation*}
    where $\CufgX$ and $\CufgY$ are as defined in \eqref{eq:defCu0fgZ}, $\CY$ denotes the embedding constant of $Y$ into $D_A(\alpha,\infty)$ or $D(A)$, and $B_p$ is the constant from Theorem \ref{thm:maxIneqQuasiContractive}.
\end{lemma}
\begin{proof}
    Since the conditions of Theorems \ref{thm:wellposed} and \ref{thm:wellposedY} are met, $U$ is pathwise continuous on $X$. By Theorem \ref{thm:wellposedY}, the pathwise continuity of $U$ follows on $Y$ as well. Moreover, the bound \eqref{eq:WPboundCu0fgZ} holds.

    Fix $t,s \in [0,T]$ with $s \le t$. From the mild solution formula \eqref{eq:mildSoldefU}, we deduce that
    \begin{align*}
        &\left(\E\|U(t)-U(s)\|^p\right)^{1/p} \le \left\|[S(t)-S(s)]u_0\right\|_{L^p(\Omega;X)}\\
        &+ \Big\|\int_0^{s} \|[S(t-r)-S(s-r)]F(r,U(r))\|\;\rmd r\Big\|_p+ \Big\|\int_{s}^t \|S(t-r)F(r,U(r))\|\;\rmd r\Big\|_p\\
        & + \Big\|\int_0^{s} [S(t-r)-S(s-r)]G(r,U(r))\;\rmd W_H(r)\Big\|_{L^p(\Omega;X)}\\
        & + \Big\|\int_{s}^t S(t-r)G(r,U(r))\;\rmd W_H(r)\Big\|_{L^p(\Omega;X)} \ec E_1+E_2+E_3+E_4+E_5,
    \end{align*}
    where $E_\ell = E_\ell(t,s)$ for $1 \le \ell \le 5$. We proceed to bound these five expressions individually.
    By the semigroup bound \eqref{eq:interpolationSgDifferenceGeneral},
    \begin{align*}
        E_1 & \le \|S(t)-S(s)\|_{\calL(Y,X)} \|u_0\|_{L^p(\Omega;Y)} \le 2\CY (t-s)^{\alpha} \|u_0\|_{L^p(\Omega;Y)}.
    \end{align*}
    Using \eqref{eq:WPboundCu0fgZ}  and \eqref{eq:interpolationSgDifferenceGeneral} as well as linear growth of $\tilde{F}$ on $Y$ and $f \in L^p(\Omega;L^1(0,T;Y))$, we obtain
    \begin{align*}
        E_2 &\le 2\CY \Big\| \int_0^s [(t-r)-(s-r)]^\alpha \|F(r,U(r))\|_Y \,\rmd r \Big\|_p\\
        &\le 2\CY (t-s)^{\alpha}\bigg(s \LFY\bigg\| \sup_{r \in [0,T]}(1+\|U(r)\|_Y) \bigg\|_p + \Big\| \int_{0}^s\|f(r)\|_Y \,\rmd r\Big\|_p\bigg)\\
        &\le 2\CY \big(T\LFY  \CufgY +\|f\|_{p,1,Y}\big)  (t-s)^{\alpha}.
    \end{align*}
    Analogously,
    \begin{align*}
        E_3 &\le (\CFX  \CufgX+\|f\|_{p,\infty,X}) (t-s)
    \end{align*}
    is obtained by contractivity of the semigroup, linear growth of $F$ on $X$ and boundedness of the solution. For the terms involving a stochastic integral, we apply Theorem \ref{thm:maxIneqQuasiContractive}. Additionally making use of the bound \eqref{eq:interpolationSgDifferenceGeneral} for semigroup differences, splitting the integral as in $E_2$, and using linear growth of $\tilde{G}$, \eqref{eq:WPboundCu0fgZ}, as well as $g\in L^p(\Omega;L^2(0,T;\gHY)))$ results in
    \begin{align*}
        E_4&\le B_{p} \bigg(\E \left( \int_0^{s} \|[S(t-r)-S(s-r)]G(r,U(r))\|_\gHX^2\;\rmd r\right)^{p/2}\bigg)^{1/p}\\
        &\le 2B_{p} \CY \big(T^{1/2}\LGY  \CufgY+\nn g\nn_{p,2,Y}\big)  (t-s)^{\alpha}.
    \end{align*}
    For the last term, the contractivity of the semigroup and linear growth of $G$ yield
    \begin{align*}
        E_5 &\le B_{p}  \bigg(\E \left( \int_{s}^t \|S(t-r)G(r,U(r))\|_\gHX^2\;\rmd r\right)^{p/2}\bigg)^{1/p} \\
        &\le B_{p} (\CGX  \CufgX+\nn g\nn_{p,\infty,X}) (t-s)^{1/2}.
    \end{align*}
    In conclusion from the five individual bounds, we obtain the statement of the lemma
    \begin{align*}
        \left(\E\|U(t)-U(s)\|^p\right)^{1/p}  &\le (\CFX  \CufgX+\|f\|_{p,\infty,X}) (t-s)\\
        &\phantom{\le }+ B_{p} \big(\CGX  \CufgX+\nn g\nn_{p,\infty,X}\big) (t-s)^{1/2}\\
        &\phantom{\le }+2\CY \Big[ \|u_0\|_{L^p(\Omega;Y)} + T\LFY\CufgY + \|f\|_{p,1,Y} \\
        &\phantom{\le }+ B_{p} \big(T^{1/2}\LGY \CufgY+\nn g\nn_{p,2,Y}\big) \Big] (t-s)^{\alpha}. \qedhere
    \end{align*}
\end{proof}
\begin{remark}
\label{rem:lemmafgRegularity}
    Suppose that $\alpha \in (0,\frac12]$. Lemma \ref{lem:mildsolest} implies $\alpha$-H\"older continuity of $U$ in $p$-th moment. The latter remains true if the pathwise continuity of $f$ and $g$ with values in $Y$ from Assumption \ref{ass:FG_structured}\ref{item:FG_rate_Yinvariance} are relaxed to $\|f\|_{p,1,Y}<\infty$ and $\nn g \nn_{p,2,Y}<\infty$. Performing an additional Hölder argument for $E_3$ and $E_5$, the pathwise continuity assumption with values in $X$ can be relaxed to $\|f\|_{p,\frac{1}{1-\alpha},X}<\infty$ and $\nn g \nn_{p,\frac{2}{1-2\alpha},X}<\infty$, where we use the convention $\frac{1}{0}=\infty$. Although the lemma could be improved, for our purposes the above version is enough since even pathwise continuity with values in $Y$ is required in Theorem \ref{thm:convergenceRate}.
\end{remark}

For time discretisation, we employ a contractive time discretisation scheme $R: [0,\infty) \to \calL(X)$ with time step $k>0$ on a uniform grid $\{t_j=jk:~j=0,\ldots, N_k\}\subseteq [0,T]$ with final time $T=t_{N_k}>0$ and $N_k=\frac{T}{k}\in \N$ being the number of time steps. As in the previous section, the discrete solution is given by $U^0 \ce u_0$ and
\begin{align}
\label{eq:defUj}
    U^j &\ce R_k U^{j-1} + k R_k F(t_{j-1},U^{j-1})+ R_k G(t_{j-1},U^{j-1})\Delta W_j \\
\label{eq:VoCU}
    &= R_k^j u_0 + k \sum_{i=0}^{j-1} R_k^{j-i}F(t_i,U^i) + \sum_{i=0}^{j-1} R_k^{j-i}G(t_i,U^i)\Delta W_{i+1}
\end{align}
for $j=1,\ldots,N_k$ with Wiener increments $\Delta W_j \ce W_H(t_j)-W_H(t_{j-1})$.

We recall from Definition \ref{def:orderScheme} that $R$ \textit{approximates $S$ to order $\alpha>0$ on $Y$} or, equivalently, \textit{$R$ converges of order $\alpha$ on $Y$} if there is a constant $C_\alpha \ge 0$ such that for all $u \in Y$
\begin{equation*}
    \|(S(t_j)-R_k^j)u\| \le C_\alpha k^\alpha\|u\|_Y.
\end{equation*}

Under the conditions of Assumption \ref{ass:FG_structured} we conclude from Proposition \ref{prop:stab} and the remark thereafter that $R$ is stable not only on $X$ but also on $Y$ provided that $u_0 \in L_{\calF_0}^p(\Omega;Y)$ and both $S$ and $R$ are contractive on both $X$ and $Y$. Thus,
\begin{equation}
    \label{eq:defKu0fgY}
    1+ \left\lVert \max\limits_{0 \le j \le N_k}\|U^j \|_Y\right\rVert_p \le \KufgY,
\end{equation}
where $\KufgY\ce \Cstab c_{u_0, f, g,T}$ with constants $\Cstab, c_{u_0, f, g,T}$ as in Proposition \ref{prop:stab} applied on $Y$ instead of $X$. Furthermore, we recall the shorthand notation $\|f\|_{p,\infty,Y}$ and $\nn g \nn_{p,\infty,Y}$ from \eqref{eq:shorthandNormfg}.

We can now state and prove the main result of this paper.

\begin{theorem}
\label{thm:convergenceRate}
    Suppose that Assumption \ref{ass:FG_structured} holds for some $\alpha \in (0,1]$ and $p \in [2,\infty)$. Let $A$ be the generator of a $C_0$-contraction semigroup $(S(t))_{t \ge 0}$ on both $X$ and $Y$.
    Let $(R_{k})_{k>0}$ be a time discretisation scheme which is contractive on $X$ and $Y$. Assume $R$ approximates $S$ to order $\alpha$ on $Y$. Suppose that $Y \hra D_A(\alpha,\infty)$ continuously if $\alpha \in (0,1)$ or $Y \hra D(A)$ continuously if $\alpha=1$. Let $u_0 \in L_{\calF_0}^p(\Omega;Y)$. Denote by $U$ the mild solution of \eqref{eq:StEvolEqnFG} and by $(U^j)_{j=0,\ldots,N_k}$ the temporal approximations as defined in \eqref{eq:defUj}. Then for $N_k \ge 2$
    \begin{equation*}
        \bigg\lVert\max_{0 \le j \le N_k}\|U(t_j)-U^j\|\bigg\rVert_p
        \le C_\ee\Big(C_1 k+ C_2k^{1/2}
        +\big(C_3 +C_4\sqrt{\max\{\log (T/k),p\}}\big)k^{\alpha}\Big)
    \end{equation*}
    with constants $C_\ee \ce (1+C^2T)^{1/2}\exp((1+C^2T)/2)$, $C \ce \CFX\sqrt{T}+B_p\CGX$, $C_1 \ce L_1(\frac{\CFX}{2}T^2+  B_p \CGX \sqrt{T})$, $C_2 \ce L_2(\frac{2}{3} \CFX T+(\frac{3}{2})^{1/2} B_p \CGX \sqrt{T} )$, $C_4 \ce C_{3,\log}\sqrt{T}$, and
    \begin{align*}
        C_3&\ce C_\alpha\|u_0\|_{L^p(\Omega;Y)}+C_{2,\alpha} T +C_{3,\alpha}\sqrt{T},\\
        C_{2,\alpha} &\ce \frac{\CFX L_3+C_{\alpha,F}}{\alpha+1}+ \big(\LFY \KufgY+\|f\|_{p,\infty,Y}\big)\left(\frac{2\CY}{\alpha+1}+C_\alpha\right),\\
        C_{3,\alpha} &\ce \frac{B_p}{\sqrt{2\alpha +1}} \Big(\sqrt{3}\CGX L_3 + C_{\alpha,G}+2 \CY\big(\LGY \KufgY+\nn g\nn_{p,\infty,Y}\big)\Big),\\
        C_{3,\log} &\ce K C_\alpha \big(\LGY \KufgY
        +\nn g\nn_{p,\infty,Y}\big),
    \end{align*}
    where $L_1,L_2,L_3$ are as defined in Lemma \ref{lem:mildsolest}, $\KufgY$ as in \eqref{eq:defKu0fgY}, $K=4\exp(1+\frac{1}{2\mathrm{e}})$, $\CY$ denotes the embedding constant of $Y$ into $D_A(\alpha,\infty)$ or $D(A)$, and $B_{p}$ is the constant from Theorem \ref{thm:maxIneqQuasiContractive}.

    In particular, the approximations $(U^j)_j$ converge at rate $\min\{\alpha,\frac{1}{2}\}$ up to a logarithmic correction factor as $k \to 0$.
\end{theorem}

This convergence result applies to schemes such as the exponential Euler, the implicit Euler, and the Crank--Nicolson method, as well as other $A$-acceptable implicit Runge-Kutta methods such as Radau  methods, BDF(2), Lobatto IIA, IIB, and IIC by virtue of Proposition \ref{prop:functionalcalculus}. If $R$ commutes with the resolvent of $A$, contractivity of $R$ and $S$ extend to fractional domain spaces and complex interpolation spaces. Hence, contractivity on $Y$ often comes together with contractivity on $X$.

The constant $C_{\ee}$ appears exponentially in the above. In the special case that $C_{F,X} = C_{G,X}=T=1$, and $p=2$, one can check that, similarly to Theorem \ref{thm:wellposedY}, this yields the numerically reasonable value $C_{\ee}=\sqrt{10} e^{5} \leq 470$.

\begin{proof}
The assumptions of Theorems \ref{thm:wellposed} and \ref{thm:wellposedY} hold, and thus the mild solution $U$ exists and the bound \eqref{eq:WPboundCu0fgZ} holds.

By definition, $U(t_j)=U^j=u_0$ for $j=0$. Let $N\in \{1,\ldots,N_k\}$.
Using \eqref{eq:VoCU}, the discretisation error can be split into three parts
\begin{align*}
    E(N) &\ce \bigg\lVert\max_{1 \le j \le N}\|U(t_j)-U^j\|\bigg\rVert_p\\
    &\le \bigg\lVert\max_{1 \le j \le N}\|(S(t_j)-R_k^j)u_0\|\bigg\rVert_p\\
    &\phantom{\le }+ \bigg\lVert\max_{1 \le j \le N}\bigg\|\int_0^{t_j}S(t_j-s)F(s,U(s))\ds-k\sum_{i=0}^{j-1}R_k^{j-i}F(t_i,U^i)\bigg\|\bigg\rVert_p\\
    &\phantom{\le }+\bigg\lVert\max_{1 \le j \le N}\bigg\|\int_0^{t_j}S(t_j-s)G(s,U(s))\dWHs - \sum_{i=0}^{j-1} R_k^{j-i}G(t_i,U^i)\Delta W_{i+1}\bigg\|\bigg\rVert_p\\
    &\ec M_1+M_2+M_3.
\end{align*}
Using convergence of $R$ of order $\alpha$ on $Y$ and the dominated convergence theorem, we obtain
\begin{align}
\label{eq:M1rate}
    M_1 & \le C_\alpha k^\alpha \|u_0\|_{L^p(\Omega;Y)}.
\end{align}

To shorten the notation for the discrete terms, we introduce the piecewise constant functions $F^k(s) \ce F(t_i,U^i)$ and $G^k(s) \ce G(t_i,U^i)$ for $s \in [t_i,t_{i+1}), 0 \le i \le N_k-1$ as well as
$S^k(s) \ce R_k^i $ for $s \in (t_{i-1},t_i], 1\le i \le N_k$. This allows us to rewrite
\begin{align*}
    M_2 &= \left\lVert\max_{1 \le j \le N}\left\|\int_0^{t_j}S(t_j-s)F(s,U(s))-S^k(t_j-s)F^k(s)\ds\right\|\right\rVert_p\\
    &\le \left\lVert \sum_{i=0}^{N-1} \int_{t_i}^{t_{i+1}} \max_{1 \le j \le N}\left\|S(t_j-s)[F(s,U(s))-F(s,U(t_i))]\right\|\ds\right\rVert_p\\
    &\phantom{\le }+\left\lVert \sum_{i=0}^{N-1} \int_{t_i}^{t_{i+1}} \max_{1 \le j \le N}\left\|S(t_j-s)[F(s,U(t_i))-F(t_i,U(t_i))]\right\|\ds\right\rVert_p\\
    &\phantom{\le }+\left\lVert \sum_{i=0}^{N-1} \int_{t_i}^{t_{i+1}} \max_{1 \le j \le N}\left\|S(t_j-s)[F(t_i,U(t_i))-F(t_i,U^i)]\right\|\ds\right\rVert_p\\
    &\phantom{\le }+ \left\lVert \int_0^{t_N}\max_{1 \le j \le N}\left\|[S(t_j-s)-S^k(t_j-s)]F^k(s)\right\|\ds\right\rVert_p\\
    &\ec M_{2,1}+M_{2,2}+M_{2,3}+M_{2,4}.
\end{align*}
Making use of Minkowski's inequality in $L^p(\Omega)$, contractivity of $(S(t))_{t \ge 0}$ and Lipschitz continuity of $\tilde{F}$, we derive the bound
\begin{equation}
\label{eq:M23rate}
    M_{2,3} \le \CFX \sum_{i=0}^{N-1}\left\lVert\int_{t_i}^{t_{i+1}} \left\|U(t_i)-U^i\right\|\ds\right\rVert_p \le \CFX  k \sum_{i=0}^{N-1} E(i)
\end{equation}
for $M_{2,3}$. Proceeding likewise for $M_{2,1}$, we obtain from Lemma \ref{lem:mildsolest} that
\begin{align}
\label{eq:M21rate}
    M_{2,1} &\le \CFX \sum_{i=0}^{N-1}\int_{t_i}^{t_{i+1}} (\E\left\|U(s)-U(t_i)\right\|^p)^{1/p}\ds\nonumber\\
    &\le \CFX  \sum_{i=0}^{N-1} \int_{t_i}^{t_{i+1}} L_1 (s-t_i)+L_2 (s-t_i)^{1/2}+L_3(s-t_i)^\alpha \ds\nonumber\\
    &\le \CFX  \sum_{i=0}^{N-1} \left(\frac{L_1}{2} k^2 + \frac{2L_2}{3} k^{3/2} + \frac{L_3}{\alpha+1} k^{\alpha+1} \right)\nonumber\\
    & = \CFX  t_N\left(\frac{L_1}{2} k + \frac{2L_2}{3} k^{1/2} + \frac{L_3}{\alpha+1} k^{\alpha} \right).
\end{align}
Analogously, uniform Hölder continuity yields
\begin{align}
\label{eq:M22rate}
    M_{2,2}&\le  \sum_{i=0}^{N-1} \int_{t_i}^{t_{i+1}} \left\|F(s,U(t_i))-F(t_i,U(t_i))\right\|_{L^p(\Omega;X)}\ds\nonumber\\
    &\le  \sum_{i=0}^{N-1} \int_{t_i}^{t_{i+1}} (s-t_i)^\alpha\ds \left\| [F(\cdot,U(t_i))]_\alpha\right\|_p\nonumber\\
    &\le  \sum_{i=0}^{N-1} \frac{k^{\alpha+1}}{\alpha+1}C_{\alpha,F} = \frac{C_{\alpha,F}t_N}{\alpha+1} k^\alpha.
\end{align}
Using the semigroup bound \eqref{eq:interpolationSgDifferenceGeneral} together with the assumed convergence rate $\alpha$ of $R$ on $Y$, the linear growth assumption and stability of $R$, we obtain
\begin{align}
\label{eq:M24rate}
    M_{2,4} &\le \left\lVert \sum_{i=0}^{N-1}\int_{t_i}^{t_{i+1}} \left\|[S(t_j-s)-S(t_j-t_i)]F(t_i,U^i)\right\|\ds\right\rVert_p\nonumber\\
    &\phantom{\le }+ \left\lVert \sum_{i=0}^{N-1}\int_{t_i}^{t_{i+1}} \left\|\left[S(t_j-t_i)-R_k^{j-i}\right]F(t_i,U^i)\right\|\ds\right\rVert_p\nonumber\\
    &\le 2\CY \sum_{i=0}^{N-1} \left\lVert \int_{t_i}^{t_{i+1}} (s-t_i)^\alpha \|F(t_i,U^i)\|_Y\ds\right\rVert_p+ C_\alpha k^\alpha\sum_{i=0}^{N-1}\left\lVert \int_{t_i}^{t_{i+1}} \left\|F(t_i,U^i)\right\|_Y\ds\right\rVert_p\nonumber\\
    &\le \left(\frac{2 \CY}{\alpha+1} +C_\alpha\right)  k^{\alpha+1} \sum_{i=0}^{N-1}\left(\LFY\left\lVert 1+\|U^i\|_Y\right\rVert_p+\|f(t_i)\|_{L^p(\Omega;Y)}\right)\nonumber\\
    &\le \left(\frac{2 \CY}{\alpha+1} +C_\alpha\right)\big(\LFY  \KufgY + \|f\|_{p,\infty,Y}\big) t_N k^\alpha.
\end{align}
In conclusion from \eqref{eq:M21rate}, \eqref{eq:M22rate}, \eqref{eq:M23rate}, and \eqref{eq:M24rate}, $M_2$ is bounded by
\begin{align}
\label{eq:M2rate}
    M_2 &\le \frac{\CFX L_1}{2}t_N k + \frac{2\CFX L_2}{3}t_N k^{1/2}
    + C_{2,\alpha}t_N k^\alpha+\CFX  k \sum_{i=0}^{N-1} E(i)\nonumber\\
&\le \frac{\CFX L_1}{2}t_N k + \frac{2\CFX L_2}{3}t_N k^{1/2}
    + C_{2,\alpha}t_N k^\alpha+\CFX  \sqrt{t_N} \bigg(k\sum_{i=0}^{N-1} E(i)^2\bigg)^{1/2},
\end{align}
where we have used the Cauchy--Schwarz inequality in the last line.

Let $\lfloor s \rfloor = \max\{t_i: 0 \le i \le N_k-1, t_i \le s\}$. The remaining term $M_3$ can be rewritten as
\begin{align*}
    M_3&= \left\lVert\max_{1 \le j \le N}\left\|\int_0^{t_j}S(t_j-s)G(s,U(s))-S^k(t_j-s)G^k(s) \dWHs\right\|\right\rVert_p\\
    &\le \left\lVert\max_{1 \le j \le N}\left\|\int_0^{t_j}S(t_j-s)[G(s,U(s))-G(s,U(\lfloor s \rfloor)] \dWHs\right\|\right\rVert_p\\
    &\phantom{\le }+\left\lVert\max_{1 \le j \le N}\left\|\int_0^{t_j}S(t_j-s)[G(s,U(\lfloor s \rfloor))-G(\lfloor s \rfloor,U(\lfloor s \rfloor)] \dWHs\right\|\right\rVert_p\\
    &\phantom{\le} +\left\lVert\max_{1 \le j \le N}\left\|\int_0^{t_j}S(t_j-s)[G(\lfloor s \rfloor, U(\lfloor s \rfloor))-G^k(s)] \dWHs\right\|\right\rVert_p\\
    &\phantom{\le }+ \left\lVert\max_{1 \le j \le N}\left\|\int_0^{t_j}[S(t_j-\lfloor s \rfloor)-S(t_j-s)]G^k(s) \dWHs\right\|\right\rVert_p\\
    &\phantom{\le }+ \left\lVert\max_{1 \le j \le N}\left\|\int_0^{t_j}[S(t_j-\lfloor s \rfloor)-S^k(t_j-s)]G^k(s) \dWHs\right\|\right\rVert_p\\
    &\ec M_{3,1}+M_{3,2}+M_{3,3}+M_{3,4}+M_{3,5}.
\end{align*}
We bound each term individually. An application of the maximal inequality Theorem \ref{thm:maxIneqQuasiContractive}, the Lipschitz continuity of $\tilde{G}$ and Lemma \ref{lem:mildsolest} result in
\begin{align}
\label{eq:M31rate}
    M_{3,1} &\le \left\lVert\sup_{t \in [0,t_N]}\left\|\int_0^t S(t-s) [G(s,U(s))-G(s,U(\lfloor s \rfloor)] \dWHs\right\|\right\rVert_p\nonumber\\
    &\le B_p \left(\E \left( \sum_{i=0}^{N-1} \int_{t_i}^{t_{i+1}} \|G(s,U(s))-G(s,U(t_i))\|_\gHX^2 \ds\right)^{p/2} \right)^{1/p}\nonumber\\
    &\le B_p\CGX \bigg(\sum_{i=0}^{N-1} \int_{t_i}^{t_{i+1}} \left( \E\|U(s)-U(t_i)\|^p\right)^{2/p} \ds\bigg)^{1/2}\nonumber\\
    &\le \sqrt{3}B_p\CGX\bigg(\sum_{i=0}^{N-1} \int_{t_i}^{t_{i+1}} L_1^2 (s-t_i)^2 + L_2^2(s-t_i) + L_3^2 (s-t_i)^{2\alpha} \ds\bigg)^{1/2}\nonumber\\
    &= \sqrt{3}B_p\CGX \sqrt{t_N} \left( \frac{L_1^2}{3}k^2+\frac{L_2^2}{2} k+\frac{L_3^2}{2\alpha+1} k^{2\alpha} \right)^{1/2}\nonumber\\
    &\le \sqrt{3}B_p\CGX \sqrt{t_N} \left( \frac{L_1}{\sqrt{3}}k+\frac{L_2}{\sqrt{2}} k^{1/2}+\frac{L_3}{\sqrt{2\alpha+1}} k^{\alpha} \right).
\end{align}
Again invoking the maximal inequality, we conclude
\begin{align}
\label{eq:M32rate}
    M_{3,2} &\le B_p\bigg( \sum_{i=0}^{N-1} \int_{t_i}^{t_{i+1}} \left\|\|G(s,U(t_i))-G(t_i,U(t_i))\|_\gHX\right\|_p^2 \ds\bigg)^{1/2}\nonumber\\
    &\le B_p \bigg(\sum_{i=0}^{N-1} \int_{t_i}^{t_{i+1}} (s-t_i)^{2\alpha} \ds \left\|[G(\cdot,U(t_i))]_\alpha\right\|_p^2\bigg)^{1/2} \le \frac{B_pC_{\alpha,G}}{\sqrt{2\alpha+1}}\sqrt{t_N} k^{\alpha}
\end{align}
from the uniform Hölder continuity of $G$.
Proceeding analogously for $M_{3,3}$ and then applying Minkowski's inequality in $L^{p/2}(\Omega)$ results in
\begin{align}
\label{eq:M33rate}
    M_{3,3} &\le \left\lVert\sup_{t \in [0,t_N]}\left\|\int_0^t S(t-s) [G(\lfloor s \rfloor,U(\lfloor s \rfloor)-G^k(s)] \dWHs\right\|\right\rVert_p\nonumber\\
    &\le B_p\CGX \left(\E \left( k\sum_{i=0}^{N-1} \|U(t_i)-U^i\|^2 \right)^{p/2} \right)^{1/p}\nonumber\\
    &= B_p\CGX k^{1/2} \left\|\sum_{l=0}^{N-1} \max_{0\le j \le l}\|U(t_j)-U^j\|^2 \right\|_{p/2}^{1/2} \nonumber\\
    &\le B_p\CGX\sqrt{k}\bigg( \sum_{l=0}^{N-1} \left\| \max_{0\le j \le l}\|U(t_j)-U^j\|^2 \right\|_{p/2}\bigg)^{1/2}\nonumber\\
    &= B_p\CGX \sqrt{k} \bigg(\sum_{l=0}^{N-1} E(l)^2\bigg)^{1/2}.
\end{align}
Since $R$ is contractive on $Y$ by assumption, the conditions of Proposition \ref{prop:stab} are fulfilled not only on $X$ but also on $Y$. Thus, we can use the estimate \eqref{eq:defKu0fgY}. Together with the maximal inequality, the semigroup difference bound \eqref{eq:interpolationSgDifferenceGeneral}, the ideal property of $\gHX$, and linear growth of $\tilde{G}$, this yields
\begin{align}
\label{eq:M34rate}
    M_{3,4}&\le\left\lVert\sup_{t \in [0,t_N]}\left\|\int_0^{t}S(t-s)\left(\sum_{i=0}^{j-1}\1_{[t_i,t_{i+1})}(s)[S(s-t_i)-I]G(t_i,U^i)\right) \dWHs\right\|\right\rVert_p \nonumber\\
    &\le B_p \left(\E \left(\int_0^{t_N} \left\|\1_{[t_i,t_{i+1})}(s)[S(s-t_i)-I]G(t_i,U^i) \right\|_\gHX^2 \ds \right)^{p/2} \right)^{1/p}\nonumber\\
    &\le 2B_p\CY \left(\E \left(\sum_{\ell=0}^{N-1}\int_{t_\ell}^{t_{\ell+1}} (s-t_\ell)^{2\alpha}\left\|G(t_\ell,U^\ell) \right\|_\gHY^2 \ds \right)^{p/2} \right)^{1/p}\nonumber\\
    &\le \frac{2B_p\CY}{\sqrt{2\alpha+1}} \sqrt{t_N} k^\alpha \left\|\max_{0 \le j \le N-1}\left\|G(t_j,U^j) \right\|_\gHY \right\|_p\nonumber\\
    &\le \frac{2B_p\CY}{\sqrt{2\alpha+1}} \big(\LGY \KufgY+\nn g\nn_{p,\infty,Y}\big) \sqrt{t_N} k^\alpha.
\end{align}
Applying Proposition \ref{prop:PropLogMainPaper} with $\Phi_s^{(j)}=\sum_{i=0}^{j-1}\1_{[t_i,t_{i+1})}(s)[S(t_j-t_i)-R_k^{j-i}]G(U^i)$ to the remaining term, we conclude that
\begin{align}
\label{eq:M35rate}
    &M_{3,5}=\bigg(\E\max_{1 \le j \le N}\bigg\|\int_0^{t_j}\sum_{i=0}^{j-1}\1_{[t_i,t_{i+1})}(s)[S(t_j-t_i)-R_k^{j-i}]G(t_i,U^i) \dWHs\bigg\|^p\bigg)^{1/p} \nonumber\\
    &\le K\sqrt{\max\{\log (N),p\}} \bigg\lVert\bigg(\sum_{\ell=0}^{N-1}k \Big(\max_{1\le j \le N} \left\|[S(t_j-t_\ell)-R_k^{j-\ell}]G(t_\ell,U^\ell)\right\|_\gHX\Big)^2\bigg)^{1/2} \bigg\rVert_p\nonumber\\
    &\le K\sqrt{\max\{\log (N),p\}} \bigg(\E\bigg(\sum_{l=0}^{N-1}k \left( C_\alpha k^\alpha\left\|G(t_\ell,U^\ell)\right\|_{\gHY}\right)^2\bigg)^{p/2} \bigg)^{1/p}\nonumber\\
    &\le KC_\alpha \sqrt{t_N} \sqrt{\max\{\log (N),p\}}k^\alpha \bigg\lVert\max_{0 \le j \le N-1} \left\|G(t_j,U^j)\right\|_\gHY \bigg\rVert_p\nonumber\\
    &\le KC_\alpha  \big(\LGY \KufgY+\nn g\nn_{p,\infty,Y}\big) \sqrt{t_N}\sqrt{\max\{\log (N),p\}}k^\alpha
\end{align}
using that $R$ approximates $S$ to order $\alpha$ on $Y$, the ideal property of $\gHX$, linear growth, and stability of $R$ on $Y$.
Combining the bounds \eqref{eq:M31rate} to  \eqref{eq:M35rate}, we deduce
\begin{align}
\label{eq:M3rate}
	M_3 &\le B_p \CGX L_1 \sqrt{t_N} k + \sqrt{\frac{3}{2}} B_p\CGX L_2\sqrt{t_N} k^{1/2} + C_{3,\alpha} \sqrt{t_N} k^{\alpha}\nonumber\\
    &\phantom{\le }+C_{3,\log}\sqrt{t_N}\sqrt{\max\{\log (N),p\}}k^{\alpha}+B_p\CGX  \bigg(k\sum_{l=0}^{N-1} E(l)^2\bigg)^{1/2}.
\end{align}
Having bounded each term individually in \eqref{eq:M1rate}, \eqref{eq:M2rate} and \eqref{eq:M3rate}, we conclude
\begin{align*}
    E(N) &\le C_1 k + C_2k^{1/2} +C_3 k^{\alpha}+C_4 \sqrt{\max\{\log(N_k),p\}}k^{\alpha}+ C \bigg(k \sum_{l=0}^{N-1} E(l)^2\bigg)^{1/2},
\end{align*}
noting that $N \le N_k$ and $t_N \le T$. Thus, by the discrete version of Gronwall's Lemma \ref{lem:KruseGronwall}
\begin{align*}
    E(N) \le (1+C^2t_N)^{1/2}\ee^{(1+C^2t_N)/2}\left(C_1 k+ C_2k^{1/2}
    +C_3 k^{\alpha} +C_4\sqrt{\max\{\log(N_k),p\}}k^{\alpha}\right)
\end{align*}
follows. The desired error estimate is obtained for $N=N_k$. As $k \to 0$, the terms with the lowest exponents dominate, i.e.
\begin{equation*}
    E(N_k) \lesssim k^{1/2}+k+\sqrt{\max\{\log(N_k),p\}}k^{\alpha} \lesssim \sqrt{\max\{\log(N_k),p\}} k^{\min\{\frac{1}{2},\alpha\}},~~~ (k \to 0). \qedhere
\end{equation*}
\end{proof}

\begin{remark}
\label{rem:Kolmogorovapproach}
The result \cite[Theorem 1.1]{CHJNW} combines H\"older regularity in the $p$-th moment and bounds on the pointwise strong error to obtain a uniform strong error. Their effective method is based on a sophisticated application of the Kolmogorov-Chentsov continuity theorem, as well as approximation arguments. Let us refer to this method for obtaining uniform strong error estimates as the {\em Kolmogorov-Chentsov method}. At first sight, one might think that the result can be used to obtain the convergence rate of Theorem \ref{thm:convergenceRate} up to an arbitrary $\varepsilon>0$. Below, we point out what can precisely be achieved via their method.

Suppose that $R$ approximates $S$ to order $1/2$, a pointwise strong error estimate of rate $1/2$ has already been established, and Assumption \ref{ass:FG_structured} holds for fixed $p\in [2, \infty)$ and $\alpha=1/2$.
This means that the fixed data $(u_0, f, g)$ is assumed to have certain $L^p(\Omega)$-integrability. We will check what type of rate the Kolmogorov-Chentsov method yields for
\[
    {\rm E}_{k}^{p,\infty}\coloneqq\bigg\lVert\max_{0 \le j \le N_k}\|U(t_j)-U^j\|\bigg\rVert_p,
\]
and compare it to the rate ${\rm E}_{k}^{p,\infty}\leq C_p k^{1/2}\sqrt{\log(T/k)}$ we obtained in Theorem \ref{thm:convergenceRate}.
We distinguish between three cases.
\begin{enumerate}[label=(\alph*)]
    \item Integrability of data in $L^2(\Omega)$: In this case, the Kolmogorov-Chentsov method does not apply, so no convergence rate is obtained.
    \item Integrability of data in $L^p(\Omega)$ for a fixed $p\in (2, \infty)$: the Kolmogorov-Chentsov method gives ${\rm E}_{k}^{p,\infty}\leq C_{\gamma,p} k^{\gamma-1/p}$ for any $\gamma\in (1/p,1/2)$.
    \item Integrability of data in $L^p(\Omega)$ for all $p\in (2, \infty)$: the Kolmogorov-Chentsov method gives ${\rm E}_{k}^{p,\infty}\leq C_{\gamma,p} k^{\gamma}$ for any $\gamma\in (0,1/2)$.
\end{enumerate}
In the last case, there is an arbitrarily small difference in the error rate. We can obtain this error rate under the assumption that the data is $L^p(\Omega)$-integrable for a fixed $p\in [2, \infty)$. In the case one has this for all $p<\infty$, one needs to choose a very large $p$ in the Kolmogorov-Chentsov method to get close to the desired rate, which in turn produces large constants in the rate estimate.
\end{remark}

\subsection{The exponential Euler method}
\label{subsec:expEulerMultNoise}

We analyse the time discretisation error for the special case $R_k\ce S(k)$ known as the \textit{exponential Euler method}. Obviously, the exponential Euler method is contractive for contractive semigroups. Furthermore, several terms in the error analysis vanish for the exponential Euler method, since $S(t_j)-R_k^j = S(t_j)-S(k)^j=0$ by the semigroup property. In particular, the logarithmic correction factor is not needed for this scheme.

\begin{corollary}[Exponential Euler]
\label{cor:expEulerMultiplicative}
    Suppose that Assumption \ref{ass:FG_structured} holds for some $\alpha \in (0,1]$ and $p \in [2,\infty)$.
    Let $A$ be the generator of a $C_0$-contraction semigroup $(S(t))_{t \ge 0}$ on both $X$ and $Y$.
    Suppose that $Y \hra D_A(\alpha,\infty)$ continuously if $\alpha \in (0,1)$ or $Y \hra D(A)$ continuously if $\alpha=1$. Let $u_0 \in L_{\calF_0}^p(\Omega;Y)$. Consider the exponential Euler method $R\ce S$ for time discretisation. Denote by $U$ the mild solution of \eqref{eq:StEvolEqnFG} and by $(U^j)_{j=0,\ldots,N_k}$ the temporal approximations as defined in \eqref{eq:defUj}.
    Then for $N_k \geq 2$
    \begin{equation*}
        \left\lVert\max_{0 \le j \le N_k}\|U(t_j)-U^j\|\right\rVert_p
        \le C_{\rmS,\ee} \left(C_{\rmS,1} k+ C_{\rmS,2}k^{1/2}
        +C_{\rmS,3}k^{\alpha}\right)
    \end{equation*}
    with constants $C_{\rmS,\ee} \ce C_\ee$, $C_{\rmS,1} \ce C_1$, $C_{\rmS,2} \ce C_2$ as in Theorem \ref{thm:convergenceRate}, $C_{\rmS,3} \ce C_{\rmS,2,\alpha}T+C_{\rmS,3,\alpha} T^{1/2}$, $C_{\rmS,3,\alpha} \ce C_{3,\alpha}$, and
    \begin{align*}
         C_{\rmS,2,\alpha} &\ce \frac{1}{\alpha+1}\left(\CFX L_3+C_{\alpha,F} + 2\CY\big(\LFY \KufgY+\|f\|_{L^p(\Omega;C([0,T];Y))}\big)\right),
    \end{align*}
    where $C_{3,\alpha}$ is as defined in Theorem \ref{thm:convergenceRate}, $L_3$ as in Lemma \ref{lem:mildsolest}, $\KufgY$ as in \eqref{eq:defKu0fgY}, $\CY$ denotes the embedding constant of $Y$ into $D_A(\alpha,\infty)$ or $D(A)$, and $B_p$ is the constant from Theorem \ref{thm:maxIneqQuasiContractive}.

    In particular, the approximations $(U^j)_j$ converge at rate $\min\{\alpha,\frac{1}{2}\}$ as $k \to 0$.
\end{corollary}

\begin{proof}
Adopt the notation from the proof of Theorem \ref{thm:convergenceRate}. Contractivity of $R$ on $X$ and $Y$ is immediate from contractivity of $S$ on these spaces. Since $S(t_j)-R_k^j=0$ for any $j \in \{0,\ldots,N_k\}$, the terms $M_1$ and $M_{3,5}$ vanish. Moreover, the second term in $M_{2,4}$ vanishes so that
\begin{align*}
    M_{2,4} &\le \frac{2 \CY }{\alpha+1} \big(\LFY \KufgY+\|f\|_{p,\infty,Y}\big)t_N k^\alpha.
\end{align*}
Combining the individual bounds for the remaining terms, the estimate follows from a discrete Gronwall argument as in the proof of Theorem \ref{thm:convergenceRate}. The logarithmic correction factor vanishes due to $M_{3,5}=0$.
\end{proof}

\begin{remark}
\label{rem:Milstein}
    Adding a term that is quadratic in the Wiener increment to the exponential Euler method yields the \emph{Milstein scheme}, which has been found to give good convergence properties \cite{Milstein}. In the parabolic case (i.e., $A$ self-adjoint and with compact resolvent), \cite[Thm.~1]{Milstein} yields convergence of rate arbitrarily close to $1$ in the cases of additive noise or multiplicative noise satisfying a commutativity condition, which has been removed in subsequent work \cite{MilsteinNonCommutative}. An extension of these results for the Milstein scheme to the hyperbolic case has been raised as a direction for future research in \cite{Milstein}, which, to the best of our knowledge, remains open. Moreover, in \cite{Milstein, MilsteinNonCommutative}, the pointwise strong error is analysed, from which a pathwise uniform convergence rate can only be obtained at the price of deteriorating the convergence rate, as discussed Remark \ref{rem:Kolmogorovapproach}.
    
\end{remark}

\subsection{Error estimates on the full time interval}
\label{subsec:fullTimeInterval}

In this subsection, we will extend the error estimates of Theorem \ref{thm:convergenceRate} and Corollary \ref{cor:expEulerMultiplicative} to the full time interval by using a suitable H\"older regularity of the paths of the mild solution.

\begin{example}
    Fix $N\geq 1$. Below, we construct a process $v_N:[0,1]\times\Omega\to \R$ such that $\sup_{t\in [0,1]}\E |v_N(t)|^p\leq 1/N$, but $v_N(t) = 1$ for all $t$ in a neighborhood of $\{i/N: i\in \{1, \ldots, N\}\}$. This show that information on the pointwise strong error does not provide much insight on the path of $v_N$ in general.

    Indeed, let $\Omega = \{\omega_{m, i}: i\in \{1, \ldots, N\}, m\in \N\}$. For every $i\in\{1, \ldots, N\}$ suppose that  $\P(\omega_{m,i}) = \frac{2^{-m}}{N}$. Let $I_{N} = \bigcup_{m\geq 1} \bigcup_{i=1}^{N} \{\omega_{m,i}\}\times(\frac{i}{N} - \frac{1}{2N}, \frac{i}{N}+ \frac{1}{2N})$, and set $v_N(\omega,t) = 1$ if $(\omega,t)\in I_{N}$. Then one can check that $v_N$ satisfies the required estimates.
\end{example}

The undesired behavior in the above example shows the need for having maximal estimates on the full time interval, i.e.\ estimates for $\|\sup_{t \in [0,T]} \|U(t)-\tilde{U}(t)\|\|_p$, where $\tilde{U}$ is the process obtained from the discrete approximation using piecewise constant interpolation.

The following simple deterministic result provides a way to connect the uniform error to the error on the grid. Given a non-decreasing function $\Phi:[0,T] \to [0,\infty)$ such that $\Phi \neq 0$ on $(0,T]$ we say that $u\in C^\Phi([0,T];X)$ if $u:[0,T]\to X$ is continuous and
\[[u]_{C^\Phi([0,T];X)} = \sup_{0\leq s<t\leq T} \frac{\|u(t)-u(s)\|}{\Phi(t-s)}<\infty.\]
Moreover, we set $\|u\|_{C^\Phi([0,T];X)} \ce \|u\|_{\infty}+ [u]_{C^\Phi([0,T];X)}$. We shall be particularly interested in the function $\Phi(r)= r^\alpha(1+\log(\frac{T}{r}))^{1/2}$ for $r \in (0,T]$ for some $\alpha>0$ and $\Phi(0)= 0$ in the following.

\begin{lemma}[Decomposition of the error on the full time interval]
\label{lem:splitContErrorPhi}
    Let $u \in C^\Phi([0,T];X)$ for a non-decreasing function $\Phi:[0,T] \to [0,\infty)$ such that $\Phi \neq 0$ on $(0,T]$. Let $\Pi \subseteq [0,T]$ be a finite time grid, and denote by $\tilde{u}:\Pi \to X$ an approximation of $u$, which is extended to $[0,T]$ by setting $\tilde{u}(t) \ce \tilde{u}(\lfloor t \rfloor_\Pi)$ for $t \notin \Pi$, where $\lfloor t \rfloor_\Pi \ce \max\{s \in \Pi: s\leq t\}$. Then it holds that
    \begin{equation*}
        \sup_{t \in [0,T]} \|u(t)-\tilde{u}(t)\| \le \Phi(h)\cdot\|u\|_{C^\Phi([0,T];X)}+\sup_{t \in \Pi} \|u(t)-\tilde{u}(t)\|
    \end{equation*}
    for the maximal time step $h\ce \sup_{t \in [0,T]} \dist(t,\Pi)$.
\end{lemma}
\begin{proof}
For $t\in [0,T]$ we can write
\begin{align*}
\|u(t)-\tilde{u}(t)\| & \leq \|u(t)-u(\lfloor t\rfloor_\Pi)\| + \|u(\lfloor t\rfloor_\Pi)-\tilde{u}(t)\|
\\ & \leq \|u\|_{C^\Phi([0,T];X)} \cdot \Phi(t-\lfloor t\rfloor_\Pi)+ \sup_{s \in \Pi} \|u(s)-\tilde{u}(s)\|,
\end{align*}
which implies the required result.
\end{proof}

From the above, we see that to estimate the uniform error on $[0,T]$, we need an (optimal) H\"older regularity result for the mild solution $U$ to \eqref{eq:StEvolEqnFG}. To obtain such a result, the main difficulty lies in estimating the stochastic convolution.
\begin{lemma}[Path regularity of stochastic convolutions]
    \label{lem:pathRegStConv}
    Let $X,Y$ be Hilbert spaces such that $Y \hra X$ continuously. Let $A$ be the generator of a $C_0$-contraction semigroup $(S(t))_{t \ge 0}$ on both $X$ and $Y$. Suppose that $Y \hra D_A(\alpha,\infty)$ holds for some $\alpha \in (0,1/2]$. Let $q\in (2, \infty]$ be such that $\frac{1}{2}-\frac1q = \alpha$ and let $2\leq p<p_0<\infty$. Suppose that \[g\in L^{p}(\Omega;L^2(0,T;\calL_2(H,Y))) \cap L^{p_0}(\Omega;L^q(0,T;\calL_2(H,X)))\]
    and define $J_g:\Omega\times [0,T]\to X$ as the stochastic convolution
    \[J_g(t) = \int_0^t S(t-s) g(s) d W_H(s).\]
    Then one has $J_g\in L^{p}(\Omega;C^\Psi([0,T];X))$ for $\Psi: (0,T] \to (0,\infty), \Psi(r) \ce r^{\alpha} (1+\log(\frac{T}{r}))^{1/2}$ and there exist constants $C_{p},C_{\alpha,p,p_0,T} \geq 0$ such that
    \begin{align*}
    \|J_g\|_{L^{p}(\Omega;C^\Psi([0,T];X))}\leq C_{p} \|g\|_{L^p(\Omega;L^2(0,T;\calL_2(H,Y)))}+ C_{\alpha,p,p_0,T} \|g\|_{L^{p_0}(\Omega;L^q(0,T;\calL_2(H,X)))}.
    \end{align*}
\end{lemma}
By a simple rescaling, the result extends to quasi-contraction semigroups. Moreover, from the proof below one can see that a certain Orlicz integrability in $\Omega$ is sufficient for $g$. Note that the above path regularity is optimal for $q=\infty$. Indeed, L\'evy's modulus of continuity theorem for a scalar Brownian motion states that a.s.\
\[\limsup_{h\downarrow 0} \sup_{t\in [0,1-h]} \frac{B(t+h) - B(t)}{\sqrt{2h\log(1/h)}} = 1,\]
which shows that $\Psi$ cannot be replaced by a ``better'' function.

\begin{proof}[Proof of Lemma \ref{lem:pathRegStConv}]
    For $0\leq s<t\leq T$, we can write
    \begin{align*}
    \|J_g(t) - J_g(s)\|&\leq \Big\|(S(t-s)-I) \int_0^s S(s-r) g(r)\,\rmd W_H(r)\Big\| + \Big\|\int_s^t S(t-r) g(r) \,\rmd W_H(r)\Big\|\\ &\ec T_1(t,s)+T_2(t,s).
    \end{align*}
    For $T_1$ we can write
    \[T_1(t,s)\leq \|S(t-s)-I\|_{\calL(Y,X)} \Big\|\int_0^s S(s-r) g(r)\,\rmd W_H(r)\Big\|_Y  \leq c(t-s)^{\alpha} \|J_g(s)\|_Y\]
    for some $c \geq 0$. Therefore, by Theorem \ref{thm:maxIneqQuasiContractive} we obtain
    \begin{align*}
    \bigg\|\sup_{0\leq s<t\leq T}\frac{T_1(t,s)}{\Psi(t-s)}\bigg\|_{p}
    \leq c \bigg\|\sup_{0\leq s<t\leq T}\frac{\|J_g(s)\|_Y}{(1+\log(\frac{T}{t-s}))^{1/2}}\bigg\|_{L^p(\Omega)}
    \leq c B_p \|g\|_{L^p(\Omega;L^2(0,T;\calL_2(H,Y)))}.
    \end{align*}

    For $T_2$ we use the dilation result of \cite[Theorem I.7.1]{Nagybook} (cf. \cite{HausSei}). We can find a Hilbert space $\wt{X}$, a contractive injection $Q:X\to \wt{X}$, a contractive projection $P:\wt{X}\to X$, and a unitary $C_0$-group $(G(t))_{t\in \R}$ on $\wt{X}$ such that $S(t) = P G(t) Q$ for $t\geq 0$. Thus, we can write
    \begin{align*}
    T_2(t,s) = \Big\|\int_s^t PG(t-r) Q g(r)\,\rmd W_H(r)\Big\|_{X}
    \leq \Big\|\int_s^t G(-r) Q g(r)\,\rmd W_H(r)\Big\|_{\wt{X}}
    = \|I(t) - I(s)\|_{\wt{X}},
    \end{align*}
    where $I(t) \ce \int_0^t G(-r) Q g(r)\,\rmd W_H(r)$. Then by \cite[(2.12) and Theorem 3.2(vi)]{OndrejatVeraar} we have $I\in L^{p}(\Omega;C^{|\cdot|^{\alpha} |\log(\cdot)|^{1/2}}([0,T];\tilde{X}))$ and thus by boundedness of $|\log(\cdot)|^{1/2}(1+\log(\frac{T}{\cdot}))^{-1/2}$ on $(0,T]$ also $I\in L^{p}(\Omega;C^\Psi([0,T];\tilde{X}))$. Moreover,
    there are constants $c_{\alpha,T}, C_{\alpha,p,p_0,T} \geq 0$ such that
    \begin{align*}
    \|I\|_{L^p(\Omega;C^\Psi([0,T];\tilde{X}))} & \leq c_{\alpha,T} \|I\|_{L^p(\Omega;B^{\alpha}_{\Phi_2,\infty}(0,T;\wt{X}))}
    \\ & \leq C_{\alpha,p,p_0,T} \|G(-r) Q g(r)\|_{L^{p_0}(\Omega;L^q(0,T;\calL_2(H,\wt{X})))}
    \\ & \leq C_{\alpha,p,p_0,T} \|g\|_{L^{p_0}(\Omega;L^q(0,T;\calL_2(H,X)))},
    \end{align*}
    where $B_{\Phi_2,\infty}^\alpha(0,T;\tilde{X})$ denotes the Besov-Orlicz space corresponding to $\Phi_2(x)\ce \exp(x^2)-1$, cf. \cite[Section~2.3]{OndrejatVeraar} for the definition. It follows that
    \begin{align*}
        \bigg\|\sup_{0\leq s<t\leq T}\frac{T_2(t,s)}{\Psi(t-s)}\bigg\|_{p} &\le \|I\|_{L^p(\Omega;C^\Psi([0,T];\tilde{X}))}
        \leq C_{\alpha,p,p_0,T} \|g\|_{L^{p_0}(\Omega;L^q(0,T;\calL_2(H,X)))}.
    \end{align*}
    Now the required estimate follows by combining the estimates for $T_1$ and $T_2$.
\end{proof}

\begin{remark}
    For analytic semigroups on $X$, the result of Lemma \ref{lem:pathRegStConv} even holds if merely $g\in L^{p_0}(\Omega;L^q(0,T;\calL_2(H,X)))$, and even $J_g\in L^p(\Omega;B^{\alpha}_{\Phi_2,\infty}(0,T;X))$ (see \cite[Theorem 5.1]{OndrejatVeraar}). In particular, the space $Y$ and contractivity of $S$ are not needed. We do not know if one can take $p_0 = p$ in Lemma \ref{lem:pathRegStConv}, even in the analytic case. Also, we do not know if the above Besov regularity of $J_g$ holds in the non-analytic case.

Sharp path regularity results such as the one of Lemma \ref{lem:pathRegStConv} play an important role in obtaining convergence rates for numerical schemes for SPDEs. In particular, recent other applications of \cite{OndrejatVeraar} to numerics include \cite{DHW, le2023class, wichmann2022temporal, Wich23}. Below, we apply Lemma \ref{lem:pathRegStConv} to obtain additional information on the numerical approximation in the Kato setting, and it seems to be the first of its kind for hyperbolic equations.
\end{remark}
After these preparations, we can now prove the required path regularity of the mild solution.
\begin{proposition}[Path regularity of the mild solution]
\label{prop:pathRegMildSol}
    Suppose that Assumption \ref{ass:FG_structured} holds for some $\alpha \in (0,1/2]$ and $p \in [2,\infty)$.
    Let $p_0 \in (p,\infty)$ and $q\in (2, \infty]$ be such that $\frac{1}{2}-\frac1q = \alpha$, and suppose that $f,g$, and $u_0$ additionally satisfy
    \[f\in L^{p_0}(\Omega;L^1(0,T;X)), \  \ g\in L^{p_0}(\Omega;L^q(0,T;\calL_2(H,X))), \ \ \text{and}  \ \ u_0 \in L_{\calF_0}^{p_0}(\Omega;X)\cap     L_{\calF_0}^{p}(\Omega;Y).\]
    Let $A$ be the generator of a $C_0$-contraction semigroup $(S(t))_{t \ge 0}$ on both $X$ and $Y$.
    Suppose that $Y \hra D_A(\alpha,\infty)$ continuously. Let $\Psi: (0,T] \to (0,\infty)$ be given by  $\Psi(r) \ce r^{\alpha} (1+\log(\frac{T}{r}))^{1/2}$.
    Then the mild solution to \eqref{eq:StEvolEqnFG}  satisfies $U\in L^{p}(\Omega;C^\Psi([0,T];X))$ and there exists a constant $C$ depending on $(T,p,p_0, \alpha, \tilde{F}, \tilde{G}, \; X, Y)$ such that
    \begin{align*}
    \|U\|_{L^p(\Omega;C^\Psi([0,T];X))} \le C\big(1&+\|u_0\|_{L^{p}(\Omega;Y)}+\|f\|_{p,\infty,Y}+\nn g\nn_{p,\infty,Y}\\ & +
    \|u_0\|_{L^{p_0}(\Omega;X)}  + \|f\|_{p_0,1,X}+\nn g\nn_{{p_0},q,X}\big).
    \end{align*}
\end{proposition}
\begin{proof}
    The mild solution formula \eqref{eq:mildSoldefU} yields an initial value term, a difference of deterministic convolutions, and a stochastic version of the latter. The first two can be estimated as in the proof of Lemma \ref{lem:mildsolest}, resulting in an upper bound of the form
    \begin{equation*}
        c(1+\|u_0\|_{L^p(\Omega;Y)}+\|f\|_{p,\infty,Y} + \nn g \nn_{p,2,Y})
    \end{equation*}
    for some $c \ge 0$ depending on $T$. To the remaining term, we apply Lemma \ref{lem:pathRegStConv} and note that
    \begin{align*}
        \nn G(\cdot,U(\cdot))\nn_{p,2,Y} &\le \LGY \CufgY+\nn g \nn_{p,\infty,Y},\\
        \nn G(\cdot,U(\cdot))\nn_{p_0,q,X} &\le T^{1/q} \nn G(\cdot,U(\cdot))\nn_{p_0,\infty,X}+\nn g \nn_{p_0,q,X} \le T^{1/q}\CGX \tilde{C}_{u_0,f,g,X}+\nn g \nn_{p_0,q,X}\\
        &\lesssim 1+\|u_0\|_{L^{p_0}(\Omega;X)}+\|f\|_{p_0,1,X}+\nn g\nn_{p_0,q,X},
    \end{align*}
    where $\tilde{C}_{u_0,f,g,X}$ is defined as $\CufgX$ in \eqref{eq:defCu0fgZ} with $p$ replaced by $p_0$.
\end{proof}

Consequently, we can now ``upgrade'' Theorem \ref{thm:convergenceRate} and Corollary \ref{cor:expEulerMultiplicative} to estimates on the full time interval.
\begin{theorem}[Uniform error on the full interval]\label{thm:generalschemeuniform0T}
    Suppose that Assumption \ref{ass:FG_structured} holds for some $\alpha \in (0,1/2]$ and $p \in [2,\infty)$. Let $A$ be the generator of a $C_0$-contraction semigroup $(S(t))_{t \ge 0}$ on both $X$ and $Y$.
    Let $(R_{k})_{k>0}$ be a time discretisation scheme which is contractive on $X$ and $Y$ and $R$ approximates $S$ to order $\alpha$ on $Y$ or suppose that $R_k=S(k)$ is the exponential Euler method. Suppose that $Y \hra D_A(\alpha,\infty)$ continuously. Let $p_0 \in (p,\infty)$ and $q\in (2, \infty]$ be such that $\frac{1}{2}-\frac1q = \alpha$, and suppose that $f,g$, and $u_0$ have additional integrability as $X$-valued processes
    \[f\in L^{p_0}(\Omega;L^1(0,T;X)), \  \ g\in L^{p_0}(\Omega;L^q(0,T;\calL_2(H,X))), \ \ \text{and}  \ \ u_0 \in L_{\calF_0}^{p_0}(\Omega;X)\cap     L_{\calF_0}^{p}(\Omega;Y).\]
    Denote by $U$ the mild solution of \eqref{eq:StEvolEqnFG} and by $(U^j)_{j=0,\ldots,N_k}$ the temporal approximations as defined in \eqref{eq:defUj}. Define the piecewise constant extension $\tilde{U}:[0,T] \to L^p(\Omega;X)$ by $\tilde{U}(t) \ce U^j$ for $t \in [t_j,t_{j+1})$, $0 \le j \le N_k-1$, and $\tilde{U}(T) \ce U^{N_k}$. Then for all $N_k \ge 2$ there is a constant $C \ge 0$ depending on $(u_0,T,p,p_0,\alpha,F,G, X,Y)$ such that
    \[
        \bigg\|\sup_{t \in [0,T]} \|U(t)-\tilde{U}(t)\|\bigg\|_p \le C \big(1+\sqrt{\log(T/k)}\big)k^\alpha.
    \]
\end{theorem}
\begin{proof}
    The error bound follows from applying Lemma \ref{lem:splitContErrorPhi} with $\Phi=(\cdot)^\alpha (1+\log(\frac{T}{\cdot}))^{1/2}$ in combination with Theorem \ref{thm:convergenceRate} and Proposition \ref{prop:pathRegMildSol} to bound the first and second term obtained from the proposition, respectively.
\end{proof}

Thus we can conclude that  Theorem \ref{thm:convergenceRate} and Corollary \ref{cor:expEulerMultiplicative} can be improved to a uniform error estimate on $[0,T]$ at the price of a slightly more restrictive integrability condition on $g$ and $u_0$. Moreover, in the exponential Euler method, an additional logarithmic factor appears. Recall from \cite[Theorem 3]{Muller-Gronbach} that already for SDEs the error has to grow at least as $\log(T/k)^{1/2}k^{1/2}$ for $k\to 0$. Therefore, for $\alpha = 1/2$, Theorem \ref{thm:generalschemeuniform0T} gives the optimal convergence rate for any scheme.

In the applications given below, we restrict ourselves to the uniform error estimate on the grid points. By the above result, these statements can be extended to the full interval $[0,T]$ with additionally the square root of a logarithmic factor by imposing extra integrability conditions on the data.

\subsection{Application to the Schrödinger equation}
\label{subsec:SchroedingerMult}

In this subsection, we reconsider the stochastic Schrödinger equation with a potential from Subsection \ref{subsec:SchroedingerAdd}, now with linear multiplicative noise
\begin{align}
\label{eq:linearSchroedingerPotentialMultNoise}
    \Bigg\{\begin{split} \rmd u &= -\iu(\Delta + V) u \;\rmd t-\iu u\; \rmd W~~~ \text{ on }[0,T],\\
    u(0)&=u_0
    \end{split}
\end{align}
and its nonlinear variant with $\phi:\C\to \C$ and $\psi:\C \to \C$,
\begin{align}
\label{eq:linearSchroedingerPotentialMultNoiseNL}
    \Bigg\{\begin{split} \rmd u &= -\iu(\Delta u + V u + \phi(u))\;\rmd t-\iu \psi(u)\; \rmd W~~~ \text{ on }[0,T],\\
    u(0)&=u_0
    \end{split}
\end{align}
in $\R^d$ for $d \in \N$, with $Q$-Wiener process $\{W(t)\}_{t\ge 0}$, potential
$V$ and initial value $u_0$ as introduced in Subsection \ref{subsec:SchroedingerAdd}.

Let $\sigma \ge 0$ and, for this subsection only, write $L^2=L^2(\R^d; \C)$ and $H^\sigma=H^\sigma(\R^d;\C)$. We recall that the well-posedness of \eqref{eq:linearSchroedingerAddNoise} required Assumption \ref{ass:sigmadVSchrodinger} on $\sigma$ and $d \in \N$ to hold so that multiplication by $V$ is a bounded operator on $X=H^\sigma$. For multiplicative noise, this assumption is also required to hold on $Y=H^{\sigma+\ell\alpha}$, where the choice of $\ell$ depends on the scheme employed. To facilitate checking the assumptions on $Y$, we use the following equivalent reformulation of Assumption \ref{ass:sigmadVSchrodinger}:
\begin{assumption}
\label{ass:sigmadVSchrodingerReformulated}
    Let $\sigma \ge 0$, $d \in \N$ and $V \in L^2$ such that
    \begin{enumerate}[label=(\roman*)]
        \item $\sigma > \frac{d}{2}$ and $V \in H^\sigma$, or
        \item $\sigma = 0$ and $V \in H^\beta$ for some $\beta>\frac{d}{2}$, or
        \item $d=1$, $\sigma \in (0,\frac12)$, and $V \in H^\beta$ for some $\beta >\frac{1}{2}$
        \item $d \ge 2$, $\sigma \in (0,1]$, and $V \in H^\beta$ for some $\beta>\frac{d}{2}$.
    \end{enumerate}
\end{assumption}

Based on the combination of the cases of Assumption \ref{ass:sigmadVSchrodingerReformulated} for $X=H^\sigma$ and $Y=H^{\sigma+\ell\alpha}$, the following assumption emerges.

\begin{assumption}
\label{ass:sigmadVSchrodingerMult}
    Let $\sigma \ge 0$, $d \in \N$, $\alpha \in \left(0,\frac{1}{2}\right]$, $\ell \in (0,\infty)$, $V \in H^\beta$ for some $\beta>0$ such that
    \begin{enumerate}[label=(\roman*)]
    \item $\sigma > \frac{d}{2}$ and $\beta=\sigma+\ell\alpha$, or \label{item:sigmad2Mult}
    \item $\sigma = 0$, $1 \le d < \ell$, $\alpha >\frac{d}{8}$, and $\beta= \ell\alpha$, or
    \label{item:sigma0largeAMult}
    \item $\sigma = 0$, $d=1$, $\alpha<\frac{1}{2\ell}$, and $\beta>\frac{1}{2}$, or
        \label{item:sigma0d1Mult}
    \item $\sigma = 0$, $d\ge 2$, $\alpha\le\frac{1}{\ell}$, and $\beta>\frac{d}{2}$, or
        \label{item:sigma0d2Mult}
    \item $d=1$, $\sigma \in (0,\frac12)$, $\alpha> \frac{1-2\sigma}{2\ell}$, and $V\in H^{\sigma+\ell \alpha}$, or
        \label{item:d1largeAMult}
    \item $d=1$, $\sigma \in (0,\frac12)$, $\alpha<\frac{1-2\sigma}{2 \ell}$, and $\beta>\frac{1}{2}$, or
        \label{item:d1smallAMult}
    \item $2 \le d < 2\sigma + \ell$, $\sigma \in (0,1]$, $\alpha >\frac{d-2\sigma}{2\ell}$, and $\beta=\sigma+\ell\alpha$, or
        \label{item:d2largeAMult}
    \item $d \ge 2$, $\sigma \in (0,1]$, $\alpha\le \frac{1-\sigma}{\ell}$, and $\beta>\frac{d}{2}$.
        \label{item:d2smallAMult}
\end{enumerate}
\end{assumption}

For the exponential Euler method, we recover the error bound from \cite[Thm.~5.5]{AC18} showing convergence rate $\frac{1}{2}$ for linear noise in the case of sufficiently regular $Q^{1/2}$ and $V$ and $\sigma>\frac{d}{2}$. Assuming less regularity of $Q^{1/2}$ and $V$ we extend their result to fractional convergence rates $\alpha \in \left(0,\frac{1}{2}\right]$ as well as the cases \ref{item:sigma0largeAMult}-\ref{item:d2smallAMult} of Assumption \ref{ass:sigmadVSchrodingerMult}.

\begin{theorem}
\label{thm:SoExpEulerMultiplicative}
    Let $\sigma \ge 0$, $d \in \N$, and $V \in L^2$. Suppose that Assumption \ref{ass:sigmadVSchrodingerMult} is satisfied for some $\ell\ge 2$ and some $\alpha \in \left(0,\frac{1}{2}\right]$, $\beta>0$, and $p \in [2,\infty)$, and that $u_0 \in L_{\calF_0}^p(\Omega;H^{\sigma+\ell\alpha})$ as well as $Q^{1/2} \in \calL_2(L^2,H^\beta)$. Denote by $U$ the mild solution of the linear stochastic Schrödinger equation with multiplicative noise \eqref{eq:linearSchroedingerPotentialMultNoise} and by $(U^j)_{j=0,\ldots,N_k}$ the temporal approximations as defined in \eqref{eq:defUj} obtained with the exponential Euler method $R\ce S$.
    Then there exists a constant $C \ge 0$ depending on $(V, u_0, T, p, \alpha, \sigma, d, \ell)$ such that for $N_k \geq 2$
    \begin{equation*}
        \bigg\| \max_{0 \le j \le N_k} \|U(t_j)-U^j\|_{H^\sigma} \bigg\|_p \le C\big(1+ \|Q^{1/2}\|_{\calL_2(L^2,H^\beta)}\big) k^\alpha.
    \end{equation*}
    In particular, the approximations $(U^j)_j$ converge at rate $\frac{1}{2}$ as $k \to 0$ if $Q^{1/2} \in \calL_2(L^2,H^{\sigma+1})$, $V \in H^{\sigma+1}$, $\sigma>\frac{d}{2}$, and $u_0 \in L_{\calF_0}^p(\Omega;H^{\sigma+1})$.
\end{theorem}
\begin{proof}
    By \cite[Lemma~2.1]{AC18}, $A= -\iu \Delta$ generates a contractive semigroup on both Hilbert spaces $X=H^\sigma$ and $Y=H^{\sigma+\ell\alpha}$. Furthermore, setting $F(u) = -\iu V \cdot u$ and $G(u) = -\iu M_u Q^{1/2}$ for $u \in H^\sigma$ with the multiplication operator $M_u$ allows us to rewrite \eqref{eq:linearSchroedingerPotentialMultNoise} in the form of a stochastic evolution equation \eqref{eq:StEvolEqnFG}. It remains to verify the mapping, linear growth and Lipschitz continuity conditions from Assumption \ref{ass:FG_structured}.

    Note that Assumption \ref{ass:sigmadVSchrodingerMult} implies that Assumption \ref{ass:sigmadVSchrodinger} is satisfied for both $\sigma$ and $\sigma+\ell\alpha$. In particular, this means that $Vu \in Y=H^{\sigma+\ell\alpha}$ for any $u \in H^{\sigma+\ell\alpha}$ and $\|Vu\|_{H^{\sigma+\ell\alpha}} \le C_V \|u\|_{H^{\sigma+\ell\alpha}}$ for some constant $C_V \ge 0$. More specifically, it can be shown that $C_V \lesssim \|V\|_{H^\beta}$, cf. Subsection \ref{subsec:SchroedingerAdd}.
    Hence, $F$ maps both $X$ and $Y$ into themselves and it is of linear growth on $Y$ because of
    \begin{equation*}
        \|F(u)\|_Y =\|-\iu V \cdot u \|_{H^{\sigma+\ell\alpha}} \le C_V \|u\|_{H^{\sigma+\ell\alpha}} = C_V \|u\|_Y,~~~u \in Y.
    \end{equation*}
    Likewise, Lipschitz continuity on $X$ is obtained.

    Set $H=L^2$. Due to
    \begin{align}
    \label{eq:estGQSo}
        \|G(u)\|_\gHY &=\|-\iu M_u \cdot Q^{1/2} \|_{\calL_2(L^2,H^{\sigma+\ell\alpha})}\nonumber\\
        &\le \|M_u\|_{\calL(H^\beta,H^{\sigma+\ell\alpha})}  \|Q^{1/2} \|_{\calL_2(L^2,H^\beta)}\nonumber\\
        &\lesssim \|Q^{1/2} \|_{\calL_2(L^2,H^\beta)} \|u\|_{H^{\sigma+\ell\alpha}} = \|Q^{1/2} \|_{\calL_2(L^2,H^\beta)}\|u\|_Y,~~~u \in Y,
    \end{align}
    $G$ is of linear growth on $Y$. To see this, we estimate the operator norm of $M_u$ from $H^\beta$ to $H^{\sigma+\ell\alpha}$ using either the Banach algebra property of $H^\beta$, a combination of Hölder's inequality and Sobolev embeddings or an argument analogously to Lemma \ref{lem:SoAddSigma01case} as discussed in Subsection \ref{subsec:SchroedingerAdd}. Likewise, we check Lipschitz continuity of $G$ on $X$ with a multiple of $\|Q^{1/2} \|_{\calL_2(L^2,H^\beta)}$ as Lipschitz constant. Measurability and Hölder continuity in time are trivially fulfilled due to $F$ and $G$ depending only on $u \in X$.
    Thus, Corollary \ref{cor:expEulerMultiplicative} is applicable with $X=H^\sigma,\, H=L^2$, and $Y=H^{\sigma+\ell\alpha}\hra H^{\sigma+2\alpha} \hra (H^\sigma, D(A))_{\alpha,\infty}$, yielding the desired error bound.
\end{proof}

Furthermore, Theorem \ref{thm:convergenceRate} enables us to extend \cite[Thm.~5.5]{AC18} to general discretisation schemes $R$ involving rational approximations at the price of an additional logarithmic factor. We focus on the implicit Euler method (IE) and the Crank--Nicolson method (CN), which approximate the Schrödinger semigroup to rate $\alpha$ on $Y=H^{\sigma+4 \alpha}$ and $Y=H^{\sigma+3\alpha}$, respectively (see Theorem \ref{thm:SoIECNAdditive}).

\begin{theorem}
\label{thm:SoIECNMultiplicative}
    Let $\sigma \ge 0$, $d \in \N$, and $V \in L^2$. Let $(R_k)_{k>0}$ be the implicit Euler method (IE) or the Crank--Nicolson method (CN) and set $\ell_0\ce 4$ or $\ell_0 \ce 3$, respectively. Suppose that Assumption \ref{ass:sigmadVSchrodingerMult} is satisfied for some $\ell\ge\ell_0$ and for some $\alpha \in \left( 0, \frac{1}{2}\right]$, $\beta >0$, and $p \in [2,\infty)$. Further, suppose that $u_0 \in L_{\calF_0}^p(\Omega;H^{\sigma+\ell\alpha})$ as well as $Q^{1/2} \in \calL_2(L^2,H^\beta)$. Denote by $U$ the mild solution of the linear stochastic Schrödinger equation with multiplicative noise \eqref{eq:linearSchroedingerPotentialMultNoise}
    and by $(U^j)_{j=0,\ldots,N_k}$ the temporal approximations as defined in \eqref{eq:defUj}. Then there exists a constant $C \ge 0$ depending on $(V, u_0, T, p, \alpha, \sigma, d, \ell)$ such that for $N_k \ge 2$
    \begin{equation*}
        \left\| \max_{0 \le j \le N_k} \|U(t_j)-U^j\|_{H^\sigma} \right\|_p \le C \big(1+ \|Q^{1/2}\|_{\calL_2(L^2,H^\beta)}\big) \sqrt{\log (T/k)} k^\alpha.
    \end{equation*}
    In particular, (IE) and (CN) converge at rate $\frac12$ up to logarithmic correction as $k \to 0$ if $V\in H^{\sigma+\ell\alpha}$, $Q^{1/2}\in \calL_2(L^2,H^{\sigma+\ell\alpha})$, $\sigma >\frac{d}{2}$, and $u_0 \in L_{\filtrF_0}^p(\Omega;H^{\sigma+\ell\alpha})$ with $\ell=4$ and $\ell=3$, respectively.
\end{theorem}
An analogous statement holds for all time discretisation schemes $(R_{k})_{k>0}$ which are contractive on $H^\sigma$ and $H^{\sigma+\ell\alpha}$ and approximate $S$ to order $\alpha$ on $H^{\sigma+\ell\alpha}$. The reader is referred to Proposition \ref{prop:functionalcalculus} for a tool to check contractivity.
As in the additive case, the conditions on the dimension $d \in \N$ are not required in the absence of a potential. In most cases, choosing $\ell=\ell_0$ is sufficient. However, in the situation of Assumption \ref{ass:sigmadVSchrodingerMult}\ref{item:sigma0largeAMult} or \ref{item:d2largeAMult}, choosing a larger $\ell$ can yield the additional regularity required to solve Schrödinger's equation in higher dimensions.
\begin{proof}
        We want to apply Theorem \ref{thm:convergenceRate} with $Y=H^{\sigma+\ell\alpha}$ for $\ell\ge \ell_0\in \{3,4\}$ and $X,H,F,G$ as in Theorem \ref{thm:SoExpEulerMultiplicative} for the exponential Euler method. The proof works analogously, replacing $\ell \ge 2$ by $\ell \ge \ell_0$. It remains to check that (IE) and (CN) are contractive on $H^\sigma$ and $H^{\sigma+\ell\alpha}$.
        But since (IE) and (CN) are defined via $A$ and a scaled version of its resolvent, $R_k$ commutes with resolvents of $A$ in both cases. Thus, Proposition \ref{prop:functionalcalculus} yields the assertion.
\end{proof}

When passing to a nonlinear situation as in \eqref{eq:linearSchroedingerPotentialMultNoiseNL}, showing Lipschitz continuity of $G$ requires estimates of the form
\begin{equation*}
    \|\psi(u)-\psi(v)\|_{H^\sigma} \lesssim \|u-v\|_{H^\sigma},~~~u,v \in H^\sigma
\end{equation*}
and similar for $\phi$. However, the best estimate known for $\sigma \in (0,1)$ and $\psi \in C^2$ with bounded first and second derivatives is \cite[Prop.~2.7.2]{TaylorBook},
\begin{equation*}
    \|\psi(u)-\psi(v)\|_{H^\sigma} \lesssim \|u-v\|_{H^\sigma} + (1+\|u\|_{H^\sigma}+\|v\|_{H^\sigma})\|u-v\|_{L^\infty}.
\end{equation*}
Since this estimate is nonlinear in $u$ and $v$, showing Lipschitz continuity of $G$ is currently out of reach for $\sigma>0$. Another reason to restrict our considerations to $\sigma=0$ in the following is the negative result from Dahlberg \cite{DahlbergAffineLinear}, see also the survey \cite{surveyBourdaud}. It states that for $\sigma+2\alpha \in \left(\frac{3}{2},1+\frac{d}{2}\right)$, the only mappings $\psi$ such that $\psi \circ u \in H^{\sigma+2\alpha}$ for all $ u \in H^{\sigma+2\alpha}$ are the affine-linear ones. Hence, in dimension $d>1$, the optimal rate $\alpha=\frac{1}{2}$ cannot be expected for all $\sigma>\frac{1}{2}$ for genuinely nonlinear $\psi$. For $\sigma=0$, however, a convergence rate can be obtained.

\begin{theorem}
    Let $\sigma = 0$, $d \in \N$, and $V \in L^2$. Suppose that one of the cases \ref{item:sigma0largeAMult}-\ref{item:sigma0d2Mult} of Assumption \ref{ass:sigmadVSchrodingerMult} is satisfied for $\ell = 2$ and for some $\alpha \in \left(0,\frac{1}{2}\right]$, $\beta>0$, and $p \in [2,\infty)$. Further, suppose that $u_0 \in L_{\calF_0}^p(\Omega;H^{\sigma+2\alpha})$ as well as $Q^{1/2} \in \calL_2(L^2,H^\beta)$. Let $\phi,\psi:\C\to \C$ be Lipschitz continuous and such that $\phi(0) = \psi(0) = 0$. Denote by $U$ the mild solution of the nonlinear stochastic Schrödinger equation with multiplicative noise \eqref{eq:linearSchroedingerPotentialMultNoiseNL} and by $(U^j)_{j=0,\ldots,N_k}$ the temporal approximations as defined in \eqref{eq:defUj} obtained with the exponential Euler method $R \ce S$. Then there exists a constant $C \ge 0$ depending on $(V, u_0, \phi,\psi,T, p, \alpha, d, \ell)$ such that for $N_k \geq 2$
    \begin{equation*}
        \left\| \max_{0 \le j \le N_k} \|U(t_j)-U^j\|_{L^2} \right\|_p \le C \big(1+ \|Q^{1/2}\|_{\calL_2(L^2,H^\beta)}\big) k^\alpha.
    \end{equation*}
    In particular, the approximations $(U^j)_{j}$ converge at rate $\frac{1}{2}$ as $k \to 0$ if $Q^{1/2} \in \calL_2(L^2,H^1)$, $V \in H^1$, and $u_0 \in L_{\calF_0}^p(\Omega;H^{1})$ for $d=1$. In dimension $d\ge 2$, this is attained for $Q^{1/2}\in \calL_2(L^2,H^{\beta})$ and $V \in H^{\beta}$ for some $\beta>\frac{d}{2}$, and $u_0 \in L_{\calF_0}^p(\Omega;H^{1})$.
\end{theorem}
\begin{proof}
    From the linear case, it is already clear that
    \begin{equation*}
        \|G(u)-G(v)\|_{\calL_2(L^2,L^2)} \lesssim \|\psi \circ u-\psi \circ v\|_{L^2} \|Q^{1/2}\|_{\calL_2(L^2,H^\beta)}.
    \end{equation*}
    Lipschitz continuity of $\psi$ with Lipschitz constant $C_\psi \ge 0$ implies Lipschitz continuity of $G$ on $X=L^2$ via
      \begin{align*}
        \|\psi \circ u-\psi \circ v\|_{L^2} \|Q^{1/2}\|_{\calL_2(L^2,H^\beta)}
        &\le C_\psi \|Q^{1/2}\|_{\calL_2(L^2,H^\beta)}\|u-v\|_{L^2}.
    \end{align*}
    Since from \eqref{eq:estGQSo} we know that
    \begin{equation}
    \label{eq:proofNLcomposition}
        \|G(u)\|_{\calL_2(L^2,H^{2\alpha})} \lesssim \|\psi \circ u\|_{H^{2\alpha}} \|Q^{1/2}\|_{\calL_2(L^2,H^\beta)},
    \end{equation}
    it remains to estimate the norm of the composition $\|\psi \circ u\|_{H^{2\alpha}}$ by a multiple of $\| u\|_{H^{2\alpha}}$ to show linear growth of $G$ on $H^{2\alpha}$. In case $\alpha<\frac{1}{2}$, $2\alpha \in (0,1)$ and thus, by \cite[Prop.~2.4.1]{TaylorBook}, $\|\psi \circ u \|_{H^{2\alpha}} \lesssim \|u\|_{H^{2\alpha}}$.  In the remaining cases, $2\alpha=1$ holds, so that
    \begin{align*}
        \|\psi \circ u \|_{H^{2\alpha}}^2 = \|\psi \circ u \|_{L^2}^2 + \|\nabla(\psi \circ u) \|_{L^2}^2 \le \|\psi \circ u \|_{L^2}^2 + C_\psi^2 \|\nabla u \|_{L^2}^2 \le \max\{1,C_\psi^2\} \|u\|_{H^1}^2,
    \end{align*}
    where in the first inequality we have invoked \cite[Prop.~2.6.1]{TaylorBook}. Hence, $G$ is of linear growth on $Y=H^{2\alpha}$. In the same way one can see that $F(u) = -\iu(Vu+\phi(u))$ is Lipschitz on $X$ and of linear growth on $Y$. The statement of this theorem follows by an application of Corollary \ref{cor:expEulerMultiplicative}.
\end{proof}

To estimate the composition in \eqref{eq:proofNLcomposition}, we required $2\alpha \in (0,1]$ to apply the composition estimates. It is an open problem whether such estimates also hold in $H^s$ for $s>1$. For real-valued functions, results have been obtained for $s< \frac32$ in \cite[Thm.~18]{BourdaudSickel}. These estimates being unknown for $s>1$ limits us to suboptimal convergence rates for schemes involving rational approximations, at least for nonlinear Schrödinger equations.

\begin{theorem}
    Let $\sigma = 0$, $d \in \N$, and $V \in L^2$. Let $(R_k)_{k>0}$ be the implicit Euler method (IE) or the Crank--Nicolson method (CN) and set $\ell_0\ce 4$ or $\ell_0 \ce 3$, respectively. Suppose that one of the cases \ref{item:sigma0largeAMult}-\ref{item:sigma0d2Mult} of Assumption \ref{ass:sigmadVSchrodingerMult} is satisfied for $\ell=\ell_0$ and some $\alpha \in \left(0,\frac{1}{\ell}\right]$, $\beta>0$, and $p\in[2,\infty)$. Further, suppose that $u_0 \in L_{\calF_0}^p(\Omega;H^{\ell\alpha})$ as well as $Q^{1/2} \in \calL_2(L^2,H^\beta)$.
    Let $\phi,\psi:\C\to \C$ be Lipschitz continuous and such that $\phi(0) = \psi(0) = 0$.
    Denote by $U$ the mild solution of the nonlinear stochastic Schrödinger equation with multiplicative noise \eqref{eq:linearSchroedingerPotentialMultNoiseNL} and by $(U^j)_{j=0,\ldots,N_k}$ the temporal approximations as defined in \eqref{eq:defUj}. Then there exists a constant $C \ge 0$ depending on $(V, u_0, \phi,\psi,T, p, \alpha, d, \ell)$ such that for $N_k \ge 2$
    \begin{equation*}
        \bigg\| \max_{0 \le j \le N_k} \|U(t_j)-U^j\|_{H^\sigma} \bigg\|_p \le C \big(1+\|Q^{1/2}\|_{\calL_2(L^2,H^\beta)}\big) \sqrt{\log(T/k)} k^\alpha.
    \end{equation*}
    In particular, in dimension $d=1$, (IE) converges at rate $\frac14$ up to logarithmic correction as $k \to 0$ if $V\in H^{1}$, $Q^{1/2}\in \calL_2(L^2,H^1)$, and $u_0 \in L_{\filtrF_0}^p(\Omega;H^1)$. For the same regularity of $V$,  $Q^{1/2}$, and $u_0$, (CN) converges at rate $\frac13$ up to logarithmic correction as $k \to 0$ in dimension $d=1$.
\end{theorem}
This theorem can be generalised to time discretisation schemes $(R_{k})_{k>0}$ that are contractive on $L^2$ and $H^{\ell\alpha}$, and that approximate $S$ to order $\alpha \in (0,\frac{1}{\ell}]$ on $H^{\ell\alpha}$.

\subsection{Numerical experiments for the Schrödinger equation}
\label{subsec:numerical}

In this subsection, we illustrate that convergence rates observed in numerical simulations correspond well to the analytic convergence rates obtained in Subsections \ref{subsec:SchroedingerAdd} and \ref{subsec:SchroedingerMult} for the Schrödinger equation. The code for the numerical simulations is available at \cite{schroedingerCode}.

We consider the linear stochastic Schrödinger equation without potential ($V=0$) and with periodic boundary conditions on $[0,2\pi]$ in the case of multiplicative noise \eqref{eq:linearSchroedingerPotentialMultNoise} and additive noise \eqref{eq:linearSchroedingerAddNoise}, respectively. For spatial discretisation, we employ a spectral Galerkin method with $M=2^{10}$ Fourier modes and calculate $L^2$-errors, i.e. $\sigma=0$. The initial values $u_0$ are taken with Fourier coefficients $(1+|\ell|^6)^{-1}$, $-M/2+1 \le \ell \le M/2$, resulting in sufficiently smooth initial values. We take the covariance operator $Q$ to have eigenvalues $\lambda_\ell = (1+|\ell|^\beta)^{-1}$ to the eigenfunctions $e_\ell = (2\pi)^{-1/2}\exp(\mathrm{i}\ell \cdot)$, $\ell \in \Z$. We choose the exponent as $\beta=5.1$ for additive noise and $\beta=3.1$ for multiplicative noise, which leads to $Q^{1/2} \in \calL_2(L^2,H^{2+\varepsilon})$ and $Q^{1/2} \in \calL_2(L^2,H^{1+\varepsilon})$ for any $\varepsilon \in (0,0.05)$, respectively. In the simulation, both the noise and the approximate solutions are truncated at wave numbers $-M/2+1 \le \ell \le M/2$. For time discretisation, we consider the exponential Euler method (EXP), the implicit Euler method (IE), and the Crank--Nicolson method (CN). For additive noise, case \ref{item:sigma0} of Assumption \ref{ass:sigmadVSchrodinger} is satisfied, so that according to Theorem \ref{thm:SoExpEulerAdditive}, for any $p \in [2,\infty)$, (EXP) shall converge with the optimal rate $1$. Analogously, by Theorem \ref{thm:SoIECNAdditive}, (IE) shall converge with rate $\frac{2+\varepsilon}{4}\approx 0.525$ and (CN) with rate $\frac{2+\varepsilon}{3} \approx 0.68$. The truncation error of the spectral Galerkin method can be computed to be of order $(M/2)^{-4} \approx 10^{-9}$, which is negligible. For multiplicative noise, case \ref{item:sigma0largeAMult} of Assumption \ref{ass:sigmadVSchrodingerMult} is satisfied, resulting in analytical rates of convergence $0.5$, $\frac{1+\varepsilon}{3} \approx 0.35$, and $\frac{1+\varepsilon}{4} \approx 0.26$ for (EXP), (CN), and (IE), respectively, based on Theorems \ref{thm:SoExpEulerMultiplicative} and \ref{thm:SoIECNMultiplicative}, respectively.

The numerical rates of convergence of the pathwise uniform error with $p=2$ of the three different schemes are illustrated in Figure \ref{fig:numerical} and stated in Table \ref{tab:numericalConvrate} for additive and multiplicative noise as described above. The expected analytical rates of convergence can be confirmed. Small deviations of the numerical from the analytical rate of convergence can be explained by the fact that the analytical solution is approximated by the exponential Euler method with a small time step $k=2^{-12}$ and $100$ samples
are used for the approximation of the expected values. For the approximations, time steps $k=2^{-5},\ldots, 2^{-9}$ are used.

\begin{center}
    \includegraphics[width=0.8\textwidth]{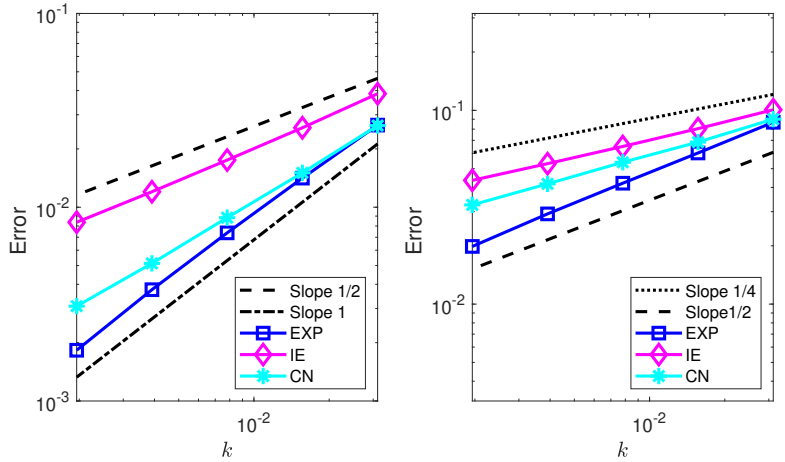}
    \captionof{figure}{Numerical rates of convergence for the stochastic Schrödinger equation with additive noise  (left) and multiplicative noise (right) for EXP (squares), IE (diamonds), and CN (asterisks).}\label{fig:numerical}
\end{center}

\begin{table}[ht]
\begin{tabular}{|c|c|c|c|}
  \hline
  & Exponential Euler & Implicit Euler & Crank--Nicolson \\ \hline
  $H^{2+\varepsilon}$-valued additive noise & 0.9650 & 0.5510 & 0.7071 \\
  \hline
  $H^{1+\varepsilon}$-valued multiplicative noise & 0.5321 & 0.3025 & 0.3675 \\
  \hline
\end{tabular}
\caption{Numerical rates of convergence for the stochastic Schrödinger equation}\label{tab:numericalConvrate}
\end{table}

\subsection{Application to Maxwell's equations}
\label{subsec:Maxwell}

As a second example, we consider the stochastic Maxwell's equations
\begin{align}
\label{eq:Maxwell}
    \Bigg\{\begin{split} \rmd U &= [AU+F(U)] \;\rmd t+ G(U)\; \rmd W~~~ \text{ on }[0,T],\\
    U(0)&=(\bfE_0^\top,\bfH_0^\top)^\top
    \end{split}
\end{align}
with boundary conditions of a perfect conductor as in \cite{CCHS20}. It describes the behaviour of the electric and magnetic field $\bfE$ and $\bfH$, respectively, on a bounded, simply connected domain $\calO \subseteq \R^3$ with smooth boundary with unit outward normal vector $\mathbf{n}$. Here, $A: D(A) \to X \ce L^2(\calO)^6$ is the Maxwell operator defined by
\begin{equation*}
    A\begin{pmatrix} \bfE\\ \bfH \end{pmatrix} \ce \begin{pmatrix} 0 & \varepsilon^{-1}\nabla\times\\-\mu^{-1}\nabla \times&0 \end{pmatrix}\begin{pmatrix} \bfE\\ \bfH \end{pmatrix} = \begin{pmatrix} \varepsilon^{-1}\nabla\times \bfH \\-\mu^{-1}\nabla \times \bfE \end{pmatrix}
\end{equation*}
on $D(A) \ce H_0(\curl,\calO) \times H(\curl,\calO)$ with $H(\curl,\calO) \ce \{\bfH \in (L^2(O))^3\,:\,\nabla \times \bfH \in L^2(\calO)^3\}$ and its subspace $H_0(\curl,\calO)$ of those $\bfH$ with vanishing tangential trace $\mathbf{n} \times \mathbf{H}\vert_{\partial\calO}$. The permittivity and permeability $\varepsilon, \mu \in L^\infty(\calO)$ are assumed to be uniformly positive, i.e., $\varepsilon,\mu \ge \kappa >0$ for some constant $\kappa$. We equip the Hilbert space $X=L^2(\calO)^6=L^2(\calO)^3\times L^2(\calO)^3$ with the weighted scalar product
\begin{equation*}
    \left\langle \begin{pmatrix} \bfE_1\\\bfH_1 \end{pmatrix},\begin{pmatrix} \bfE_2\\\bfH_2 \end{pmatrix}\right\rangle
    \ce \int_\calO \big(\mu \langle \bfH_1,\bfH_2\rangle + \varepsilon \langle \bfE_1,\bfE_2\rangle\big)\,\rmd x,
\end{equation*}
where $\langle\cdot,\cdot\rangle$ denotes the standard scalar product in $L^2(\calO)^3$.
Furthermore, $W$ is a $Q$-Wiener process for a symmetric, non-negative operator $Q$ with finite trace such that $Q^{1/2}\in \calL_2(H,X)$, where $H=L^2(\calO)^6$ is equipped with the standard norm.

For $F:\Omega \times[0,T] \times X \to X$ we consider the linear drift term given by
\begin{equation}
\label{eq:defFMaxwell}
    (\omega,t,U) \mapsto F(\omega,t,U)=\begin{pmatrix} \sigma_1(\cdot,t) \bfE\\\sigma_2(\cdot,t) \bfH\end{pmatrix},\quad U=(\bfE^\top,\bfH^\top)^\top,
\end{equation}
for sufficiently smooth $\sigma_1,\sigma_2: \calO \times [0,T] \to \R$. We assume boundedness of $\sigma_1,\sigma_2$ and their partial derivatives w.r.t. the spatial variables. In particular, let $\sigma_j$ be uniformly Lipschitz continuous in time and let $\partial_{x_i}\sigma_j, \sigma_j \in L^\infty(\calO \times [0,T])$ for $i=1,2,3$ and $j=1,2$. Then $F$ is Lipschitz on $X$ due to
\begin{align*}
    \|F(t,V)\|_X^2 &= \int_\calO \left( \mu(x) \|\sigma_2(\cdot,t)\bfH_V\|_{L^2(\calO)^3}^2+\varepsilon(x) \|\sigma_1(\cdot,t)\bfE_V\|_{L^2(\calO)^3}^2\right)\dx\\
    &\le \max\{\|\sigma_1\|_\infty,\|\sigma_2\|_\infty\}^2 \|V\|_X^2 \ec C_F^2 \|V\|_X^2,\quad V=(\bfE_V^\top,\bfH_V^\top)^\top,
\end{align*}
and linearity of $F$. A straightforward explicit calculation of the curl operator shows that
\begin{align*}
    \|AF(t,V)\|_X^2
    &= \left\| \begin{pmatrix}
        \varepsilon^{-1} \nabla \times (\sigma_2(\cdot,t)\bfH_V)\\-\mu^{-1}\nabla \times (\sigma_1(\cdot,t) \bfE_V)
    \end{pmatrix}\right\|_X^2\\
    &\le \kappa^{-2} \int_\calO \mu \|\nabla \times (\sigma_1(\cdot,t) \bfE_V)\|_{L^2(\calO)^3}^2 + \varepsilon \|\nabla \times(\sigma_2(\cdot,t)\bfH_V)\|_{L^2(\calO)^3}^2\dx\\
    &\le 3\kappa^{-2}\left(C_F^2 \|AV\|_X^2+2\max_{j=1,2}\max_{i=1,2,3} \|\partial_{x_i}\sigma_j\|_\infty^2\|V\|_X^2\right).
\end{align*}
We conclude linear growth of $F$ on $Y \ce D(A)$ by
\begin{align*}
    \|F(t,V)\|_{D(A)}^2 &= \|F(t,V)\|_X^2+\|AF(t,V)\|_X^2\\
    &\le \left(\max\{1,3\kappa^{-2}\}C_F^2+6\kappa^{-2}\max_{j=1,2}\max_{i=1,2,3} \|\partial_{x_i}\sigma_j\|_\infty^2\right) \|V\|_{D(A)}^2.
\end{align*}
As noise $G(V)$, where $V=(\bfE_V^\top,\bfH_V^\top)^\top\in L^2(\calO)^6$, we consider the Nemytskij map associated to $\diag((-\varepsilon^{-1}\bfE_V^\top,-\mu^{-1}\bfH_V^\top))Q^{1/2}$, i.e., for $h\in L^2(\calO)^6$ and $x \in \calO$, we have
\begin{equation}
\label{eq:defGMaxwell}
    (G(V)h)(x)=\begin{pmatrix}
        -\varepsilon^{-1}(x)\diag(\bfE_V(x))&0\\
        0&-\mu^{-1}(x)\diag(\bfH_V(x))
    \end{pmatrix}
    (Q^{1/2}h)(x) \in \R^6.
\end{equation}
Since for $V_1, V_2 \in L^2(\calO)^6$,
\begin{align*}
    \|G(V_1-V_2)\|_{\calL_2(H,X)}\le \kappa^{-1}\|Q^{1/2}\|_{\calL_2(H,X)} \|V_1-V_2\|_X,
\end{align*}
$G:X \to \calL_2(H,X)$ is Lipschitz continuous on $X$.
As discussed in \cite[p.5]{CCHS20},
$G$ is of linear growth on $D(A)$ under higher regularity assumptions on $Q^{1/2}$. To be precise, if $Q^{1/2} \in \calL_2(L^2(\calO)^6,H^{1+\beta}(\calO)^6)$ for some $\beta>\frac{3}{2}$, then, for some $C \ge 0$,
\begin{align*}
    \|G(V)\|_{\calL_2(H,D(A))}
    &\le C \|Q^{1/2}\|_{\calL_2(L^2(\calO)^6,H^{1+\beta}(\calO)^6)}(1+\|V\|_{D(A)}).
\end{align*}
This directly follows from the estimate \cite[formula (7)]{CCHS20} for $\mathbb{G}$ defined by $G=\mathbb{G}Q^{1/2}$ taking into account that for an orthonormal basis $(e_l)_{l\in\N}$ of $H$, we have
\begin{align*}
    \|G(V)\|_{\calL_2(H,D(A))} &= \sum_{l \in \N} \|G(V)e_l\|_{D(A)} = \sum_{l \in \N} \|\mathbb{G}(V)Q^{1/2}e_l\|_{D(A)}= \|\mathbb{G}(V)\|_{\calL_2(Q^{1/2}H,D(A))}.
\end{align*}
The choice of the coefficient $\beta>\frac{3}{2}$ stems from the fact that the Sobolev embedding $H^\beta(\calO) \hra L^\infty(\calO)$ holds for $\beta>\frac{d}{2}=\frac{3}{2}$ since $\calO \seq \R^3$ \cite[Ex.~9.3.4]{AnalysisBanachSpacesII}. Thus, for the embedding into $D(A)$ to hold, $Q^{1/2}$ is required to map into $H^{1+\beta}(\calO)^6$.

\begin{theorem}
    Let $p \in [2,\infty)$ and $F,G$ as introduced in \eqref{eq:defFMaxwell} and \eqref{eq:defGMaxwell}, respectively. Suppose that
    $u_0 \in L_{\calF_0}^p(\Omega;D(A))$ and $Q^{1/2} \in \calL_2(L^2(\calO)^6,H^{1+\beta}(\calO)^6)$ for some $\beta >\frac{3}{2}$. Denote by $U$ the mild solution to the stochastic Maxwell's equations \eqref{eq:Maxwell} with multiplicative noise \eqref{eq:linearSchroedingerPotentialMultNoise}
    and by $(U^j)_{j=0,\ldots,N_k}$ the temporal approximations as defined in \eqref{eq:defUj} obtained with the exponential Euler method $R \ce S$. Then there exists a constant $C \ge 0$ depending on $(\sigma_1,\sigma_2,u_0,T,p,\alpha, \varepsilon,\mu, \kappa)$ such that for $N_k \geq 2$
    \begin{equation*}
        \left\| \max_{0 \le j \le N_k} \|U(t_j)-U^j\|_{H^\sigma} \right\|_p \le C\big(1+ \|Q^{1/2}\|_{\calL_2(L^2(\calO)^6,H^{1+\beta}(\calO)^6)}\big) k^{1/2},
    \end{equation*}
    i.e., the approximations $(U^j)_j$ converge at rate $\frac{1}{2}$ as $k \to 0$.
\end{theorem}
\begin{proof}
    The theorem follows from Corollary \ref{cor:expEulerMultiplicative} with $\alpha=\frac{1}{2}$ and $Y=D(A)$. From the above considerations, it follows that the conditions on $F$ and $G$ are met. It remains to verify that $Y$ is Hilbert and $(S(t))_{t \ge 0}$ is a contraction semigroup on both $X$ and $Y$. Since $Y=D(A)$ is a Banach space \cite[p. 410]{MaxwellBanachDA} and $\lambda-A$ defines an isomorphism between $D(A)$ and $X$ for $\lambda \in \rho(A)$, it is also a Hilbert space. By \cite[Formula (3)]{CCHS20}, $(S(t))_{t \ge 0}$ is a contraction semigroup on $X$. By definition of the graph norm, this implies contractivity on $D(A)$.
\end{proof}

We can extend \cite[Thm.~3.3]{CCHS20} to schemes involving rational approximations.

\begin{theorem}
    Let $p \in [2,\infty)$ and $F,G$ as introduced in \eqref{eq:defFMaxwell} and \eqref{eq:defGMaxwell}, respectively. Suppose that
    $u_0 \in L_{\calF_0}^p(\Omega;D(A))$ and $Q^{1/2} \in \calL_2(L^2(\calO)^6,H^{1+\beta}(\calO)^6)$ for some $\beta >\frac{3}{2}$. Let $(R_{k})_{k>0}$ be a time discretisation scheme which is contractive on $L^2(\calO)^6$ and $D(A)$. Assume $R$ approximates $S$ to order $\frac{1}{2}$ on $D(A)$. Denote by $U$ the mild solution to the stochastic Maxwell's equations \eqref{eq:Maxwell} with multiplicative noise \eqref{eq:linearSchroedingerPotentialMultNoise} and by $(U^j)_{j=0,\ldots,N_k}$ the temporal approximations as defined in \eqref{eq:defUj}. Then there exists a constant $C \ge 0$ depending on $(\sigma_1,\sigma_2,u_0,T,p,\alpha, \varepsilon,\mu, \kappa)$ such that for $N_k \ge 2$
    \begin{equation*}
        \left\| \max_{0 \le j \le N_k} \|U(t_j)-U^j\|_{H^\sigma} \right\|_p \le C \big(1+\|Q^{1/2}\|_{\calL_2(L^2(\calO)^6,H^{1+\beta}(\calO)^6)} \big) \sqrt{\log(T/k)} k^{1/2},
    \end{equation*}
     i.e., the approximations $(U^j)_j$ converge at rate $1/2$ up to a logarithmic correction factor as $k \to 0$. In particular, rate $\frac12$ is attained for the implicit Euler method and the Crank--Nicolson method.
\end{theorem}

\section{Convergence rates for abstract wave equations}
\label{sec:rateOfConvergenceWave}

In this section, we shall be concerned with rates of convergence for abstract stochastic wave equations of the form
\begin{equation}
\label{eq:StEvolEqnWave}
    \rmd U = (AU + \bfF(t,U))\,\rmd t + \bfG(t,U)\,\rmd W_H(t),~~U(0)=U_0=(u_0,v_0) \in L^p(\Omega;X)
\end{equation}
on a phase space $X=V\times V_{-1}$ of product structure to be specified later, which takes different regularities of the first and second components of the mild solution into account. We achieve the following convergence rates for sufficiently regular noise:
\begin{itemize}
\item ${\rm E}_{k}^{\infty}\lesssim k^{\alpha} \sqrt{\log(T/k)}$ with $\alpha$ close to one (general contractive schemes, multiplicative noise);
\item ${\rm E}_{k}^{\infty}\lesssim k$ (exponential Euler, multiplicative noise).
\end{itemize}
Up to a logarithmic factor, these rates are optimal for the given problem. They provide an alternative proof of \cite[Thm.~3.1]{Wang15} for the exponential Euler method under less regularity assumptions on $\bfF$ and $\bfG$ and without making use of the group structure of the semigroup. The latter is crucial in order to extend the convergence result beyond the exponential Euler method. We extend the convergence result to general contractive schemes, which, to the best of our knowledge, is novel.

At the heart of our proof lies the higher Hölder continuity of the first component of the mild solution in $V$ compared to the mild solution vector in $X$, which emerges from the product structure of the phase space on which the abstract wave equation is considered. This allows for better estimates of those error terms depending on the Hölder continuity of the mild solution. Incorporating this into the setting of Section \ref{sec:rateOfConvergenceMultNoise} leads to the main Theorem \ref{thm:convergenceRateWave} in Subsection \ref{subsec:GeneralSchemesWave}. Subsection \ref{subsec:expEulerWave} covers the exponential Euler method. An extension of the error estimates to the full time interval is presented in Subsection\ref{subsec:fullTimeIntervalWave}.
The results are illustrated for the stochastic wave equation with trace class noise, space-time white noise, and smooth noise in Subsections \ref{subsec:traceclassnoisewave} to \ref{subsec:smoothnoisewave}.

Let $V$ be a separable Hilbert space equipped with the norm $\|\cdot \|_{V}$. Consider a densely defined, positive self-adjoint invertible operator $\Lambda: D(\Lambda) \seq V \to V$. For $\beta \in \R$, define the norm $\|u\|_{V_\beta} \ce \|\Lambda^{\beta/2}u\|_{V}$ for $u \in V_\beta$ and, for $\beta \ge 0$, denote the domain of $\Lambda^{\frac{\beta}{2}}$ by $V_\beta$ and equip it with this norm. For negative $\beta$, we denote by $V_\beta$ the completion of $V$ with respect to $\|\cdot\|_{V_\beta}$. We can thus interpret $\Lambda$ as an operator mapping from $V_1$ to $V_{-1}$ and it holds that $V=V_0$. In this section, we consider stochastic evolution equations on the phase space $X \ce V_0 \times V_{-1} =V \times V_{-1}$. More generally, we introduce the product spaces
\begin{equation}
\label{eq:defXbeta}
    X_\beta \ce V_\beta \times V_{\beta-1} = D(\Lambda^\frac{\beta}{2}) \times D(\Lambda^\frac{\beta-1}{2})
\end{equation}
for $\beta \in \R$, equipped with the norm $\|U\|_{X_\beta} \ce (\|u\|_{V_\beta}^2+\|v\|_{V_{\beta-1}}^2)^{1/2}$ for $U=(u,v) \in X_\beta$. Clearly, it then holds that $X=X_0$.

The stochastic evolution equation \eqref{eq:StEvolEqnWave} depends on the nonlinearity $\bfF: \Omega \times [0,T] \times X \to X$ and the multiplicative noise $\bfG: \Omega \times [0,T] \times X \to \gHX$ on the phase space $X$.
However, the product structure of $X$ considered in this section motivates an interpretation of \eqref{eq:StEvolEqnWave} as a system of two evolution equations. Setting
    \begin{equation}
    \label{eq:defAbfFG_wave}
        A = \begin{pmatrix} 0&I\\-\Lambda&0 \end{pmatrix},\quad\bfF(t,U) = \begin{pmatrix} 0\\F(t,u) \end{pmatrix},\quad \bfG(t,U) = \begin{pmatrix} 0\\G(t,u) \end{pmatrix}\quad \text{ for }
        U = \begin{pmatrix} u\\v \end{pmatrix} \in X
    \end{equation}
gives rise to the system of evolution equations
\begin{align*}
    \bigg\{\begin{split}
        \rmd u &= v \;\rmd t,\\
        \rmd v &= (-\Lambda u + F(t,u)) \;\rmd t + G(t,u) \;\rmd W_H(t).
    \end{split}
\end{align*}
This precisely captures the setting of stochastic wave equations when thinking of $v(t)$ as the derivative of $u(t)$, thus yielding a stochastic evolution equation for the derivative $\dot{u}(t)$ with left-hand side $\rmd \dot{u}$. The invertibility of $\Lambda$ does not lead to restrictions, because we can always reduce to this case by writing $-\Lambda u + F(t,u)= -(\Lambda +\varepsilon) u + \varepsilon u +F(t,u)$ without changing the properties of $F$.

The operator $A$ from \eqref{eq:defAbfFG_wave} generates a $C_0$-semigroup $(S(t))_{t \ge 0}$ given by
 \begin{equation}
 \label{eq:defS_wave}
    \quad S(t)= \begin{pmatrix} \cos(t\Lambda^{1/2})&\Lambda^{-1/2}\sin(t\Lambda^{1/2})\\-\Lambda^{1/2}\sin(t\Lambda^{1/2})&\cos(t\Lambda^{1/2}) \end{pmatrix},
\end{equation}
where we use the spectral theorem for self-adjoint operators to define the matrix entries. Indeed,
\begin{equation*}
    \lim_{t \to 0} \|\cos(t \Lambda^{1/2}) x - x\| = \lim_{t \to 0} \Big\|\int_0^t \sin(s \Lambda^{1/2}) \Lambda^{1/2}x \;\rmd s\Big\| \leq \lim_{t \to 0} t \|\Lambda^{1/2}x\| =0
\end{equation*}
and, analogously, $\lim_{t \to 0} \|\pm\Lambda^{\mp 1/2}\sin(t\Lambda^{1/2})x-x\|=0$ for $x\in D(\Lambda^{1/2})$. Strong continuity of the semigroup follows by the density of $D(\Lambda^{1/2})$, and the spectral theorem. It is straightforward to see that $S$ satisfies the semigroup property and that $A$ is its infinitesimal generator. Due to $-\Lambda u \in V_{-1}$ if and only if $u \in V_1$, we find that the domain of $A$ is given by
\begin{equation*}
    D(A)=\{U \in X: AU \in X\} = \{(u,v) \in X: (v,-\Lambda u) \in V_0 \times V_{-1}\} = X_1.
\end{equation*}
Let $\beta\in \R$. Combining the respective one-dimensional statements with the spectral theorem, we obtain that $\sin(t\Lambda^{1/2})$ and $\cos(t\Lambda^{1/2})$ are contractive on $V_{\beta}$, $\sin(0\cdot\Lambda^{1/2})=0$, and that $\Lambda$ and powers thereof commute with both $\sin(t\Lambda^{1/2})$ and $\cos(t\Lambda^{1/2})$. The trigonometric identity satisfied by $\sin(t\Lambda^{1/2})$ and $\cos(t\Lambda^{1/2})$ implies contractivity of the semigroup, that is,
\begin{equation}
\label{eq:ScontractiveWave}
    \|S(t)U\|_{X_{\beta}} \le \|U\|_{X_{\beta}}.
\end{equation}

Our aim is to derive conditions on $F$ and $G$ rather than $\bfF$ and $\bfG$ under which the temporal approximations
\begin{equation}
\label{eq:defUjWave}
    U^j = R_k^j U_0 + k \sum_{i=0}^{j-1} \bfF(t_i,U^i)+\sum_{i=0}^{j-1} \Delta W_{i+1}R_k^{j-i}\bfG(t_i,U^i),\quad 0 \le j \le N_k,
\end{equation}
converge to the mild solution $U(t)=(u(t),v(t)) \in X$ at a certain rate. As will become apparent, rates of convergence $>1/2$ can be attained up to a logarithmic correction factor even for general contractive schemes. The key aspect of our main theorem, Theorem \ref{thm:convergenceRate}, enabling this optimal rate consists of higher-order Hölder continuity of the first component of the mild solution.

\subsection{General contractive time discretisation schemes}
\label{subsec:GeneralSchemesWave}

As will be shown, the following assumptions on $F$ and $G$ imply that $\bfF$ and $\bfG$ fall within the scope of Section \ref{sec:rateOfConvergenceMultNoise}.

\begin{assumption}
\label{ass:FG_wave}
     Let $V$ be a Hilbert space, $\Lambda:D(\Lambda) \seq V \to V$ a densely defined, positive, self-adjoint, and invertible operator, and $p \in [2,\infty)$. Let $F:\Omega \times [0,T] \times V \to V_{-1}$, $F(\omega, t,x)=\tilde{F}(\omega,t,x)+f(\omega,t)$ and $G:\Omega \times [0,T] \times V \to \calL_2(H,V_{-1})$, $G(\omega, t,x)=\tilde{G}(\omega,t,x)+g(\omega,t)$ be strongly $\calP \otimes \calB(V)$-measurable, and such that $\tilde{F}(\cdot,\cdot,0) = 0$ and $\tilde{G}(\cdot,\cdot,0) = 0$, and suppose that for some $\delta >0$ and $\alpha \in (0,1]$,
    \begin{enumerate}[label=(\alph*)]
        \item \emph{(Lipschitz continuity from $V$ to $V_{-1}$)} there exist constants $C_F, C_G \ge 0$ such that for all $\omega \in \Omega, t \in [0,T]$ and $x,y\in V$, it holds that
        \begin{align*}
            \|\tilde{F}(\omega,t,x)-\tilde{F}(\omega,t,y)\|_{V_{-1}} &\le C_F\|x-y\|_{V},\\
            \|\tilde{G}(\omega,t,x)-\tilde{G}(\omega,t,x)\|_{\calL_2(H,V_{-1})}
            &\le C_G\|x-y\|_{V},
        \end{align*}
        \item \emph{(Hölder continuity with values in $V_{-1}$)} there are constants $C_{\alpha,F}, C_{\alpha,G} \ge 0$ such that
        \begin{align*}
            \sup_{\omega\in \Omega, x \in V} [\Lambda^{-\frac{1}{2}}F(\omega,\cdot,x)]_\alpha \le C_{\alpha,F},~\sup_{\omega\in \Omega, x \in V} [\Lambda^{-\frac{1}{2}}G(\omega,\cdot,x)]_\alpha \le C_{\alpha,G},
        \end{align*}
        \item \emph{(continuity with values in $V_{\delta-1}$)} $f \in L^p_\calP(\Omega; C([0,T];V_{\delta-1}))$, and $g \in L^p_\calP(\Omega; C([0,T];\calL_2(H,V_{\delta-1})))$,
        \item \emph{(invariance)} \label{item:waveAssYinvariance} $F:\Omega \times [0,T] \times V_{\delta} \to V_{\delta-1}$ and $G:\Omega \times [0,T] \times V_{\delta} \to \calL_2(H,V_{\delta-1})$ are strongly $\calP \otimes \calB(V_\delta)$-measurable,
        \item \emph{(linear growth from $V_{\delta}$ to $V_{\delta-1}$)} \label{item:linGrowthWaveAss}
            there exist constants $L_F, L_G \ge 0$ such that for all $\omega \in \Omega$, $t \in [0,T]$ and $x\in V$, it holds that
        \begin{align*}
            \|\tilde{F}(\omega,t,x)\|_{V_{\delta-1}} &\le L_F(1+\|x\|_{V_{\delta}}),\\
            \|\tilde{G}(\omega,t,x)\|_{\calL_2(H,V_{\delta-1})} &\le L_G(1+\|x\|_{V_{\delta}}).
        \end{align*}
    \end{enumerate}
\end{assumption}
It is important to note that both $\delta\in (0,1]$ and $\delta\in (1, 2]$ will be considered. As for $\delta=2$, optimal rates are obtained for the usual schemes, larger values of $\delta$ are not considered.

Next, we first show that we satisfy the required conditions for the well-posedness and thus \eqref{eq:StEvolEqnWave} has a unique mild solution. Adopt the notation of the proof of Theorem \ref{thm:convergenceRate}, replacing $F, \tilde{F}, f, G, \tilde{G}$ and $g$ by $\bfF, \tilde{\bfF}, \mathbf{f}, \bfG, \tilde{\bfG}$ and $\mathbf{g}$, respectively.

Setting $Y\ce X_\delta$ for some $\delta \ge \alpha$, it is clear from $X=X_0$, invertibility of $\Lambda$, and $D(A^{n})=X_{n}$ that $Y \hra X$ and $Y \hra D_A(\beta,\infty)$ for any $\beta\in (0,\delta)$. Since $V_\delta$ are separable Hilbert spaces for $\delta \in \R$, so are $X$ and $Y$. Contractivity of the semigroup follows from \eqref{eq:ScontractiveWave}. Note that strong $\calP \otimes \calB(X)$-measurability of $\bfF$ and $\bfG$, and that $\tilde{\bfF}, \tilde{\bfG}$ vanish in $0$ immediately follow from the respective assumptions on $\tilde{F}, \tilde{G}$ due to the structure \eqref{eq:defAbfFG_wave}. We are left to prove Lipschitz continuity, linear growth, $Y$-invariance, and Hölder continuity of $\bfF,\bfG$, and continuity of $\mathbf{f}$ and $\mathbf{g}$. Deducing $Y$-invariance from Assumption \ref{ass:FG_wave}  is straightforward noting that
\begin{align}
\label{eq:fnormEstWave}
    \|\mathbf{f}\|_{p,\infty,Y} = \bigg\|\sup_{t \in [0,T]}\|\mathbf{f}(t)\|_Y\bigg\|_p = \bigg\|\sup_{t \in [0,T]}\|f(t)\|_{V_{\delta-1}}\bigg\|_p = \|f\|_{p,\infty,V_{\delta-1}}
\end{align}
and, likewise, $\nn\mathbf{g}\nn_{p,\infty,Y}=\nn g\nn_{p,\infty,V_{\delta-1}}$. The mapping properties on $Y$ and strong $\calP \otimes \calB(Y)$-measurability of $\bfF$ and $\bfG$ follow from Assumption \ref{ass:FG_wave}\ref{item:waveAssYinvariance} because $Y=V_\delta \times V_{\delta-1}$. Linear growth of $\tilde{\bfF}$ from $Y$ to $Y$ follows from linear growth of $\tilde{F}$ from $V_{\delta}$ to $V_{\delta-1}$ as stated in Assumption \ref{ass:FG_wave} taking the structure \eqref{eq:defAbfFG_wave} of $\bfF$ into account via
    \begin{align*}
        \|\tilde{\bfF}(t,U)\|_{Y} &=  \|\tilde{F}(t,u)\|_{V_{\delta-1}}
        \le L_{F}(1+\|u\|_{V_{\delta}}) \le L_{F}(1+\|U\|_{Y})
    \end{align*}
    for $t \in [0,T]$, $U=(u,v) \in Y=V_\delta \times V_{\delta-1}$. Analogously, linear growth of $\tilde{\bfG}$ from $Y$ to $\gHY$ is obtained, since
     \begin{align*}
        \|\tilde{\bfG}(t,U)\|_{Y} = \|\tilde{G}(t,u)\|_{V_{\delta-1}} &\le L_G(1+\|u\|_{V_{\delta}}) \le L_G(1+\|U\|_{Y}).
    \end{align*}
    Lipschitz continuity of $\bfF$ from $X$ to $X$ holds due to
      \begin{align*}
        \|\bfF(t,U_1)-\bfF(t,U_2)\|_{X} &=  \|F(t,u_1)-F(t,u_2)\|_{V_{-1}} = \|\Lambda^{-\frac{1}{2}}[\tilde{F}(t,u_1)-\tilde{F}(t,u_2)]\|_{V}\\
        &\le C_F\|u_1-u_2\|_{V} \le C_F \|U_1-U_2\|_X
    \end{align*}
    for $t \in [0,T]$ and $U_1=(u_1,v_1), U_2=(u_2,v_2) \in X$. Analogously,
    \begin{align*}
        \|\bfG(t,U_1)-\bfG(t,U_2)\|_{\gHX} &= \|\Lambda^{-\frac{1}{2}}[\tilde{G}(t,u_1)-\tilde{G}(t,u_2)]\|_{\calL_2(H,V)}\le C_G \|U_1-U_2\|_X.
    \end{align*}
    Hence, $\bfG:X \to \gHX$ is Lipschitz continuous. Via the same argument,
    \begin{equation*}
        [\bfF(\omega,\cdot,U)]_\alpha = \sup_{0 \le s \le t \le T} \frac{\|\bfF(t,U)-\bfF(s,U)\|_X}{(t-s)^\alpha} = \sup_{0 \le s \le t \le T} \frac{\|\Lambda^{-\frac{1}{2}}[F(t,u)-F(s,u)]\|_{V}}{(t-s)^\alpha},
    \end{equation*}
    from which we conclude $\alpha$-Hölder continuity of $\bfF$.

The above leads to:
\begin{lemma}[Well-posedness]\label{lem:well-posednesswave}
    Suppose that Assumption \ref{ass:FG_wave} holds for some $\alpha \in (0,1]$, $\delta \ge \alpha$, and $p \in [2,\infty)$. Let $Y \ce X_\delta$ as defined in \eqref{eq:defXbeta} and $U_0 \in L_{\calF_0}^p(\Omega;Y)$.
    Under these conditions there exists a unique mild solution $U\in L^p(\Omega;C([0,T];X))$ to \eqref{eq:StEvolEqnWave}. Furthermore, it is in $L^p(\Omega;C([0,T];Y))$ and
    \begin{align*}
        \|U\|_{L^p(\Omega;C([0,T];Y))}\leq \Cbdd^Y\Big(&1+\|U_0\|_{L^p(\Omega;Y)}+ \|f\|_{L^p(\Omega;L^1(0,T;V_{\delta-1}))}\\
        &+B_{p}\|g\|_{L^p(\Omega;L^2(0,T;\calL_2(H,V_{\delta-1})))}\Big),
    \end{align*}
    where $\Cbdd^Y \ce (1+C^2 T)^{1/2}\ee^{(1+C^2T)/2}$ with
    $C \ce L_{F}T^{1/2}+B_{p}L_{G}$, and $B_{p}$ is the constant from Theorem \ref{thm:maxIneqQuasiContractive}.
\end{lemma}

As established in \eqref{eq:WPboundCu0fgZ}, the well-posedness on $Z\in \{X,Y\}$ implies
\begin{equation*}
    1+\bigg\|\sup_{r \in [0,T]} \|U(r)\|_Z\bigg\|_p \le C_{U_0,\mathbf{f},\mathbf{g},Z} < \infty
\end{equation*}
with $C_{U_0,\mathbf{f},\mathbf{g},Z}$ as defined in \eqref{eq:defCu0fgZ}. In the abstract wave equation setting, the constant simplifies to
\begin{equation}
\label{eq:defCu0fgZwave}
    C_{U_0,\mathbf{f},\mathbf{g},Z} = 1+ \Cbdd^Z(1+\|U_0\|_{L^p(\Omega;Z)}+\|f\|_{p,1,Z_2}+\nn g\nn_{p,2,Z_2}),
\end{equation}
where $\Cbdd^Z$ denotes the constant from Lemma \ref{lem:well-posednesswave}, $Z_2\ce V_{-1}$ if $Z=X$, and $Z_2 \ce V_{\delta-1}$ if $Z=Y$.

\begin{lemma}[Stability]\label{lem:stabilitywave}
    Suppose that Assumption \ref{ass:FG_wave} holds for some $\alpha \in (0,1]$, $\delta \ge \alpha$, and $p \in [2,\infty)$. Let $Y \ce X_\delta$ as defined in \eqref{eq:defXbeta} and $U_0 \in L_{\calF_0}^p(\Omega;Y)$. Let $(R_{k})_{k>0}$ be a time discretisation scheme which is contractive on $X$ and $Y$, and let $N_k \geq 2$.
    Then the temporal approximations $(U^j)_{j=0,\ldots,N_k}$ obtained via \eqref{eq:defUjWave} are stable on both $X$ and $Y$. That is, for $Z \in \{X,Y\}$,
        \[1+\left\| \max_{0 \le j \le N_k} \|U^j\|_Z\right\|_p \le \Cstab^Z c_{U_0, f, g,T,Z},\]
    where $\Cstab^Z \ce (1+C_Z^2T)^{1/2}e^{(1+C_Z^2T)/2}$ with $C_X\ce C_F T^{1/2}+B_pC_G$, $C_Y \ce L_FT^{1/2}+B_pL_G$,
    \begin{align*}
    c_{U_0, f, g,T,Z}\ce    1+\|U_0\|_{L^p(\Omega;Z)} + \|f\|_{L^p(\Omega;C([0,T];Z_2))} T + \|g\|_{L^p(\Omega;C([0,T];\calL_2(H,Z_2)))} B_p T^{1/2},
    \end{align*}
    $Z_2\ce V_{-1}$ if $Z=X$, $Z_2\ce V_{\delta-1}$ if $Z=Y$, and $B_{p}$ is the constant from Theorem \ref{thm:maxIneqQuasiContractive}.
\end{lemma}

We denote
\begin{equation}
\label{eq:KufgYwave}
     K_{U_0,f,g,Y} \ce \Cstab ^Y c_{U_0,f,g,T,Y} = \Cstab^Y (1+\|U_0\|_{L^p(\Omega;Y)} + \|f\|_{p,\infty,V_{\delta-1}} T + \nn g\nn_{p,\infty,V_{\delta-1}} B_p T^{1/2})
\end{equation}
so that $K_{U_0,f,g,Y}=\KUffggY$ with $\KUffggY$ as defined in \eqref{eq:defKu0fgY}.

For future estimates, it is useful to know the decay of differences of the sine and cosine operators $\sin(t\Lambda^{1/2})$ and $\cos(t\Lambda^{1/2})$. We include a short proof for the convenience of the reader.
\begin{lemma}
\label{lem:sinCosEst}
    Let $t \in [0,T]$. Then for all $\alpha \in [0,1]$, we have
    \begin{align*}
        \|\Lambda^{-\frac{\alpha}{2}}[\sin(t\Lambda^{1/2})-\sin(s\Lambda^{1/2})]\|_{\calL(V)} &\le 2 (t-s)^\alpha,\\
        \|\Lambda^{-\frac{\alpha}{2}}[\cos(t\Lambda^{1/2})-\cos(s\Lambda^{1/2})]\|_{\calL(V)} &\le 2 (t-s)^\alpha
    \end{align*}
    for all $0 \le s \le t \le T$.
\end{lemma}
\begin{proof}
    The statement is trivially fulfilled for $t=s$. Let $0\leq s<t\leq T$. We claim that
    \begin{align*}
        \zeta_{\alpha}(t,s)\ce \frac{|\sin(t) - \sin(s)|}{|t-s|^{\alpha}}\leq 2.
    \end{align*}
    Indeed, if $|t-s|\leq 1$, then by the mean value theorem $\zeta_{\alpha}(t,s)\leq \zeta_{1}(t,s) \leq 1$.  If $|t-s|>1$, then $\zeta_{\alpha}(t,s)\leq 2$.
    Now let $\lambda>0$. Applying the claim with $t\lambda^{1/2}$ and $s\lambda^{1/2}$ gives
    \begin{align*}
        \lambda^{-\alpha/2} |\sin(t \lambda^{1/2}) - \sin(s \lambda^{1/2})|\leq 2 {|t-s|^{\alpha}}.
    \end{align*}
    Thus by the spectral theorem for self-adjoint operators and positivity of $\Lambda$, we get the desired statement. The statement for the cosine is proven analogously.
\end{proof}

While the mild solution $U$ has at most $1/2$-Hölder continuous paths as follows from Lemma \ref{lem:mildsolest}, the product structure of the stochastic evolution equation results in higher Hölder continuity of the first component $u$ of $U$, as the following lemma illustrates. In particular, $u$ has Lipschitz continuous paths for sufficiently regular $F$ and $G$.

\begin{lemma}
\label{lem:mildsolestWave}
    Suppose that Assumption \ref{ass:FG_wave} holds for some $\alpha \in (0,1]$, $\delta\geq \alpha$, and $p \in [2,\infty)$. Let $X \ce X_0$ and $Y \ce X_\delta$ as defined in \eqref{eq:defXbeta} and $U_0 \in L_{\calF_0}^p(\Omega; Y)$. Then for all $0 \le s \le t \le T$, the first component $u$ of the mild solution $U$ of \eqref{eq:StEvolEqnWave} satisfies
    \begin{equation*}
        \|u(t)-u(s)\|_{L^p(\Omega;V)} \le L(t-s)^\alpha
    \end{equation*}
    with constant
    \begin{equation*}
        L \ce 2 C_Y\bigg[\sqrt{2}\|U_0\|_{L^p(\Omega;Y)}+L_{1,F}T\frac{\alpha+2}{\alpha+1}+B_pL_{2,G}T^{1/2}\Big(1+\frac{1}{\sqrt{2\alpha+1}}\Big) \bigg],
    \end{equation*}
    where $L_{1,F} \ce L_F \CUffggY+\|f\|_{L^p(\Omega; L^\infty(0,T;V_{\delta-1}))}$, $L_{2,G} \ce L_G \CUffggY+\|g\|_{L^p(\Omega;L^\infty(0,T;\calL_2(H,V_{\delta-1})))}$ with $\CUffggY$ as in \eqref{eq:defCu0fgZwave}, $\CY$ denotes the embedding constant of $X_\delta$ into $X_\alpha$, and $B_p$ is the constant from Theorem \ref{thm:maxIneqQuasiContractive}.
\end{lemma}
\begin{proof}
    From the structure \eqref{eq:defS_wave} of the semigroup as well as \eqref{eq:defAbfFG_wave} of $\bfF$ and $\bfG$, we deduce the following variation-of-constants formula for the first component of the mild solution.
    \begin{align*}
        u(t)&=\cos(t\Lambda^{1/2})u_0+\Lambda^{-\frac{1}{2}}\sin(t\Lambda^{1/2})v_0+\int_0^t \Lambda^{-\frac{1}{2}}\sin((t-r)\Lambda^{1/2})F(r,u(r))\,\rmd r\\
        &\phantom{= }+ \int_0^t \Lambda^{-\frac{1}{2}}\sin((t-r)\Lambda^{1/2})G(r,u(r))\,\rmd W_H(r).
    \end{align*}
    Hence, the difference can be split up as
     \begin{align*}
        &\|u(t)-u(s)\|_{L^p(\Omega;V)} \le \big\|[\cos(t\Lambda^{1/2})-\cos(s\Lambda^{1/2})]u_0+\Lambda^{-\frac{1}{2}}[\sin(t\Lambda^{1/2})-\sin(s\Lambda^{1/2})]v_0\big\|_{L^p(\Omega;V)}\\
        &+ \Big\|\int_0^{s} \|\Lambda^{-\frac{1}{2}}[\sin((t-r)\Lambda^{1/2})-\sin((s-r)\Lambda^{1/2})]F(r,u(r))\|_{V}\;\rmd r\Big\|_p\\
        &+ \Big\|\int_{s}^t \|\Lambda^{-\frac{1}{2}}\sin((t-r)\Lambda^{1/2})F(r,u(r))\|_{V}\;\rmd r\Big\|_p\\
        & + \Big\|\int_0^{s} \Lambda^{-\frac{1}{2}}[\sin((t-r)\Lambda^{1/2})-\sin((s-r)\Lambda^{1/2})]G(r,u(r))\;\rmd W_H(r)\Big\|_{L^p(\Omega;V)}\\
        & + \Big\|\int_{s}^t \Lambda^{-\frac{1}{2}}\sin((t-r)\Lambda^{1/2})G(r,u(r))\;\rmd W_H(r)\Big\|_{L^p(\Omega;V)} \ec E_1+E_2+E_3+E_4+E_5,
    \end{align*}
    where $E_\ell \ce E_\ell(t,s)$ for $1 \le \ell \le 5$. We proceed to bound these five expressions individually. Lemma \ref{lem:sinCosEst} yields
    \begin{align*}
        E_1 &\le \Big\|\|[\cos(t\Lambda^{1/2})-\cos(s\Lambda^{1/2})]\Lambda^{-\frac{\alpha}{2}}\|_{\calL(V)}\|\Lambda^{\frac{\alpha}{2}}u_0\|_{V}\\
        &\phantom{\le }+\|[\sin(t\Lambda^{1/2})-\sin(s\Lambda^{1/2})]\Lambda^{-\frac{\alpha}{2}}\|_{\calL(V)}\|\Lambda^{\frac{\alpha-1}{2}}v_0\|_{V}\Big\|_p\\
        &\le 2(t-s)^\alpha \|\|u_0\|_{V_\alpha}+\|v_0\|_{V_{\alpha-1}}\|_p \le 2\sqrt{2}\|U_0\|_{L^p(\Omega;X_\alpha)} \cdot(t-s)^\alpha\\
        &\le 2\sqrt{2} \CY \|U_0\|_{L^p(\Omega;Y)}\cdot(t-s)^\alpha,
    \end{align*}
    where we have used the embedding $Y=X_\delta \hra X_\alpha$ in the last line.
    Using the same trick of inserting $\Lambda^{-\frac{\alpha}{2}}$, applying Lemma \ref{lem:sinCosEst}, and using the embedding $V_{\delta-1} \hra V_{\alpha-1}$ as well as linear growth of $\tilde{F}$ from $V_\delta$ to $V_{\delta-1}$, we obtain
    \begin{align*}
        E_2 &\le 2 s(t-s)^\alpha \bigg\|\sup_{r \in [0,T]}\|\Lambda^{\frac{\alpha-1}{2}}F(r,u(r))\|_{V}\bigg\|_p \le 2\CY s(t-s)^\alpha \bigg\|\sup_{r \in [0,T]}\|F(r,u(r))\|_{V_{\delta-1}}\bigg\|_p \\
        &\le 2\CY s (t-s)^\alpha\bigg(L_F\bigg(1+\bigg\|\sup_{r \in [0,T]}\|u(r)\|_{V_\delta}\bigg\|_p\bigg)+\|f\|_{p,\infty,V_{\delta-1}}\bigg)\le 2\CY L_{1,F}T(t-s)^\alpha.
    \end{align*}
    Likewise, for the stochastic integral, we conclude
    \begin{equation*}
        E_4 \le 2\CY B_p(L_G \CUffggY+\nn g\nn_{p,\infty,V_{\alpha-1}})s^\frac{1}{2}(t-s)^\alpha
        \le 2\CY B_p L_{2,G} T^\frac{1}{2}(t-s)^\alpha.
    \end{equation*}
    Recalling that $\sin(0\cdot\Lambda^{1/2})=0$, we can estimate
     \begin{align*}
        E_3 &\le \Big\|\int_s^t\|[\sin((t-r)\Lambda^{1/2})-\sin(0\cdot\Lambda^{1/2})]\Lambda^{-\frac{\alpha}{2}}\|_{\calL(V)}\|\Lambda^{\frac{\alpha-1}{2}}F(r,u(r))\|_{V} \;\rmd r\Big\|_p\\
        &\le 2 \CY \int_s^t(t-r)^\alpha\;\rmd r \bigg\|\sup_{r \in [0,T]}\|F(r,u(r))\|_{V_{\delta-1}}\bigg\|_p
        \\ & \le \frac{2 \CY L_{1,F}}{\alpha+1}(t-s)^{\alpha+1}\leq \frac{2 \CY L_{1,F} T}{\alpha+1}(t-s)^{\alpha},
    \end{align*}
    and, analogously,
    \begin{align*}
        E_5 \le \frac{2\CY B_p L_{2,G}}{\sqrt{2\alpha+1}}(t-s)^{\alpha+\frac{1}{2}}\leq \frac{2\CY B_p L_{2,G} T^{1/2}}{\sqrt{2\alpha+1}}(t-s)^{\alpha}.
    \end{align*}
    Adding the bounds for $E_1$ to $E_5$ results in the desired statement.
\end{proof}

Analogous to the considerations in Remark \ref{rem:lemmafgRegularity}, the regularity assumptions on $f$ and $g$ can be relaxed in this lemma. Having established Hölder continuity of $u$ of order up to $1$, we can derive an error bound attaining the optimal order $1$ for sufficiently good schemes and regular nonlinearity, noise and initial values. The following main theorem  of this section generalises \cite[Thm.~3.1]{Wang15} from the exponential Euler method to general contractive schemes as well as more general $F$ and $G$.

\begin{theorem}
\label{thm:convergenceRateWave}
    Suppose that Assumption \ref{ass:FG_wave} holds for some $\alpha \in (0,1]$, $\delta\geq \alpha$, and $p \in [2,\infty)$. Let $X \ce X_0$ and $Y \ce X_\delta$ as defined in \eqref{eq:defXbeta} and $U_0 \in L_{\calF_0}^p(\Omega; Y)$. Let $(R_k)_{k>0}$ be a contractive time discretisation scheme on $X$ which commutes with the resolvent of $A$. Assume $R$ approximates $S$ to order $\alpha$ on $Y$. Denote by $U$ the mild solution of \eqref{eq:StEvolEqnWave} and by $(U^j)_{j=0,\ldots,N_k}$ the temporal approximations as defined in \eqref{eq:defUjWave}.
    Then for $N_k \ge 2$
    \begin{equation*}
        \left\lVert\max_{0 \le j \le N_k}\|U(t_j)-U^j\|\right\rVert_p
        \le C_\ee\big(C_1 +C_2\sqrt{\max\{\log (T/k),p\}}\big)k^{\alpha}
    \end{equation*}
    with $C_\ee \ce (1+C^2T)^{1/2}\exp((1+C^2T)/2)$, $C \ce C_F\sqrt{T}+B_p C_G$, $C_2 \ce K C_\alpha K_G \sqrt{T}$, and
    \begin{align*}
        C_1&\ce C_\alpha \|U_0\|_{L^p(\Omega;Y)} + \Big( \frac{1}{\alpha+1}(C_FL+C_{\alpha,F}+2\CY K_F)+C_\alpha K_F\Big)T\\
        &\phantom{\ce }+\frac{B_p\sqrt{T}}{\sqrt{2\alpha+1}}(C_G L+C_{\alpha,G}+2\CY K_G),
        \end{align*}
        $K_F \ce L_F K_{U_0,f,g,Y}+\|f\|_{L^p(\Omega; C([0,T];V_{\delta-1}))}$,
        $K_G  \ce L_G K_{U_0,f,g,Y}+\|g\|_{L^p(\Omega;C([0,T];\calL_2(H,V_{\delta-1})))}$, $L$ as defined in Lemma \ref{lem:mildsolestWave}, $K_{U_0,f,g,Y}$ as in \eqref{eq:KufgYwave}, $K=4\exp(1+\frac{1}{2\mathrm{e}})$, $\CY$ denotes the embedding constant of $Y$ into $D_A(\alpha,\infty)$, and $B_p$ is the constant from Theorem \ref{thm:maxIneqQuasiContractive}.

    In particular, the approximations $(U^j)_j$ converge at rate $\min\{\alpha,1\}$ up to a logarithmic correction factor as $k \to 0$.
\end{theorem}
Possible choices for $R$ in the above include but are not limited to the exponential Euler, the implicit Euler, and the Crank--Nicolson method, as well as other $A$-stable schemes. We recall that the contractivity of a large class of schemes follows from Proposition \ref{prop:functionalcalculus}.
\begin{proof}
By the discussion before  Lemma \ref{lem:well-posednesswave}, the conditions of Theorem \ref{thm:convergenceRate} follow from Assumption \ref{ass:FG_wave}. Second, we make use of Lemma \ref{lem:mildsolestWave} to obtain decay of rate $\alpha$ for those terms limiting the rate of convergence in Theorem \ref{thm:convergenceRate}
    to $\frac{1}{2}$.

Contractivity of $S$, Lipschitz continuity of $\tilde{F}$ from $V$ to $V_{-1}$ and Lemma \ref{lem:mildsolestWave} together yield
    \begin{align*}
        M_{2,1} &\le \sum_{i=0}^{N-1} \int_{t_i}^{t_{i+1}} \|\bfF(s,U(s))-\bfF(s,U(t_i))\|_{L^p(\Omega;X)} \ds\\
        &= \sum_{i=0}^{N-1} \int_{t_i}^{t_{i+1}} \|\tilde{F}(s,u(s))-\tilde{F}(s,u(t_i))\|_{L^p(\Omega;V_{-1})} \ds \\
        &\le C_F \sum_{i=0}^{N-1} \int_{t_i}^{t_{i+1}} \|u(s)-u(t_i)\|_{L^p(\Omega;V)} \ds\le C_F L \sum_{i=0}^{N-1} \int_{t_i}^{t_{i+1}} (s-t_i)^\alpha \ds = \frac{C_F L}{\alpha+1} t_N k^\alpha.
    \end{align*}
    Combining this with the bounds for $M_{2,2}$ to $M_{2,4}$ from Theorem \ref{thm:convergenceRate} leads to
    \begin{align*}
        M_2 \le \left( \frac{C_FL+C_{\alpha,F}+2\CY K_F}{\alpha+1}+C_\alpha K_F\right)t_N k^\alpha+C_F \sqrt{t_N}\left(k \sum_{i=0}^{N-1} E(i)^2\right)^{1/2}.
    \end{align*}
    Here, we have used \eqref{eq:fnormEstWave} to pass from the $Y$-norm of $\mathbf{f}$ to the $V_{\delta-1}$-norm of $f$ appearing in $K_F$.
    For the term $M_{3,1}$, an application of the maximal inequality is required additionally. By the same reasoning as for $M_{2,1}$, we then deduce
    \begin{align*}
        M_{3,1}&\le B_p C_G \bigg(\sum_{i=0}^{N-1} \int_{t_i}^{t_{i+1}} \|u(s)-u(t_i)\|_{L^p(\Omega;V)}^2 \ds\bigg)^{1/2}
        \le \frac{B_pC_GL}{\sqrt{2\alpha+1}}\sqrt{t_N} k^{\alpha}.
    \end{align*}
    In conclusion from the bounds for $M_{3,1}$ to $M_{3,5}$,
    \begin{align*}
        M_3 &\le C_{p,\alpha,G} \sqrt{t_N}k^{\alpha}+KC_\alpha K_G\sqrt{t_N}\sqrt{\max\{\log N,p\}} k^{\alpha}+ B_pC_G\Big(k\sum_{i=0}^{N-1}E(i)^2\Big)^{1/2}
    \end{align*}
    with $C_{p,\alpha,G} \ce B_p(2\alpha+1)^{-1/2}(C_GL+C_{\alpha,G}+2\CY  K_G)$.
    The final statement follows by summing the estimates for $M_1, M_2$ and $M_3$ and then applying Gronwall's inequality from Lemma \ref{lem:KruseGronwall}.
\end{proof}

\subsection{The exponential Euler method}
\label{subsec:expEulerWave}

Also for the abstract stochastic wave equation, the logarithmic correction factor vanishes when using the exponential Euler method. Hence, we obtain convergence of the optimal rate.

\begin{corollary}
\label{cor:expEulerWave}
    Suppose that Assumption \ref{ass:FG_wave} holds for some $\alpha \in (0,1]$, $\delta\geq \alpha$, and $p \in [2,\infty)$. Let $X \ce X_0$ and $Y \ce X_\delta$ as defined in \eqref{eq:defXbeta} and $U_0 \in L_{\calF_0}^p(\Omega; Y)$. Consider the exponential Euler method $R \ce S$ for time discretisation. Denote by $U$ the mild solution of \eqref{eq:StEvolEqnWave} and by $(U^j)_{j=0,\ldots,N_k}$ the temporal approximations as defined in \eqref{eq:defUjWave}.
    Then for $N_k \geq 2$
    \begin{equation*}
        \bigg\| \max_{j=0,\ldots,N_k} \|U(t_j)-U^j\|_X\bigg\|_p \le C_{\rmS,\ee} C_{\rmS} \cdot k^\alpha
    \end{equation*}
    with constants $C_{\rmS,\ee} \ce C_\ee$ as in Theorem \ref{thm:convergenceRateWave} and
    \begin{align*}
        C_{\rmS}&\ce \frac{C_FL+C_{\alpha,F}+2\CY K_F}{\alpha+1} T + \frac{B_p \sqrt{T}}{\sqrt{2\alpha+1}} (C_G L+C_{\alpha,G}+2\CY K_G),
    \end{align*}
    where $L$ is as defined in Lemma \ref{lem:mildsolestWave}, $K_F$ and $K_G$ are as in Theorem \ref{thm:convergenceRateWave}, $\CY$ denotes the embedding constant of $Y$ into $D_A(\alpha,\infty)$, and $B_p$ is the constant from Theorem \ref{thm:maxIneqQuasiContractive}.

    In particular, the approximations $(U^j)_j$ converge at rate $\min\left\{\alpha,1\right\}$ as $k \to 0$.
\end{corollary}

\subsection{Error estimates on the full time interval}
\label{subsec:fullTimeIntervalWave}

In the same way as in the proof of Theorem \ref{thm:generalschemeuniform0T}, we see that the next result follows from Theorem \ref{thm:convergenceRateWave}.
\begin{corollary}\label{cor:waveuniform0Talpha12}
Suppose that the conditions of Theorem \ref{thm:convergenceRateWave} hold for $\alpha\in (0,1/2]$. Let $p_0 \in (p,\infty)$ and $q \in (2,\infty]$ be such that $\frac{1}{2}-\frac{1}{q}=\alpha$, and suppose that $f,g$, and $U_0$ have additional integrability
\begin{equation*}
    f \in L^{p_0}(\Omega;L^1(0,T;V)),\quad g \in L^{p_0}(\Omega;L^q(0,T;\calL_2(H,V))),\quad\text{ and }U_0 \in L_{\calF_0}^{p_0}(\Omega;X) \cap L_{\calF_0}^p(\Omega;X_\delta).
\end{equation*}
Denote by $U$ the mild solution of \eqref{eq:StEvolEqnWave} and by $(U^j)_{j=0,\ldots,N_k}$ the temporal approximations as defined in \eqref{eq:defUjWave}. Define the piecewise constant extension $\tilde{U}:[0,T] \to L^p(\Omega;X)$ of $(U^j)_{j=0,\ldots,N_k}$ by $\tilde{U}(t) \ce U^j$ for $t \in [t_j,t_{j+1})$, $0 \le j \le N_k-1$, and $\tilde{U}(T) \ce U^{N_k}$. Then for all $N_k \ge 2$ there is a constant $C \ge 0$ depending on $(T,p,p_0,\alpha,u_0,F, G,V,\delta)$ such that
\begin{equation*}
    \bigg\|\sup_{t \in [0,T]} \|U(t)-\wt{U}(t)\|_X\bigg\|_p \le C\big(1+\sqrt{\log(T/k)}\big)k^\alpha.
\end{equation*}
\end{corollary}

In case we only estimate the first component $u$, more can be said about the convergence rate on the full time interval. Under weaker integrability conditions and for general $\alpha\in (0,1]$ we obtain the following.
\begin{corollary}\label{cor:waveuniform0T}
Suppose that the conditions of Theorem \ref{thm:convergenceRateWave} hold. Define the piecewise constant extension $\tilde{U}=(\tilde{u},\tilde{v}):[0,T] \to L^p(\Omega;X)$ of $(U^j)_{j=0,\ldots,N_k}$ by $\tilde{U}(t) \ce U^j$ for $t \in [t_j,t_{j+1})$, $0 \le j \le N_k-1$, and $\tilde{U}(T) \ce U^{N_k}$. Let $\delta_1 \ce \min\{\delta, 1\}$. Then the following two error estimates hold.
\begin{enumerate}[label=(\roman*)]
\item \emph{(general schemes)} It holds that
\[\bigg\|\sup_{t \in [0,T]} \|u(t)-\tilde{u}(t)\|_V\bigg\|_{p} \le 2 C_{U_0,\textbf{f},\textbf{g},X_{\delta_1}} k^{\delta_1} + C_\ee\big(C_1 +C_2\sqrt{\log (\max\{T/k,p\})}\big)k^{\alpha}.\]
\item \emph{(exponential Euler)} If $R_k = S(k)$ then
\[\bigg\|\sup_{t \in [0,T]} \|u(t)-\tilde{u}(t)\|_V\bigg\|_{p} \le 2 C_{U_0,\textbf{f},\textbf{g},X_{\delta_1}} k^{\delta_1} + C_{\rmS,\ee} C_{\rmS} \cdot k^\alpha.\]
\end{enumerate}
\end{corollary}
\begin{proof}

Since the mild solution is also a weak solution to \eqref{eq:StEvolEqnWave}, writing $U  = (u,v)\in L^p(\Omega;C([0,T];V\times V_{-1}))$ we see that $(u(t), \varphi) - (u_0, \varphi) = \int_0^t (v(s), \varphi) \;\rmd s$ for all $\varphi\in V_{-1}$. Therefore, $u$ is continuously differentiable as a $V_{-1}$-valued function.

By \eqref{eq:WPboundCu0fgZ},
\begin{align}\label{eq:maxuuprime}
\max\{\|u\|_{L^p(\Omega;C([0,T];V_{\delta_1}))}, \|u'\|_{L^p(\Omega;C([0,T];V_{\delta_1-1}))}\} \leq \|U\|_{L^p(\Omega;C([0,T];X_{\delta_1})}\leq C_{U_0,\textbf{f},\textbf{g},X_{\delta_1}}.
\end{align}
Using the above and the interpolation estimate $\|x\|_{V}\leq \|x\|^{\delta_1}_{V_{\delta_1-1}}\|x\|^{1-\delta_1}_{V_{\delta_1}}$ we find that
\begin{align*}
\|u(t) - u(s)\|_{V} = \|u(t) - u(s)\|_{V_{\delta_1-1}}^{\delta_1} \|u(t) - u(s)\|_{V_{\delta_1}}^{1-\delta_1} \leq 2|t-s|^{\delta_1} \|u'\|_{C([0,T];V_{\delta_1-1})}^{\delta_1} \|u\|_{C([0,T];V_{\delta_1})}^{1-\delta_1}.
\end{align*}
Therefore, by  H\"older's inequality and \eqref{eq:maxuuprime} we find that
\begin{align*}
[u]_{L^p(\Omega;C^{\delta_1}([0,T];V))}\leq \|u'\|_{L^p(\Omega;C([0,T];V_{\delta_1-1}))}^{\delta_1} \|u\|_{L^p(\Omega;C([0,T];V_{\delta_1}))}^{1-\delta_1}
\leq 2 C_{U_0,\textbf{f},\textbf{g},X_{\delta_1}}.
\end{align*}
By Lemma \ref{lem:splitContErrorPhi}, we find that for $U^j=(u^j,v^j)$,
\begin{align*}
\sup_{t \in [0,T]} \|u(t)-\tilde{u}(t)\|_V &\le k^{\delta_1} \|u\|_{C^{\delta_1}([0,T];V)} + \max_{j=0,\ldots,N_k}  \|u(t_j)-u^{j}\|_V.
\end{align*}
Therefore, taking $L^p$-norms and using the error estimate of Theorem \ref{thm:convergenceRateWave} we find that
\begin{align*}
\bigg\|\sup_{t \in [0,T]} \|u(t)-\tilde{u}(t)\|_V\bigg\|_{p} &\le 2 C_{U_0,\textbf{f},\textbf{g},X_{\delta_1}} k^{\delta_1} + \bigg\|\max_{j=0,\ldots,N_k}  \|u(t_j)-u^{j}\|_V \bigg\|_{p}
\\ & \leq 2 C_{U_0,\textbf{f},\textbf{g},X_{\delta_1}} k^{\delta_1} + C_\ee\big(C_1 +C_2\sqrt{\log (\max\{T/k,p\})}\big)k^{\alpha}.
\end{align*}
The second estimate is obtained from Corollary \ref{cor:expEulerWave} in place of Theorem \ref{thm:convergenceRateWave} in the last step.
\end{proof}

\subsection{Application to the stochastic wave equation with trace class noise}
\label{subsec:traceclassnoisewave}
 As an example, we consider the classical stochastic wave equation on an open and bounded subset $\calO\subseteq \R^d$:
 \begin{align}
 \label{eq:waveEqnEx}
     \bigg\{\begin{split}
         \rmd \dot{u} &= (\Delta u + F(u)) \;\rmd t + G(u) \;\rmd W(t)\quad \text{on }[0,T],\\
         u(0)&=u_0,~\dot{u}(0)=v_0,
     \end{split}
 \end{align}
 with Dirichlet boundary conditions. In the current subsection, we consider trace class noise in $L^2$ for any $d \in \N$, and in Subsection \ref{subsec:exampleWavewhitenoise} space-time white noise in case $d=1$.

It is well-known that $\Lambda = -\Delta$ is a positive and self-adjoint operator on $L^2(\calO)$, which is invertible. Let $\{W(t)\}_{t \in [0,T]}$ be a $Q$-Wiener process with $Q\in \calL(L^2(\calO))$ so that $Q$ is positive and self-adjoint. Finite-dimensional noise is included, since $Q$ need not be strictly positive. Assume
\begin{equation}\label{eq:condQtraceClass}
Q^{1/2}\in \calL(L^2(\calO), L^\infty(\calO)).
\end{equation}
In particular, this implies $Q^{1/2}\in \calL_2(L^2(\calO), L^2(\calO))$ and that $Q$ is trace class (see \cite[Corollary 9.3.3]{AnalysisBanachSpacesII}).

We consider the stochastic wave equation \eqref{eq:waveEqnEx} on $V \ce L^2(\calO)$ and set $H \ce L^2(\calO)$. For the nonlinearity and the multiplicative noise, we choose Nemytskij operators
$F: V \to V$ and $G: V \to \calL_2(H,V) = \calL_2(L^2(\calO),L^2(\calO))$ determined by
\begin{equation}
\label{eq:defFGNemytskij}
    F(u)(\xi)=\phi(\xi,u(\xi)),\quad (G(u)(h))(\xi)=\psi(\xi,u(\xi))Q^{1/2}h(\xi),\quad \xi \in \calO.
\end{equation}
Here, the measurable functions $\phi,\psi: \calO \times \R \to \R$ are Lipschitz and of linear growth in the second coordinate, i.e., there is a constant $L \ge 0$ such that for all $u,u_1,u_2 \in \R$, $\xi \in \calO$ it holds that
\begin{equation}
\label{eq:assumptionshWave}
    |\phi(\xi,u)| + |\psi(\xi,u)| \le L(1+|u|),\quad |\phi(\xi,u_1)-\phi(\xi,u_2)| + |\psi(\xi,u_1)-\psi(\xi,u_2)| \le L|u_1-u_2|.
\end{equation}
It is clear that $F$ is Lipschitz from $V$ to $V$. To see that the same holds for $G$, note that by \eqref{eq:condQtraceClass}
\begin{align*}
|G(u) h(\xi)| = |\psi(\xi,u(\xi))| |Q^{1/2} h(\xi)|\leq C_{\psi,Q}(1+|u(\xi)|)\|h\|_{H},
\end{align*}
where $C_{\psi,Q}\ce L \|Q^{1/2}\|_{\calL(L^2(\calO), L^\infty(\calO))}$. Therefore, arguing as in \cite[Theorem 9.3.6 (3)$\Rightarrow$(4)]{AnalysisBanachSpacesII} by Riesz' theorem we can find $k_u:\calO\to H$ such that for a.e.\ $\xi\in \calO$ for all $h\in H$, $(k_u(\xi), h)_H = (G(u)h)(\xi)$, and $\|k_u(\xi)\|_H\leq C_{\psi,Q}(1+|u(\xi)|)$. Therefore, for an orthonormal basis $(h_n)_{n\geq 1}$ of $H$, we find that
\begin{align*}
\|G(u)\|_{\calL_2(H,V)}^2 &=
\sum_{n\geq 1}\|G(u)h_n\|^2_{V} =  \int_{\calO} \sum_{n\geq 1}|(k_u(\xi), h_n)|^2 \mathrm{d}\xi = \int_{\calO} \|k_u(\xi)\|_{H}^2 \mathrm{d}\xi \\ & \leq C_{\psi,Q}^2\|1+|u|\|_{V}^2\leq C_{\psi,Q}^2(|\calO|^{1/2} + \|u\|_{V})^2.
\end{align*}
with $|\calO|$ denoting the Lebesgue measure of the set $\calO$. Likewise, we obtain Lipschitz continuity of $G$.
In particular, $F$ and $G$ satisfy the required mapping properties of Assumption \ref{ass:FG_wave} for any $\delta \in (0,1]$.

The semigroup associated with \eqref{eq:waveEqnEx} is the wave semigroup $(S(t))_{t \ge 0}$.

As an immediate consequence of Theorem \ref{thm:convergenceRateWave} and Corollary \ref{cor:expEulerWave}, this yields the following convergence estimate generalising \cite[Cor.~4.2]{Wang15} to arbitrary contractive schemes and slightly more general $Q$-Wiener processes $W$.
\begin{theorem}[Wave equation with trace class noise in $L^2$]\label{thm:tracewave}
    Let $\calO \seq \R^d$, $d \in \N$, be a bounded and open set, $V \ce L^2(\calO)$, $X \ce V \times V_{-1}$, $p \in [2,\infty)$, and $0<\alpha\leq \delta\leq 1$. Suppose that $(u_0,v_0) \in L_{\calF_0}^p(\Omega;X_\delta)$. Let $F$ and $G$ be the Nemytskij operators as in \eqref{eq:defFGNemytskij} with $\phi$ and $\psi$ satisfying \eqref{eq:assumptionshWave}. Suppose the covariance operator $Q\in \calL(L^2(\calO))$ satisfies \eqref{eq:condQtraceClass}. Let $Y \ce X_\delta$ be as defined in \eqref{eq:defXbeta}. Let $(R_{k})_{k>0}$ be a time discretisation scheme which is contractive on both $X$ and $Y$. Suppose that $R$ approximates $S$ to order $\alpha$ on $Y$.
    Denote by $U$ the mild solution of \eqref{eq:StEvolEqnWave} with trace class noise and by $(U^j)_{j=0,\ldots,N_k}$ the temporal approximations as defined in \eqref{eq:defUjWave}.
    Then there exists a constant $C \ge 0$ depending on $(u_0,v_0,\phi,\psi,T,p,\alpha,\calO,d,V,\delta)$ such that for $N_k \ge 2$
    \begin{equation*}
        \left\| \max_{0 \le j \le N_k} \|U(t_j)-U^j\|_{X} \right\|_p \le C \big(1+\|Q^{1/2}\|_{\calL(L^2(\calO),L^\infty(\calO))}\big)\sqrt{\log(T/k)} k^\alpha.
    \end{equation*}
    In particular, the approximations $(U^j)_j$ converge at rate $1$ if $(u_0,v_0) \in L_{\calF_0}^p(\Omega;X_1)$ and the exponential Euler method $R=S$ is used. The logarithmic factor can be omitted in this case.
\end{theorem}
In case $\delta=1$, for the implicit Euler and the Crank--Nicolson method, we can take $\alpha = 1/2$ and $\alpha = 2/3$, respectively. This is due to convergence at rate $\alpha$ on $D((-A)^{2\alpha})$ and $D((-A)^{3\alpha/2})$, respectively. Using higher-order schemes, we can come as close to rate $1$ as we want. In Theorem \ref{thm:smoothwave} we show that for smoother noise $\alpha=1$ can be reached even for the implicit Euler method.

\subsection{Application to the stochastic wave equation with space-time white noise}
\label{subsec:exampleWavewhitenoise}
We use the same notation as in Subsection \ref{subsec:traceclassnoisewave}, but this time with $\calO=(0,1)$ and $Q=I$, so that \eqref{eq:waveEqnEx} is the classical wave equation with space-time white noise. The required mapping properties can be checked as in \cite[Cor.~4.3]{Wang15}. For convenience of the reader, we include the details. The functions $F$ and $G$ are defined via \eqref{eq:defFGNemytskij}, but this time we have to consider $G$ as a mapping $G: V \to \calL_2(H,V_{-1})$.

The eigenvalues of the negative Dirichlet Laplacian $\Lambda=-\Delta$ are $\lambda_i=\pi^2i^2$, $i \in \N$, with the corresponding orthonormal basis $\{e_i=\sqrt{2}\sin(i\pi \cdot)\,:\,i \in \N\}$ of $V$ consisting of eigenfunctions of $\Lambda$. Clearly,
\begin{equation*}
    \sup_{i \in \N} \sup_{\xi \in [0,1]} |e_i(\xi)| \le \sqrt{2},\quad\text{and}~\|\Lambda^{-\frac{\varepsilon+1}{4}}\|_{\calL(V)}^2 = \pi^{-(\ve+1)}\sum_{i=1}^\infty i^{-(\ve+1)} \ec c_\ve < \infty
\end{equation*}
then hold for every $\vareps>0$. Now let $\ve \in (0,1]$. Using the properties above, we conclude that
\begin{align*}
    \|\Lambda^{-\frac{\vareps+1}{4}}G(u)\|_{\calL_2(H,V)}^2 &= \sum_{i=1}^\infty \sum_{j=1}^\infty |\langle G(u)e_i,\Lambda^{-\frac{\vareps+1}{4}}e_j \rangle_V|^2 = \sum_{i=1}^\infty \sum_{j=1}^\infty \lambda_j^{-\frac{\vareps+1}{2}}\left|\int_\calO g(\xi,u(\xi))e_i(\xi)e_j(\xi)\;\rmd \xi \right|^2 \\
    &\le 2 \left( \sum_{j=1}^\infty \lambda_j^{-\frac{\vareps+1}{2}} \right)\|g(\cdot,u(\cdot))\|_V^2 \le 2L^2c_\vareps (|\calO|^{1/2}+\|u\|_V)^2.
\end{align*}
Hence, $G$ satisfies the linear growth condition of Assumption \ref{ass:FG_wave} with $\delta=\frac{1-\vareps}{2}$. Repeating the arguments for $\Lambda^{-1/2}[G(u_1)-G(u_2)]$ and using $c_1 = \pi^2/6$ results in
\begin{align*}
    \|\Lambda^{-1/2}[G(u_1)-G(u_2)]\|_{\calL_2(V)}^2 &\le 2 \left( \sum_{j=1}^\infty \frac{1}{\pi^2j^2}\right)\|g(\cdot,u_1(\cdot))-g(\cdot,u_2(\cdot))\|_V^2 \le \frac{L^2}{3}\|u_1-u_2\|_V^2.
\end{align*}
The nonlinearity $F$ was already considered in Subsection \ref{subsec:traceclassnoisewave}. In conclusion, we obtain the following generalisation of \cite[Cor.~4.3]{Wang15} to contractive time discretisation schemes.

\begin{theorem}[Wave equation with white noise]
     Let $\calO =(0,1)$, $V \ce L^2(\calO)$, $X \ce V \times V_{-1}$, $p \in [2,\infty)$, and $0<\alpha\leq \delta<1/2$.  Suppose that $(u_0,v_0) \in L_{\calF_0}^p(\Omega;X_\delta)$.
    Let $F$ and $G$ be Nemytskij operators as above with $\phi$ and $\psi$ satisfying \eqref{eq:assumptionshWave}. Suppose the covariance operator $Q = I$ on $L^2(\calO)$. Let $Y = X_{\delta}$. Let $(R_{k})_{k>0}$ be a time discretisation scheme which is contractive on $X$ and $Y$. Assume that $R$ approximates $S$ on $Y$ to order $\alpha$. Denote by $U$ the mild solution of \eqref{eq:StEvolEqnWave} with space-time white noise and by $(U^j)_{j=0,\ldots,N_k}$ the temporal approximations as defined in \eqref{eq:defUjWave}.
    Then there exists a constant $C \ge 0$ depending on $(u_0,v_0,\phi,\psi,T,p,\alpha,\calO,d,V,\delta)$ such that for $N_k \ge 2$
    \begin{equation*}
        \left\| \max_{0 \le j \le N_k} \|U(t_j)-U^j\|_{X} \right\|_p \le C \sqrt{\log(T/k)} k^{\alpha}.
    \end{equation*}
    In particular, the approximations $(U^j)_j$ converge at rate arbitrarily close to $\frac{1}{2}$ if $(u_0,v_0) \in L_{\calF_0}^p(\Omega;X_1)$ and the exponential Euler method $R=S$ is used. The logarithmic factor can be omitted in this case.
\end{theorem}
For the implicit Euler and the Crank--Nicolson method, we can take $\alpha = \delta/2$ and $\alpha = 2\delta/3$, respectively. Since we can choose $\delta$ arbitrarily close to $1/2$ this leads to rates which are almost $1/4$ and $1/3$, respectively.

\subsection{Application to the stochastic wave equation with smooth noise}\label{subsec:smoothnoisewave}

We have already seen that the exponential Euler method leads to convergence rates of any order $\alpha\in (0,1]$ depending on the given data. In this section, we show that this can also be attained for other schemes such as the implicit Euler and the Crank--Nicolson method under some smoothness conditions on the noise. To avoid problems with boundary conditions we only consider periodic boundary conditions. Consider
 \begin{align}
 \label{eq:waveEqnExTorus}
     \bigg\{\begin{split}
         \rmd \dot{u} &= ((\Delta-1) u + F(u)) \;\rmd t + G(u) \;\rmd W(t)\quad \text{on }[0,T],\\
         u(0)&=u_0,~\dot{u}(0)=v_0,
     \end{split}
 \end{align}
 with $\Lambda = 1-\Delta$ and periodic boundary conditions on the $d$-dimensional torus $\T^d = [0,1]^d$. For notational convenience we will write $H^{\beta} = H^{\beta}(\T^d) = V_{\beta}$. Note that $\|\Lambda^{-\beta}\|_{\calL(L^2)}\leq 1$ for all $\beta>0$. The additional $+1$ in the definition of $\Lambda$ is in order to ensure invertibility. Of course, $F$ can be suitably redefined so that this is without loss of generality.

Let $\delta\in (1, 2]$ and write $s=\delta-1$. Let
\begin{equation*}
    F(u)(\xi)=\phi(u(\xi)),\quad (G(u)(h))(\xi)=\psi(u(\xi))Q^{1/2}h(\xi),\quad \xi \in \T^d.
\end{equation*}
Here, the measurable functions $\phi,\psi: \R \to \R$ are Lipschitz with Lipschitz constants $L_{\phi}$ and $L_{\psi}$, respectively. The Lipschitz estimates for $F$ and $G$ follow as in Subsection \ref{subsec:traceclassnoisewave} since we will assume even more restrictive conditions on $Q$. The growth estimates for $F$ and $G$ as in Assumption \ref{ass:FG_wave} \ref{item:linGrowthWaveAss} are more complicated. In case $\delta=2$ the paraproduct constructions from \cite{TaylorBook} can be avoided, but we will consider the general case.

By the torus version of \cite[Prop.~2.4.1]{TaylorBook}  for $u\in V_{\delta}$, there is a constant $C_{s,\phi} \ge 0$ such that
\[\|F(u)\|_{V_{\delta-1}} = \|\phi(u)\|_{H^{\delta-1}}\leq C_{s,\phi} (\|u\|_{H^{\delta-1}}+1)\leq C_{s,\phi} (\|u\|_{H^{\delta}}+1) = C_{s,\phi} (\|u\|_{V_{\delta}}+1).\]

For $G$ the estimate is still more complicated. In order to estimate the Hilbert--Schmidt norm of $G(u)$, paraproduct estimates are required, as, for instance, in \eqref{eq:exampleProductEstimateBanach}. These paraproduct estimates involve Bessel potential spaces $H^{s,q}$, which, in general, are not Hilbert spaces. Consequently, an extension of Hilbert--Schmidt operators to Banach spaces is needed; the so-called \emph{$\gamma$-radonifying operators} \cite[Section~9.1]{AnalysisBanachSpacesII}. For a Banach space $E$, let $\gamma(H,E)$ denote the space of $\gamma$-radonifying operators. Let $(\gamma_n)_{n\geq 1}$ be an i.i.d.\ sequence of standard Gaussian random variables taking values in $\R$. Suppose that
$\Lambda^{\frac{\delta-1}{2}} Q^{1/2}:L^2\to L^\infty$. Then by \cite[Corollary 9.3.3]{AnalysisBanachSpacesII}, $Q^{1/2}\in \gamma(H,H^{\beta,q})$ for all $q\in [1, \infty)$ and all $\beta\leq \delta-1$, and
\begin{align}\label{eq:gammaQsmoothwave}
C_{q,\beta}\ce\|Q^{1/2}\|_{\gamma(H,H^{\beta,q})} \leq \|Q^{1/2}\|_{\gamma(H,H^{\delta-1,q})}\leq c_q \|\Lambda^{\frac{\delta-1}{2}} Q^{1/2}\|_{\calL(L^2,L^\infty)},
\end{align}
where $c_q = \|\gamma_1\|_{L^q(\Omega)}$.
Let $(h_n)_{n\geq1}$ be an orthonormal basis for $H$ and fix $N\geq 1$. Let $\eta_N \ce \sum_{n=1}^N \gamma_n Q^{1/2} h_n\in L^2(\Omega;V_{\delta-1})$. Then $\|\eta_N\|_{L^2(\Omega;V_{\beta})}\leq \|Q^{1/2}\|_{\gamma(H,H^{\beta,q})}$ for all $\beta\leq \delta-1$. It follows that
\begin{align*}
\sum_{n=1}^N \|G(u) h_n\|^2_{V_{\delta-1}} = \|\psi(u) \eta_N \|_{L^2(\Omega;V_{\delta-1})}^2.
\end{align*}
Next, we estimate $\|\psi(u) \eta_N \|_{V_{\delta-1}}$ pointwise in $\Omega$.
By the torus version of
\cite[Proposition 2.1.1]{TaylorBook} (see \cite[Proposition 4.1(1)]{AV21_SMR_torus}) and \cite[Prop.~2.4.1]{TaylorBook}, there is a constant $C_{\delta,d,1} \ge 0$ such that
\begin{align}
\label{eq:exampleProductEstimateBanach}
\|\psi(u) \eta_N\|_{V_{\delta-1}} &=  \|\psi(u) \eta_N\|_{H^{\delta-1}}  \leq \|\psi(u)\|_{L^{q_1}} \|\eta_N\|_{H^{\delta-1,q_2}} + \|\psi(u)\|_{H^{\delta-1,r_2}} \|\eta_N\|_{L^{r_1}}
\\ & \leq L_{\psi} (\|u\|_{L^{q_1}} +1) \|\eta_N\|_{H^{\delta-1,q_2}} + L_{\psi} C_{\delta,d,1} (\|u\|_{H^{\delta-1,r_2}}+1) \|\eta_N\|_{H^{\delta-1,r_1}}, \nonumber
\end{align}
where $\frac{1}{q_1} + \frac{1}{q_2}= \frac{1}{r_1} + \frac{1}{r_2} =\frac12$ and $q_1, r_1\in (2, \infty]$ and $q_2,r_2\in [2, \infty)$. Taking $r_1<\infty$ and using \eqref{eq:gammaQsmoothwave}, we find that
\begin{align*}
\|\psi(u) \eta_N\|_{L^2(\Omega;V_{\delta-1})}\leq L_{\psi} C_{q_2, \delta-1}(\|u\|_{L^{q_1}} +1) + L_{\psi} C_{\delta,d,1} C_{r_1, \delta-1} (\|u\|_{H^{\delta-1,r_2}}+1)
\end{align*}
for suitable constants $C_{q_2,\delta-1}, C_{r_1,\delta-1}\ge 0$. It remains to estimate $\|u\|_{L^{q_1}}$ and $\|u\|_{H^{\delta-1,r_2}}$ by $\|u\|_{H^{\delta}} = \|u\|_{V_{\delta}}$ using suitable Sobolev embeddings and choosing $q_1\in (2, \infty]$ and $r_2\in (2, \infty)$ suitably.
As soon as we have done that we can let $N\to\infty$ and conclude the required estimate
\begin{align*}
\|G(u)\|_{\calL_2(H,V_{\delta-1})}\leq K(1+\|u\|_{V_{\delta}}).
\end{align*}

To obtain $H^{\delta}\hra L^{q_1}$ we consider two cases. If $\delta\leq d/2$ (e.g.\ $d\in \{1, 2\}$) we can take $q_1<\infty$ arbitrary. If $\delta>d/2$, then we take $q_1 = \frac{2d}{d-2\delta}$, and thus $q_2 = \frac{d}{\delta}$.

To obtain $H^{\delta}\hra H^{\delta-1,r_2}$ we consider two cases. If $d\in \{1, 2\}$, then we can take $r_2\in (2, \infty)$ arbitrary. If $d\geq 3$, then we set $r_2 =  \frac{2d}{d-2}$, and thus $r_1 = d$.
\begin{theorem}[Wave equation with smooth noise]\label{thm:smoothwave}
     Let $V \ce L^2(\T^d)$, $X \ce V \times V_{-1}$, $p \in [2,\infty)$, and $0<\alpha\leq 1<\delta\leq 2$.  Suppose that $(u_0,v_0) \in L_{\calF_0}^p(\Omega;X_\delta)$.
    Let $F$ and $G$ be Nemytskij operators as above with Lipschitz functions $\phi$ and $\psi$. Suppose the covariance operator $Q$ on $L^2(\calO)$ satisfies $\Lambda^{\frac{\delta-1}{2}} Q^{1/2}\in \calL(L^2(\T^d), L^\infty(\T^d))$.
    Let $Y \ce X_{\delta}$ be as defined in \eqref{eq:defXbeta}. Let $(R_{k})_{k>0}$ be a time discretisation scheme which is contractive on both $X$ and $Y$. Assume that $R$ approximates $S$ to order $\alpha$ on $Y$. Denote by $U$ the mild solution of \eqref{eq:waveEqnExTorus} driven by a $Q$-Wiener process $W$ and by $(U^j)_{j=0,\ldots,N_k}$ the temporal approximations as defined in \eqref{eq:defUjWave}.
    Then there exists a constant $C \ge 0$ depending on $(u_0,v_0,\phi,\psi,T,p,\alpha,d,V,\delta)$ such that for $N_k \ge 2$
    \begin{equation*}
        \left\| \max_{0 \le j \le N_k} \|U(t_j)-U^j\|_{X} \right\|_p \le C \big(1+\|\Lambda^{(\delta-1)/2}Q^{1/2}\|_{\calL(L^2(\T^d),L^\infty(\T^d))}\big) \sqrt{\log(T/k)} k^{\alpha}.
    \end{equation*}
\end{theorem}
The above result is not useful for the exponential Euler method, since Theorem \ref{thm:tracewave} is better in that case. However, if we specialize to the implicit Euler and the Crank--Nicolson method, then we obtain rates $\alpha = \frac{\delta}{2}$ and $\alpha = \min\{\frac{2}{3} \delta,1\}$, respectively. In particular, this leads to convergence of order one if $\delta=2$ for many numerical schemes. Note that $\delta=2$ more or less corresponds to a noise $W$ which is in $H^{1,q}(\T^d)$ for all $q<\infty$.

\begin{remark}
    Theorem \ref{thm:smoothwave} gives an explanation for the numerical convergence rates obtained in \cite[Fig.~6.1, right figure]{Wang15}. There, trace class noise determined by $\psi(u)=u$ and $Q$ with eigenvalues $q_j =j^{-\beta}$, $j \in \N$, $\beta =1.1$ has been investigated. Denote by $(e_j)_{j \in \N}$ the orthonormal basis of $V$ and by $\lambda_j=Cj^2$ the eigenvalues of $\Lambda$ as in Subsection \ref{subsec:exampleWavewhitenoise} for some constant $C > 0$. We calculate that
    \begin{equation*}
        \Lambda^{\frac{\delta-1}{2}} Q^{\frac12} e_j = q_j^{\frac12} \Lambda^{\frac{\delta-1}{2}} e_j = j^{-\frac{\beta}{2}} \lambda_j^{\frac{\delta-1}{2}} e_j = C^{\frac{\delta-1}{2}} j^{\delta-1-\frac{\beta}{2}} e_j
    \end{equation*}
    for $j \in \N$. Thus, $\Lambda^{\frac{\delta-1}{2}}Q^{\frac12}$ maps $L^2$ into $L^\infty$ if $\delta \leq 1+\frac{\beta}{2}$. Setting $\delta \ce \min\{1+\frac{\beta}{2},2\}=1+\frac{1.1}{2}=1.55$, we derive convergence of rate $\frac{\delta}{2}=0.775$ for the implicit Euler method and $\min\{\frac{2}{3}\delta,1\}=1$ for the Crank--Nicolson method. Taking numerical errors into account, this corresponds exactly to the numerical convergence rates obtained in \cite[Fig.~6.1, right figure]{Wang15}.
\end{remark}

\section*{Data Availability Statement}

The data underlying this article are available at a public github repository, cf. \cite{schroedingerCode}.

\end{document}